\documentclass[11pt]{article}
\usepackage[english]{babel}										\usepackage[margin = 1.2in]{geometry}					
\usepackage{amssymb,cite}	

\usepackage{amsmath,amsfonts,amsthm, mathtools, bm} 
\usepackage{mathrsfs}
\usepackage[pdftex]{graphicx}	
\usepackage[title,titletoc,page]{appendix} 
\usepackage{listings} 
\usepackage{color}
\usepackage{comment}
\usepackage{cases}
\usepackage[ruled,vlined]{algorithm2e}
\usepackage{sectsty}
\usepackage[prependcaption,textsize=tiny,textwidth=3cm]{todonotes}
\usepackage{cases}
\allsectionsfont{\centering \normalfont\scshape}

\usepackage{stmaryrd}
\usepackage{fancyhdr}
\usepackage{float}
\usepackage[cal=cm]{mathalfa}
\usepackage{mathrsfs}
\usepackage[usenames,dvipsnames]{xcolor}

\usepackage{todonotes}

\usepackage{tikz}
\usetikzlibrary{shapes.geometric, shapes.misc, positioning, arrows, decorations.pathreplacing, angles, quotes,calc}
\usepackage{booktabs}
\usepackage{xfrac}
\usepackage{siunitx}
\usepackage{centernot}
\usepackage{comment}
\usepackage{chngcntr}
\usepackage{extarrows}
\usepackage{adjustbox}
\usepackage{tikz-cd} 
\usetikzlibrary{decorations.pathreplacing}
\usepackage{subcaption}

\makeatletter
\def\thmhead@plain#1#2#3{%
  \thmname{#1}\thmnumber{\@ifnotempty{#1}{ }\@upn{#2}}%
  \thmnote{ {\the\thm@notefont#3}}}
\let\thmhead\thmhead@plain
\makeatother

\usepackage[hidelinks]{hyperref}
\hypersetup{colorlinks = true,
linkcolor = MidnightBlue, anchorcolor =red,
citecolor = ForestGreen,
filecolor = red,urlcolor = red,
            pdfauthor=author}
\usepackage{cleveref}

\usepackage{titlesec}
\titleformat{\subsection}{\normalfont\normalsize\bfseries}{\thesubsection}{1em}{}
\titleformat{\section}{\normalfont\large\bfseries}{\thesection}{1em}{}

\newtheorem{theorem}{Theorem}[section]
\newtheorem{lemma}[theorem]{Lemma}
\newtheorem{proposition}[theorem]{Proposition}
\newtheorem{corollary}[theorem]{Corollary}

\numberwithin{equation}{section}

\theoremstyle{remark}
\newtheorem{remark}[theorem]{Remark}

\theoremstyle{definition}
\newtheorem{definition}[theorem]{Definition}
\newtheorem*{example*}{Example}

\newcommand{\mc}[1]{\mathcal{#1}}

\newcommand{\w}[1]{\widetilde{#1}}
\newcommand{\h}[1]{\hat{#1}}

\def\q {\quad}

\def \l{\langle}
\def \r{\rangle}
\def \ti{\tilde}

\def\bb{\begin{equation}
  \left\{\ 
   \begin{aligned} }
\def\ee{   \end{aligned}
  \right.
  \end{equation}}

\def\mm{ \left[
 \begin{matrix}}
\def\nn{\end{matrix} \right] } 

\def\p{\partial}
\def \dd{\mathrm{d}}

\def \w {\widetilde}
\def \h{\hat}

\def\d{\delta}

\def\s{\sigma}

\def \diag{\operatorname{diag}}

\def \sdiv {\sdiv_S}

\def \R{\mathbb{R}}

\def \C{\mathbb{C}}

\def \Z{\mathbb{Z}}

\def \e{\mathrm{e}}
\def \i{\mathrm{i}}

\def \ww{\omega}

\def \x{{\rm{\bf x}}}

\def\l{\left}
\def\r{\right}
\def\lf{\lfloor}
\def\rf{\rfloor}
\def\lc{\lceil}
\def\rc{\rceil}
\def\lb{\llbracket}
\def\rb{\rrbracket}

\def\al{\bm\alpha}
\def\be{\bm\beta}

\everymath{\displaystyle}

\begin{document}

\title{Frequency-dependent capacitance matrix formulation for Fabry-Pérot resonances. Part I: One-dimensional finite systems}


\author{Habib Ammari\footnote{
  ETH Z\"urich, Department of Mathematics, Rämistrasse 101, 8092 Z\"urich, Switzerland, (habib.ammari@math.ethz.ch).} \and Bowen Li\footnote{Department of Mathematics,  City University of Hong Kong, Kowloon Tong, Hong Kong SAR, (bowen.li@cityu.edu.hk).}
\and Ping Liu\footnote{
School of Mathematical Sciences, Zhejiang University, Hangzhou, 310027, China, (pingliu@zju.edu.cn, shaoyingjie323@zju.edu.cn).}  \thanks{
Institute of
Fundamental and Transdisciplinary Research, Zhejiang University, Hangzhou, 310027, China.}
\and Yingjie Shao\footnotemark[3]
}

\date{}
\maketitle
\begin{abstract}
We study scattering resonances of finite one-dimensional systems of high-contrast 
resonators beyond the subwavelength regime. Introducing a novel tridiagonal 
frequency-dependent capacitance matrix, we derive quantitative asymptotic 
expansions of the hybridized Fabry-P\'erot resonant frequencies in terms of the 
material contrast parameter. The leading-order shifts are governed by 
the eigenvalues of this matrix, while the corresponding eigenmodes are 
approximated, to leading order, by trigonometric functions on selected spacings 
between resonators. Our results extend the use of discrete approximations as a powerful tool for characterizing the resonant properties of a system of high-contrast resonators at arbitrarily
high frequencies.
\end{abstract} 

    

\section{Introduction}

The study of scattering resonances of high-contrast resonator systems is fundamental 
in the design of metamaterials. In the subwavelength regime, it was shown 
in~\cite{ammari.davies.ea2024Functional,cbms,feppon.cheng.ea2023Subwavelength} 
that a \emph{capacitance matrix formulation} based on first-principles analysis 
provides a natural starting point for both theoretical analysis and numerical 
simulation of wave localization and manipulation. This formulation yields a discrete 
approximation to the low-frequency part of the spectrum of the continuous Helmholtz 
model, valid in the high-contrast asymptotic limit, and applies to a wide variety 
of settings, including finite and infinite systems and Hermitian and 
non-Hermitian models, enabling the description of a rich array of exotic 
subwavelength phenomena; 
see~\cite{barandun2023,ammariMathematicalFoundationsNonHermitian2024,%
ammari.barandun.ea2025Subwavelength,ammari2020exceptional}.

In~\cite{pm1}, using a propagation matrix approach, the scattering resonances of a finite system of resonators beyond the subwavelength regime are characterized as zeros of an 
explicit trigonometric polynomial. Nevertheless, their asymptotic expansions in terms of the contrast parameter are derived using this characterization only in the subwavelength regime.


In~\cite{li2025high}, for a single resonator in three dimensions, the existence of a family of infinite resonances near the real axis is established, and first-order asymptotic expansions of these resonances in terms of the material contrast parameter are obtained. The study of such 
high-frequency modes is motivated by the significant potential of high-contrast 
resonator systems to realize high-frequency and widely tunable 
resonant devices~\cite{HF1,HF2}. 

In this paper, we consider a finite system of resonators in which the scattering resonant frequencies hybridize due to inter-resonator interactions. We derive a frequency-dependent capacitance matrix formulation that provides the first-order asymptotic expansions of the hybridized scattering resonances with respect to the contrast parameter, valid beyond the subwavelength regime. The capacitance matrix introduced here is tridiagonal, reflecting the nearest-neighbor interaction structure 
of one-dimensional systems, and depends on the frequency, as it is designed to capture resonances away from zero. In particular, at zero frequency, it reduces to the classical capacitance matrix of~\cite{feppon.cheng.ea2023Subwavelength}.


The present paper is the first in a series devoted to scattering resonances of high-contrast resonator systems beyond the subwavelength regime. In Part~II, we consider periodic infinite structures, combining the approach developed here with those of~\cite{junshan1,barandun2023} to generalize our results and study high-frequency wave localization. In Part~III, we generalize the frequency-dependent capacitance matrix 
formulation to three-dimensional systems with long-range inter-resonator interactions, extending the results of~\cite{li2025high} to systems of resonators 
and enabling efficient computation of their high-frequency hybridized resonances. Part~IV treats nonlinear high-contrast resonator systems and elucidates the effect 
of nonlinearities on their spectral properties. Part~V addresses time-modulated systems, 
extending the scattering and transmission results 
of~\cite{liora1,liora2,liora3} to show, in particular, that nonreciprocal wave propagation is achievable in space-time modulated media beyond the subwavelength regime. 

The remainder of this paper is organized as follows. In~\Cref{sec:2}, we introduce the model problem and state our main results. In~\Cref{sec:3}, we study the spectral properties of the frequency-dependent capacitance matrix, which is of independent interest. \Cref{sec:4} characterizes the resonant frequencies as zeros of an analytic function via the total transfer matrix. Building on this characterization, \Cref{sec:5} provides a 
rigorous derivation of the asymptotic expansions of the resonant frequencies with respect to the contrast parameter. Finally, \Cref{sec:6} derives approximations of the eigenmodes beyond the subwavelength regime, showing that they can be approximated by trigonometric functions whose amplitudes are determined by the eigenvectors of a certain matrix.

\section{Model setting and 
Main results} \label{sec:2}

In this section, we introduce the mathematical model and summarize the main results of this work. We first formulate the one-dimensional system of high-contrast resonators and the governing wave equation in Section \ref{subsec:model}. Section \ref{subsec:frequ_capmat} defines the frequency-dependent capacitance matrix, which acts as a fundamental analytic tool to characterize hybridized scattering resonances beyond the subwavelength regime. Lastly, Section \ref{subsec:mainresult} details our main results, providing the asymptotic expansions for both the resonant frequencies and their corresponding eigenmodes in the high-contrast limit.

\subsection{Model setting} \label{subsec:model}
We consider a one-dimensional chain of $N$ identical disjoint resonators $D_i\coloneqq (x_{2i-1},x_{2i})$, where $(x_i)_{1\leq i\leq 2N} \subset \R$ are the $2N$ boundaries satisfying $x_1<x_2<\cdots<x_{2N}$. We also denote by  $\ell_i = x_{2i} - x_{2i-1}$ the length of each of the resonators and by $s_i= x_{2i+1} -x_{2i}$ the spacing between the $i$\textsuperscript{th}  and $(i+1)$\textsuperscript{th}  resonators. The system is illustrated in \Cref{fig:setting}. 

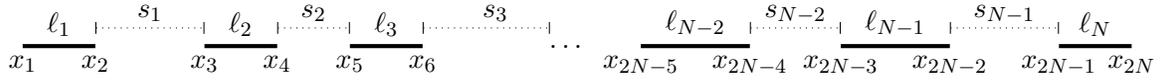
\begin{figure}[!htb]
    \centering
    \begin{adjustbox}{width=\textwidth}
    \begin{tikzpicture}
        \coordinate (x1l) at (1,0);
        \path (x1l) +(1,0) coordinate (x1r);
        \path (x1r) +(0.75,0.7) coordinate (s1);
        \path (x1r) +(1.5,0) coordinate (x2l);
        \path (x2l) +(1,0) coordinate (x2r);
        \path (x2r) +(0.5,0.7) coordinate (s2);
        \path (x2r) +(1,0) coordinate (x3l);
        \path (x3l) +(1,0) coordinate (x3r);
        \path (x3r) +(1,0.7) coordinate (s3);
        \path (x3r) +(2,0) coordinate (dots);
        \path (dots) +(1,0) coordinate (x6l);
        \path (x6l) +(1.5,0) coordinate (x6r);
        \path (x6r) +(1.25,0) coordinate (x7l);
        \path (x6r) +(0.625,0.7) coordinate (s6);
        \path (x7l) +(1.5,0) coordinate (x7r);
        \path (x7r) +(1.5,0) coordinate (x8l);
        \path (x7r) +(0.75,0.7) coordinate (s7);
        \path (x8l) +(1,0) coordinate (x8r);
        \draw[ultra thick] (x1l) -- (x1r);
        \node[anchor=north] at (x1l) {$x_1$};
        \node[anchor=north] at (x1r) {$x_2$};
        \node[anchor=south] at ($(x1l)!0.5!(x1r)$) {$\ell_1$};
        \draw[dotted,|-|] ($(x1r)+(0,0.25)$) -- ($(x2l)+(0,0.25)$);
        \draw[ultra thick] (x2l) -- (x2r);
        \node[anchor=north] at (x2l) {$x_3$};
        \node[anchor=north] at (x2r) {$x_4$};
        \node[anchor=south] at ($(x2l)!0.5!(x2r)$) {$\ell_2$};
        \draw[dotted,|-|] ($(x2r)+(0,0.25)$) -- ($(x3l)+(0,0.25)$);
        \draw[ultra thick] (x3l) -- (x3r);
        \node[anchor=north] at (x3l) {$x_5$};
        \node[anchor=north] at (x3r) {$x_6$};
        \node[anchor=south] at ($(x3l)!0.5!(x3r)$) {$\ell_3$};
        \draw[dotted,|-|] ($(x3r)+(0,0.25)$) -- ($(dots)+(-.25,0.25)$);
        \node at (dots) {\dots};
        \draw[ultra thick] (x6l) -- (x6r);
        \node[anchor=north] at (x6l) {$x_{2N-5}$};
        \node[anchor=north] at (x6r) {$x_{2N-4}$};
        \node[anchor=south] at ($(x6l)!0.5!(x6r)$) {$\ell_{N-2}$};
        \draw[dotted,|-|] ($(x6r)+(0,0.25)$) -- ($(x7l)+(0,0.25)$);
        \draw[ultra thick] (x7l) -- (x7r);
        \node[anchor=north] at (x7l) {$x_{2N-3}$};
        \node[anchor=north] at (x7r) {$x_{2N-2}$};
        \node[anchor=south] at ($(x7l)!0.5!(x7r)$) {$\ell_{N-1}$};
        \draw[dotted,|-|] ($(x7r)+(0,0.25)$) -- ($(x8l)+(0,0.25)$);
        \draw[ultra thick] (x8l) -- (x8r);
        \node[anchor=north] at (x8l) {$x_{2N-1}$};
        \node[anchor=north] at (x8r) {$x_{2N}$};
        \node[anchor=south] at ($(x8l)!0.5!(x8r)$) {$\ell_N$};
        \node[anchor=north] at (s1) {$s_1$};
        \node[anchor=north] at (s2) {$s_2$};
        \node[anchor=north] at (s3) {$s_3$};
        \node[anchor=north] at (s6) {$s_{N-2}$};
        \node[anchor=north] at (s7) {$s_{N-1}$};
    \end{tikzpicture}
    \end{adjustbox}
    \caption{A chain of $N$ resonators, with lengths
    $(\ell_i)_{1\leq i\leq N}$ and spacings $(s_{i})_{1\leq i\leq N-1}$.}
    \label{fig:setting}
\end{figure}
We denote by
\begin{align*}
   D\coloneqq \bigcup_{i=1}^N(x_{2i-1},x_{2i})
\end{align*}
the set of resonators. 
In this work, we consider the one-dimensional wave equation propagating in a heterogeneous medium with space-dependent material parameters:
\begin{align}
    \frac{\omega^{2}}{\kappa(x)}u(x) +\frac{\dd}{\dd x}\left( \frac{1}{\rho(x)}\frac{\dd}{\dd
    x}  u(x)\right) =0,\qquad x \in\R.
    \label{eq: gen Strum-Liouville}
\end{align}
We assume that the material parameters $\kappa(x)$ and $\rho(x)$ are piecewise constant in the interior and exterior of the resonators
\begin{align*}
    \kappa(x)=
    \begin{dcases}
        \kappa_b, & x\in D,\\
        \kappa, &  x\in\R\setminus D,
    \end{dcases}\quad\text{and}\quad
    \rho(x)=
    \begin{dcases}
        \rho_b, & x\in D,\\
        \rho, &  x\in\R\setminus D,
    \end{dcases}
\end{align*}
where the constants $\rho_b, \rho, \kappa, \kappa_b \in \R_{>0}$. We denote the wave speeds inside the set $D$ of resonators and inside the background medium $\R\setminus D$  by $v_b$ and $v$, the wave numbers respectively by $k_b$ and $k$, the contrast between the densities of the resonators and the background medium by $\delta$, and the ratio between wave speeds $v$ and $v_b$ by $r$
\begin{align} \label{eq:contrast}
    v_b:=\sqrt{\frac{\kappa_b}{\rho_b}}, \qquad v:=\sqrt{\frac{\kappa}{\rho}},\qquad
    k_b:=\frac{\omega}{v_b},\qquad k:=\frac{\omega}{v},\qquad
    \delta:=\frac{\rho_b}{\rho},\qquad r:=\frac{v}{v_b}.
\end{align}
\begin{remark}
It should be noted that while $\delta>0$ in (\ref{eq:contrast}) has a physical meaning, we will always extend it to the complex domain $\delta\in\C$. Moreover, we will always consider $r$ as a positive real parameter.
\end{remark}
For these step-wise defined material parameters, the wave problem determined by \eqref{eq: gen Strum-Liouville} reduces to the following system of coupled one-dimensional Helmholtz equations:
\begin{align}
    \label{equ: scattering problem}
    \begin{dcases}
        \frac{\dd{^2}}{\dd x^2}u(x)+ \frac{\omega^2}{v^2}u(x) = 0, & x\in \R \setminus D ,\\
        \frac{\dd{^2}}{\dd x^2}u(x)+ \frac{\omega^2}{v_b^2}u(x) = 0, & x\in D,\\
        u\vert_{+}(x_{{j}}) = u\vert_{-}(x_{{j}}) , &  1\leq j\leq 2N ,\\
        \left.\frac{\dd u}{\dd x}\right\vert_{+}(x_{{2j-1}}) = \delta\left.\frac{\dd u}{\dd x}\right\vert_{-}(x_{{2j-1}}), & 1\leq j\leq N ,\\
        \left.\frac{\dd u}{\dd x}\right\vert_{-}(x_{{2j}})=\delta\left.\frac{\dd u}{\dd x}\right\vert_{+}(x_{{2j}})   , & 1\leq j\leq N,\\
        \big(\frac{\dd}{\dd |x|} - \mathrm{i} k \big) u = 0 & \text{for } x \in (-\infty, x_{1}) \cup (x_{2N}, +\infty),  \\
    \end{dcases}
\end{align}
where for a function $w$ we denote by
\begin{align*}
    w\vert_{-}(x) \coloneqq \lim_{\substack{s\to 0\\ s>0}}w(x-s) \quad \mbox{and} \quad  w\vert_{+}(x) \coloneqq \lim_{\substack{s\to 0\\ s>0}}w(x+s),
\end{align*}
if the limits exist. We call $\omega$ a \textit{resonant frequency} (or \textit{resonance}) if (\ref{equ: scattering problem}) admits a non-trivial solution $u(x)$ and we call such a solution $u(x)$ an eigenmode. 

\subsection{Frequency-dependent capacitance matrix and its structure} \label{subsec:frequ_capmat}

In this section, we introduce the frequency-dependent capacitance matrix and elucidate its structural property. We let the vector $\bm t = (t_1, \dots, t_{2N-1})^\top$ be defined by 
\begin{equation} \label{equ: bm t def}
    \bm t:=(r\ell_1,\ s_{1},\ r\ell_{2},\ s_{2},\ \cdots,\ r\ell_{N-1},\ s_{N-1},\ r\ell_N)^\top\in\R^{2N-1}_{>0}.
    \end{equation}
We define a wave-number-dependent parameter 
\begin{align}\label{equ: t_j(k) def}
t_j(k) = 
\begin{cases}
    t_j, & \text{if } \pi \mid t_j k, \\
    \infty, & \text{otherwise},
\end{cases}
\end{align}
with the convention that $1/\infty = 0$. Based on this definition, we introduce the coupling coefficients
\begin{equation} \label{eq:coupling}
     \theta_j = \frac{1}{t_j t_{j+1}}, \quad \theta_j(k) = \frac{1}{t_j(k) t_{j+1}(k)}, \quad k \in \mathbb{R},\ 1 \leq j \leq 2N-2. 
\end{equation}
The frequency-dependent capacitance matrix is defined as the following tridiagonal matrix:
{\small
\begin{align}\label{eq:mathcalCk_def}
\mathcal{C}(k) :=
\begin{pmatrix}
    \theta_1(k) & -\theta_1(k) & & & \\
    -\theta_2(k) & \theta_2(k) + \theta_3(k) & -\theta_3(k) & & \\
    & -\theta_4(k) & \theta_4(k) + \theta_5(k) & -\theta_5(k) & \\
    & & \ddots & \ddots & \ddots \\
    & & & -\theta_{2N-4}(k) & \theta_{2N-4}(k) + \theta_{2N-3}(k) & -\theta_{2N-3}(k) \\
    & & & & -\theta_{2N-2}(k) & \theta_{2N-2}(k)
\end{pmatrix}.    
\end{align}
}

To further investigate the structure of $\mathcal{C}(k)$, we introduce several more notations.  Our subsequent analysis focuses on a fixed
\begin{equation} \label{def:E}
    k_0 \in E := \bigcup_{j=1}^{2N-1} \frac{\pi}{t_j} \mathbb{Z}.
\end{equation}
since $\mathcal{C}(k_0)$ is not a zero matrix only if $k_0\in E$. For $k_0\in E$, we define the set of resonant indices
\begin{equation} \label{def:Ik0}
    I = I(k_0) := \bigl\{ j : t_j k_0 \in \pi\mathbb{Z},\ 1 \leq j \leq 2N-1 \bigr\},
    \qquad n = n(k_0) := \# I.
\end{equation}
So $j \in I$ if and only if there exists $m_j \in \mathbb{Z}$ such that 
\begin{equation}\label{equ:defmj}
t_j k_0 = m_j \pi.
\end{equation}
For $a, b \in \mathbb{Z}$, we denote by 
$\llbracket a, b \rrbracket := [a, b] \cap \mathbb{Z}$ the corresponding integer 
interval. We partition $I$ into $p$ maximal disjoint integer intervals
\begin{equation}\label{equ:partition}
    I = \bigcup_{j=1}^p \mathcal{I}_j, \qquad \mathcal{I}_j = \llbracket a_j, b_j \rrbracket,
\end{equation}
where $b_j \prec a_{j+1}$ for $j = 1, \ldots, p-1$, which means $a_{j+1} - b_j > 1$.
We write, for $1 \leq j \leq p$, 
\begin{equation} \label{def:n_j}
 n_j := \#\mathcal{I}_j = b_j - a_j + 1, 
\end{equation}
so that
$n = \sum_{j=1}^p n_j$.

\begin{example*}
To illustrate (\ref{def:E})–(\ref{def:n_j}), we consider a concrete configuration with $N=8$ (so $2N-1=15$ segments) and take
\[
\bm t=(1,\,2,\,1.5,\,2.5,\,2,\,2,\,3,\,1,\,0.5,\,2,\,1,\,1.5,\,1,\,1,\,1)^{\top},\qquad k_0=\pi.
\]
Since $k_0=\pi$, an interval index $j$ belongs to the resonant set $I$ precisely when $t_j$ is an integer. Hence
\[
I = \{1,2,5,6,7,8,10,11,13,14,15\},\qquad n = \#I = 11.
\]
Figure \ref{fig:example_intervals} illustrates the partition (\ref{equ:partition}) of $I$.

\begin{figure}[!htb]
\centering
\begin{tikzpicture}[scale=0.6, every node/.style={font=\small}]
\def\N{15}
\def\resonant{1,2,5,6,7,8,10,11,13,14,15}
\foreach \i in {1,...,\N} { \fill[blue!30] (\i,0) rectangle (\i+1,1); }
\foreach \i in \resonant { \fill[red!30] (\i,0) rectangle (\i+1,1); }
\foreach \i in {1,...,\N} {
\draw[gray!50] (\i,0) rectangle (\i+1,1);
\node at (\i+0.5,0.5) {\i};
}
\draw[decorate,decoration={brace,amplitude=5pt,mirror}] (1,0) -- (3,0) node[midway,below=6pt] {$\mathcal{I}_1$ ($n_1=2$)};
\draw[decorate,decoration={brace,amplitude=5pt,mirror}] (5,0) -- (9,0) node[midway,below=6pt] {$\mathcal{I}_2$ ($n_2=4$)};
\draw[decorate,decoration={brace,amplitude=5pt,mirror}] (10,0) -- (12,0) node[midway,below=6pt] {$\mathcal{I}_3$ ($n_3=2$)};
\draw[decorate,decoration={brace,amplitude=5pt,mirror}] (13,0) -- (16,0) node[midway,below=6pt] {$\mathcal{I}_4$ ($n_4=3$)};
\end{tikzpicture}
\caption{Resonant intervals (red) and non‑resonant intervals (blue) for the example.}
\label{fig:example_intervals}
\end{figure}
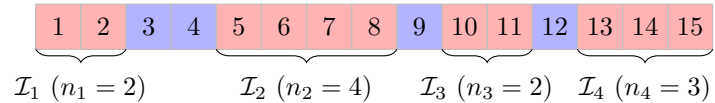
\end{example*}

Note that while the frequency-dependent capacitance matrix $\mathcal{C}(k_0)$ defined 
in~\eqref{eq:mathcalCk_def} does not itself have an exact block-diagonal structure, 
its symmetrized version $\mathcal{C}^{\mathrm{sym}}(k_0)$ does (see 
Figures~\ref{fig:interval_matrix} and~\ref{fig:compare}). This structural property 
underlies the following spectral result, whose proof is given in 
Section~\ref{sec: generalized capcitance matrix}. 

\begin{theorem}\label{thm: C_j's}
Let $C_j$, $1 \leq j \leq p$, be the principal submatrices of $\mathcal{C}(k_0)$ 
defined in Section~\ref{sec: generalized capcitance matrix}. Then the nonzero 
eigenvalues of $\mathcal{C}(k_0)$ coincide with the union of the nonzero 
eigenvalues of $C_1, \ldots, C_p$, counted with multiplicity. 
\end{theorem}

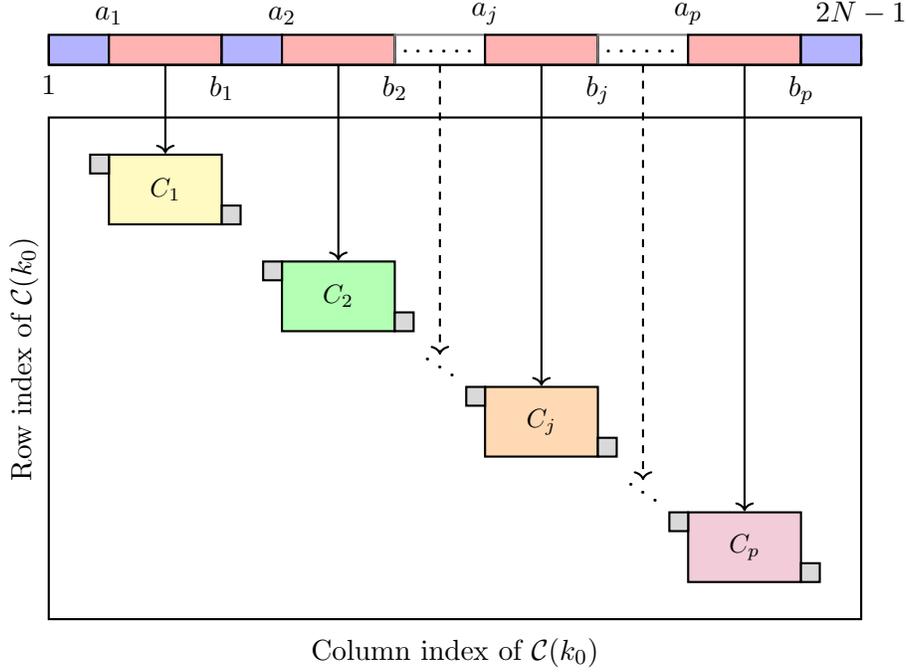
\begin{figure}[!htb]
\centering
\begin{tikzpicture}[
    thick,
    block/.style={draw, thick, align=center, font=\small},
    arrow/.style={thick, ->},
    ddot/.style={draw=none, font=\large}
]

\def\N{5}                         
\def\leftEnd{1}
\def\rightEnd{2N-1}               

\coordinate (start) at (0,0);
\coordinate (a1)   at (0.8,0);    
\coordinate (b1)   at (2.3,0);    
\coordinate (a2)   at (3.1,0);    
\coordinate (b2)   at (4.6,0);    
\coordinate (aj)   at (5.8,0);    
\coordinate (bj)   at (7.3,0);    
\coordinate (ap)   at (8.5,0);    
\coordinate (bp)   at (10,0);     
\coordinate (end)  at (10.8,0);   

\def\barHeight{0.4}

\newcounter{pt}
\setcounter{pt}{1}

\fill[blue!30] (start) rectangle (a1 |- {0,\barHeight});
\draw (start) rectangle (a1 |- {0,\barHeight});   
\draw (start) -- (start |- 0,\barHeight);
\draw (a1) -- (a1 |- 0,\barHeight);
\ifodd\value{pt}
    \node[below] at (start) {$\leftEnd$};
\else
    \node[above] at (start |- 0,\barHeight) {$\leftEnd$};
\fi
\stepcounter{pt}
\draw[dashed] (a1) -- (a1 |- 0,\barHeight);
\ifodd\value{pt}
    \node[below] at (a1) {$a_1$};
\else
    \node[above] at (a1 |- 0,\barHeight) {$a_1$};
\fi

\fill[red!30] (a1) rectangle (b1 |- {0,\barHeight});
\draw (a1) rectangle (b1 |- {0,\barHeight});   
\draw (b1) -- (b1 |- 0,\barHeight);
\stepcounter{pt}
\draw[dashed] (b1) -- (b1 |- 0,\barHeight);
\ifodd\value{pt}
    \node[below] at (b1) {$b_1$};
\else
    \node[above] at (b1 |- 0,\barHeight) {$b_1$};
\fi

\fill[blue!30] (b1) rectangle (a2 |- {0,\barHeight});
\draw (b1) rectangle (a2 |- {0,\barHeight});   
\draw (a2) -- (a2 |- 0,\barHeight);
\stepcounter{pt}
\draw[dashed] (a2) -- (a2 |- 0,\barHeight);
\ifodd\value{pt}
    \node[below] at (a2) {$a_2$};
\else
    \node[above] at (a2 |- 0,\barHeight) {$a_2$};
\fi

\fill[red!30] (a2) rectangle (b2 |- {0,\barHeight});
\draw (a2) rectangle (b2 |- {0,\barHeight});   
\draw (b2) -- (b2 |- 0,\barHeight);
\stepcounter{pt}
\draw[dashed] (b2) -- (b2 |- 0,\barHeight);
\ifodd\value{pt}
    \node[below] at (b2) {$b_2$};
\else
    \node[above] at (b2 |- 0,\barHeight) {$b_2$};
\fi

\draw[gray, thick] (b2) rectangle (aj |- {0,\barHeight});
\node at ($(b2)!0.5!(aj)+(0,0.17)$) {$\cdots\cdots$};
\draw (aj) -- (aj |- 0,\barHeight);
\stepcounter{pt}
\draw[dashed] (aj) -- (aj |- 0,\barHeight);
\ifodd\value{pt}
    \node[below] at (aj) {$a_j$};
\else
    \node[above] at (aj |- 0,\barHeight) {$a_j$};
\fi

\fill[red!30] (aj) rectangle (bj |- {0,\barHeight});
\draw (aj) rectangle (bj |- {0,\barHeight});   
\draw (bj) -- (bj |- 0,\barHeight);
\stepcounter{pt}
\draw[dashed] (bj) -- (bj |- 0,\barHeight);
\ifodd\value{pt}
    \node[below] at (bj) {$b_j$};
\else
    \node[above] at (bj |- 0,\barHeight) {$b_j$};
\fi

\draw[gray, thick] (bj) rectangle (ap |- {0,\barHeight});
\node at ($(bj)!0.5!(ap)+(0,0.17)$) {$\cdots\cdots$};
\draw (ap) -- (ap |- 0,\barHeight);
\stepcounter{pt}
\draw[dashed] (ap) -- (ap |- 0,\barHeight);
\ifodd\value{pt}
    \node[below] at (ap) {$a_p$};
\else
    \node[above] at (ap |- 0,\barHeight) {$a_p$};
\fi

\fill[red!30] (ap) rectangle (bp |- {0,\barHeight});
\draw (ap) rectangle (bp |- {0,\barHeight});   
\draw (bp) -- (bp |- 0,\barHeight);
\stepcounter{pt}
\draw[dashed] (bp) -- (bp |- 0,\barHeight);
\ifodd\value{pt}
    \node[below] at (bp) {$b_p$};
\else
    \node[above] at (bp |- 0,\barHeight) {$b_p$};
\fi

\fill[blue!30] (bp) rectangle (end |- {0,\barHeight});
\draw (bp) rectangle (end |- {0,\barHeight});   
\draw (end) -- (end |- 0,\barHeight);
\stepcounter{pt}
\draw[dashed] (end) -- (end |- 0,\barHeight);
\ifodd\value{pt}
    \node[below] at (end) {$\rightEnd$};
\else
    \node[above] at (end |- 0,\barHeight) {$\rightEnd$};
\fi

\draw[thick] (start) -- (end);

\def\yStart{-2.5}        
\def\stepY{-0.9}         

\path let \p1 = (a1), \p2 = (b1) in
    \pgfextra{\pgfmathsetmacro{\len}{\x2-\x1}}
    \pgfextra{\pgfmathsetmacro{\h}{\len * 0.618}}
    node[block, minimum width=\len pt, minimum height=\h pt, inner sep=0, fill=yellow!30] (C1) at ($(a1)!0.5!(b1) + (0,-1.6579)$) {$C_1$};

\path let \p1 = (a2), \p2 = (b2) in
    \pgfextra{\pgfmathsetmacro{\len}{\x2-\x1}}
    \pgfextra{\pgfmathsetmacro{\h}{\len * 0.618}}
    node[block, minimum width=\len pt, minimum height=\h pt, inner sep=0, fill=green!30] (C2) at ($(a2)!0.5!(b2) + (0,-3.0793)$) {$C_2$};

\path let \p1 = (aj), \p2 = (bj) in
    \pgfextra{\pgfmathsetmacro{\len}{\x2-\x1}}
    \pgfextra{\pgfmathsetmacro{\h}{\len * 0.618}}
    node[block, minimum width=\len pt, minimum height=\h pt, inner sep=0, fill=orange!30] (Cj) at ($(aj)!0.5!(bj) + (0,-4.7479)$) {$C_j$};

\path let \p1 = (ap), \p2 = (bp) in
    \pgfextra{\pgfmathsetmacro{\len}{\x2-\x1}}
    \pgfextra{\pgfmathsetmacro{\h}{\len * 0.618}}
    node[block, minimum width=\len pt, minimum height=\h pt, inner sep=0, fill=purple!20] (Cp) at ($(ap)!0.5!(bp) + (0,-6.4165)$) {$C_p$};

\node[ddot] (ddots1) at ($(C2)!0.5!(Cj)$) {$\ddots$};
\node[ddot] (ddots2) at ($(Cj)!0.5!(Cp)$) {$\ddots$};

\draw[arrow] ($(a1)!0.5!(b1)$) -- (C1);
\draw[arrow] ($(a2)!0.5!(b2)$) -- (C2);
\draw[arrow] ($(aj)!0.5!(bj)$) -- (Cj);
\draw[arrow] ($(ap)!0.5!(bp)$) -- (Cp);

\draw[dashed, ->, shorten >= -10] ($(b2)!0.5!(aj)$) -- (ddots1) ;
\draw[dashed, ->, shorten >= -10] ($(bj)!0.5!(ap)$) -- (ddots2);

\draw[thick] (0,-7.3744) rectangle (10.8,-0.7);

\node[below, yshift=-8pt] at (5.4,-7.2) {Column index of $\mathcal{C}(k_0)$};
\node[rotate=90, left, xshift=-10pt] at (-0.33,-1.8) {Row index of $\mathcal{C}(k_0)$};

\def\len{0.25}

\fill[gray!30] (0.8-\len,-1.1944-\len) rectangle (0.8,-1.1944); 
\draw (0.8-\len,-1.1944-\len) rectangle (0.8,-1.1944); 

\fill[gray!30] (2.3,-2.1214) rectangle (2.3+\len,-2.1214+\len); 
\draw (2.3,-2.1214) rectangle (2.3+\len,-2.1214+\len); 

\fill[gray!30] (3.1-\len,-2.6158-\len) rectangle (3.1,-2.6158); 
\draw (3.1-\len,-2.6158-\len) rectangle (3.1,-2.6158); 

\fill[gray!30] (4.6,-3.5428) rectangle (4.6+\len,-3.5428+\len); 
\draw (4.6,-3.5428) rectangle (4.6+\len,-3.5428+\len); 

\fill[gray!30] (5.8-\len,-4.2844-\len) rectangle (5.8,-4.2844); 
\draw (5.8-\len,-4.2844-\len) rectangle (5.8,-4.2844); 

\fill[gray!30] (7.3,-5.2114) rectangle (7.3+\len,-5.2114+\len); 
\draw (7.3,-5.2114) rectangle (7.3+\len,-5.2114+\len); 

\fill[gray!30] (8.5-\len,-5.953-\len) rectangle (8.5,-5.953); 
\draw (8.5-\len,-5.953-\len) rectangle (8.5,-5.953);

\fill[gray!30] (10,-6.88) rectangle (10+\len,-6.88+\len); 
\draw (10,-6.88) rectangle (10+\len,-6.88+\len); 

\end{tikzpicture}
\caption{Illustration of the correspondence between the integer intervals 
$\mathcal{I}_j = \llbracket a_j, b_j \rrbracket$ and the principal submatrices 
$C_j$ of $\mathcal{C}(k_0)$ defined in Section~\ref{sec: generalized capcitance 
matrix}. The large box represents $\mathcal{C}(k_0)$. Colored blocks indicate the 
positions of the submatrices $C_j$; adjacent gray squares ($1 \times 1$) denote 
possibly nonzero entries outside the blocks, while all remaining entries are zero. 
In the chain diagram above, blue segments are non-resonant and red segments are 
resonant.}
\label{fig:interval_matrix}
\end{figure}
\subsection{Main results}\label{subsec:mainresult}
The following results are our main findings in this paper. An illustration of Theorem \ref{thm:mainresult1} can be found in Figure \ref{fig:all} and Figure \ref{fig:convergence}.


\begin{theorem} \label{thm:mainresult1}
   Let $\d\in\mathbb C$. Assume that the frequency-dependent capacitance matrix $\mathcal{C}(k_0)$ has $m$ nonzero eigenvalues, then  $m\leq\lf n/2\rf$. For cases when $\d \to 0$ , the scattering problem \eqref{equ: scattering problem} has exactly $n$ (nonzero) resonant frequencies near $k_0 v$ for $k_0 \in E$: 
    \begin{itemize}
        \item[$\bullet$] The first $2m$ eigenfrequencies are analytic in $\delta^{1/2}$ and have the following asymptotic expansion:
        \begin{equation*}
        \omega^\pm_j(\delta)=k_0v\pm v\sqrt{\frac{\lambda_j}{r}\delta}+\mc{O}(\delta),\quad j=1,2,\cdots,m,
        \end{equation*}
        where $\lambda_1,\cdots,\lambda_m$ are the $m$ nonzero eigenvalues of $\mathcal{C}(k_0)$. 
        
        Moreover, let $\lambda$ be a nonzero eigenvalue of $\mathcal{C}(k_0)$ with multiplicity $\mathfrak{m}$. Then $\lambda$ is an eigenvalue of $\mathfrak{m}$ submatrices  of $\mathcal{C}(k_0)$, denoted by $C_{j_1},\cdots,C_{j_\mathfrak{m}}$. 
    The $2\mathfrak{m}$ branches satisfying 
    $\omega^\pm(\delta) = k_0 v \pm v\sqrt{\lambda\delta/r} + \mathcal{O}(\delta)$ admit the higher-order expansion:
\begin{align}\label{equ: omega^i,pm expand}
\omega^{i,\pm}(\delta)=k_0v\pm v\sqrt{\frac{\lambda}{r}\delta}+\frac{v}{2r}\frac{\al_{j_i}^\top B_{j_i} \al_{j_i}}{\al_{j_i}^\top \mathcal{L}_{j_i}\al_{j_i}}\delta+\mc{O}(\delta^{3/2}),\quad i=1,2,\cdots,\mathfrak{m},
\end{align}
where $(\lambda,\al_{j_i})$ is an eigenpair of $C_{j_i}$,
\[
B_j=\diag\{\cot(t_{a_{j-1}}k_0)\chi_{a_j},0,\cdots,0,\cot(t_{b_{j+1}}k_0)\chi_{b_j}\},
\]
with the convention that $\cot(t_jk_0):=-\i$ when $j=0$ or $j=2N$, and $\mathcal{L}_j$ is defined in (\ref{equ: C_j decomposition}). Here, $a_j,b_j$ are the endpoints of $\mathcal{I}_j$ in \eqref{equ:partition}, and
\[
\chi_j=\begin{cases}
    1, & j\text{ odd},\\
    \dfrac{1}{\lambda t_j^2}, & j\text{ even}.
\end{cases}
\]
        
        \item[$\bullet$] The remaining $n-2m$ eigenfrequencies have the following asymptotic expansion:
       \begin{equation}\label{eq:eigenfrequencydelta1}
\omega_j(\delta)=k_0v+\mc{O}(\delta),\qquad j=1,2,\cdots,n-2m.
        \end{equation}
\end{itemize}
\end{theorem}
\begin{figure}[!htb]
\centering

\begin{subfigure}[b]{0.24\textwidth}
    \includegraphics[height=0.18\textheight]{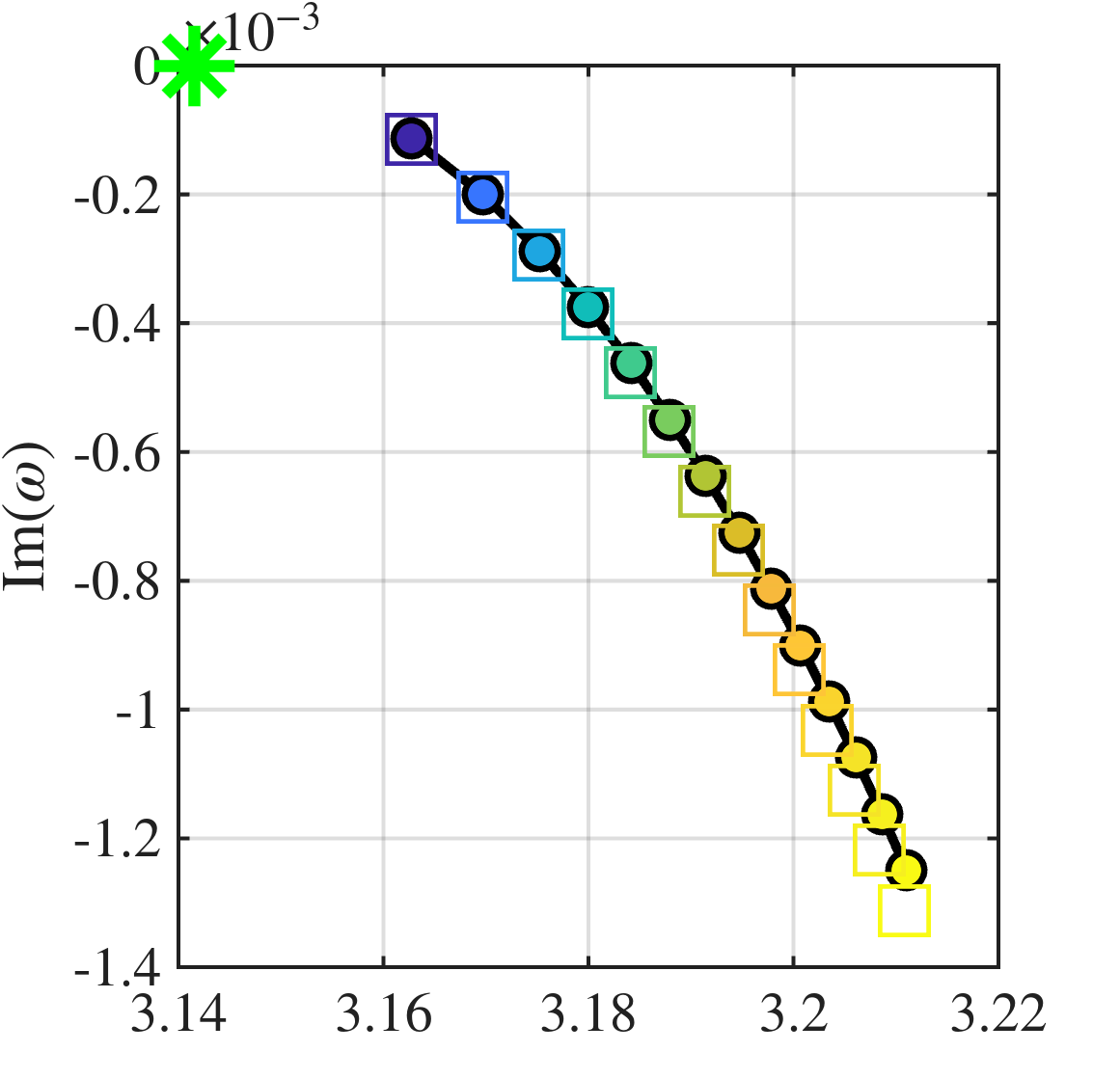}
    \caption{$c_1=1,c_2=-\textstyle\frac{1+\i}{4}.$}   
    \label{fig:sub1}
\end{subfigure}
\begin{subfigure}[b]{0.24\textwidth}
    \includegraphics[height=0.18\textheight]{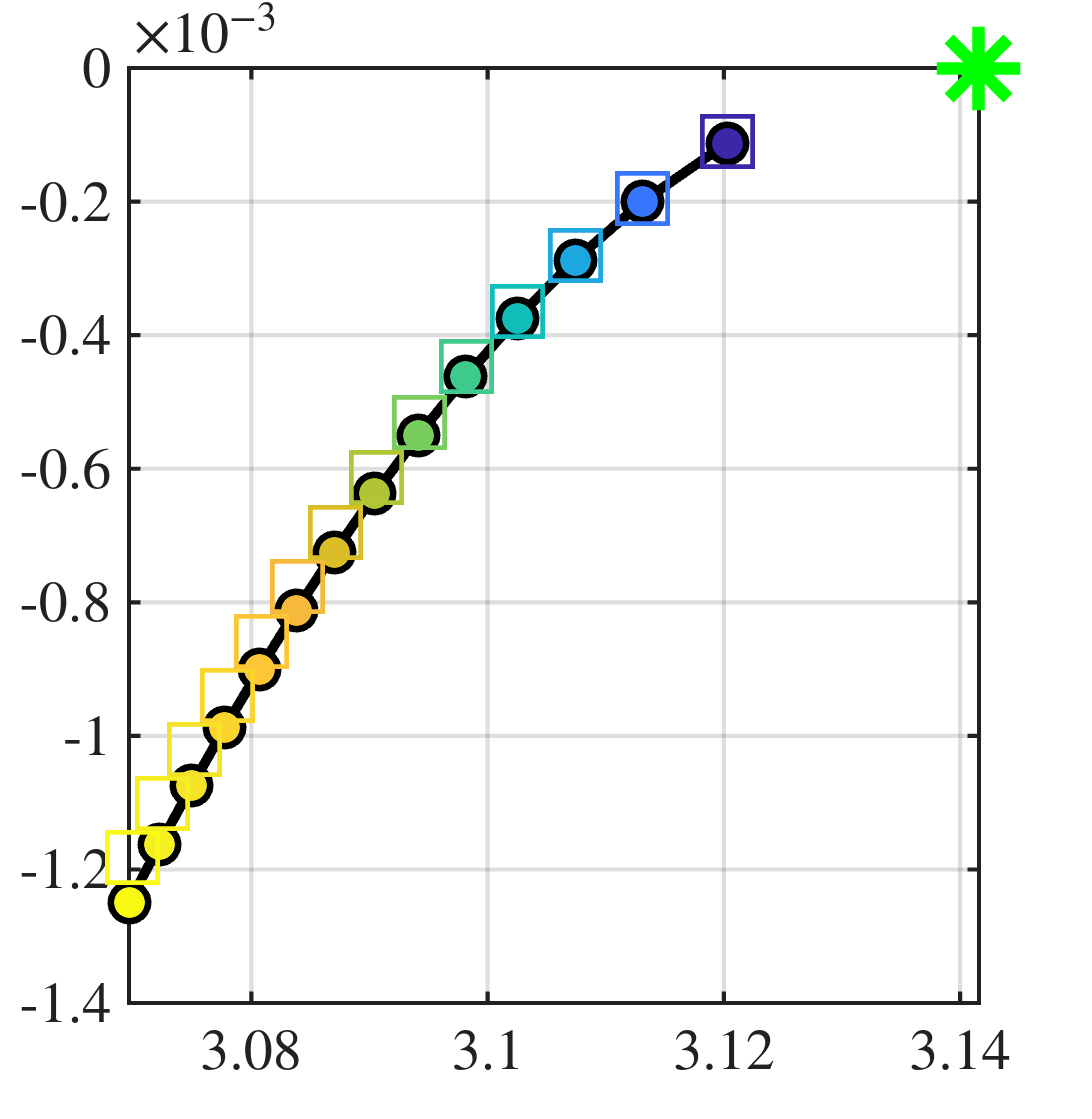}
    \caption{$c_1=-1,c_2=-\textstyle\frac{1+\i}{4}.$}
    \label{fig:sub2}
\end{subfigure}
\begin{subfigure}[b]{0.24\textwidth}
    \includegraphics[height=0.18\textheight]{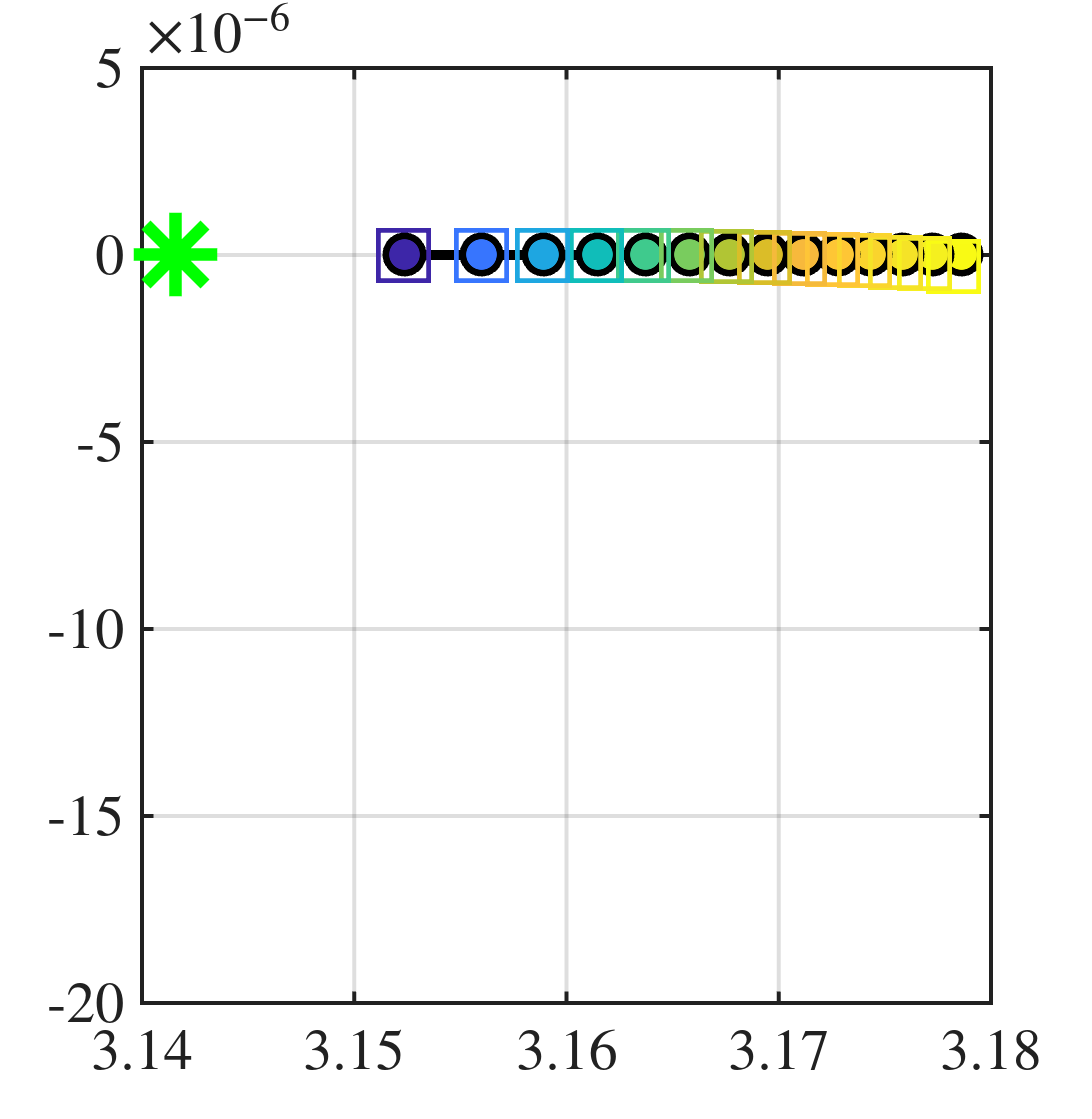}
    \caption{$c_1=1/2,c_2=1/3.$}
    \label{fig:sub3}
\end{subfigure}
\begin{subfigure}[b]{0.24\textwidth}
    \includegraphics[height=0.18\textheight]{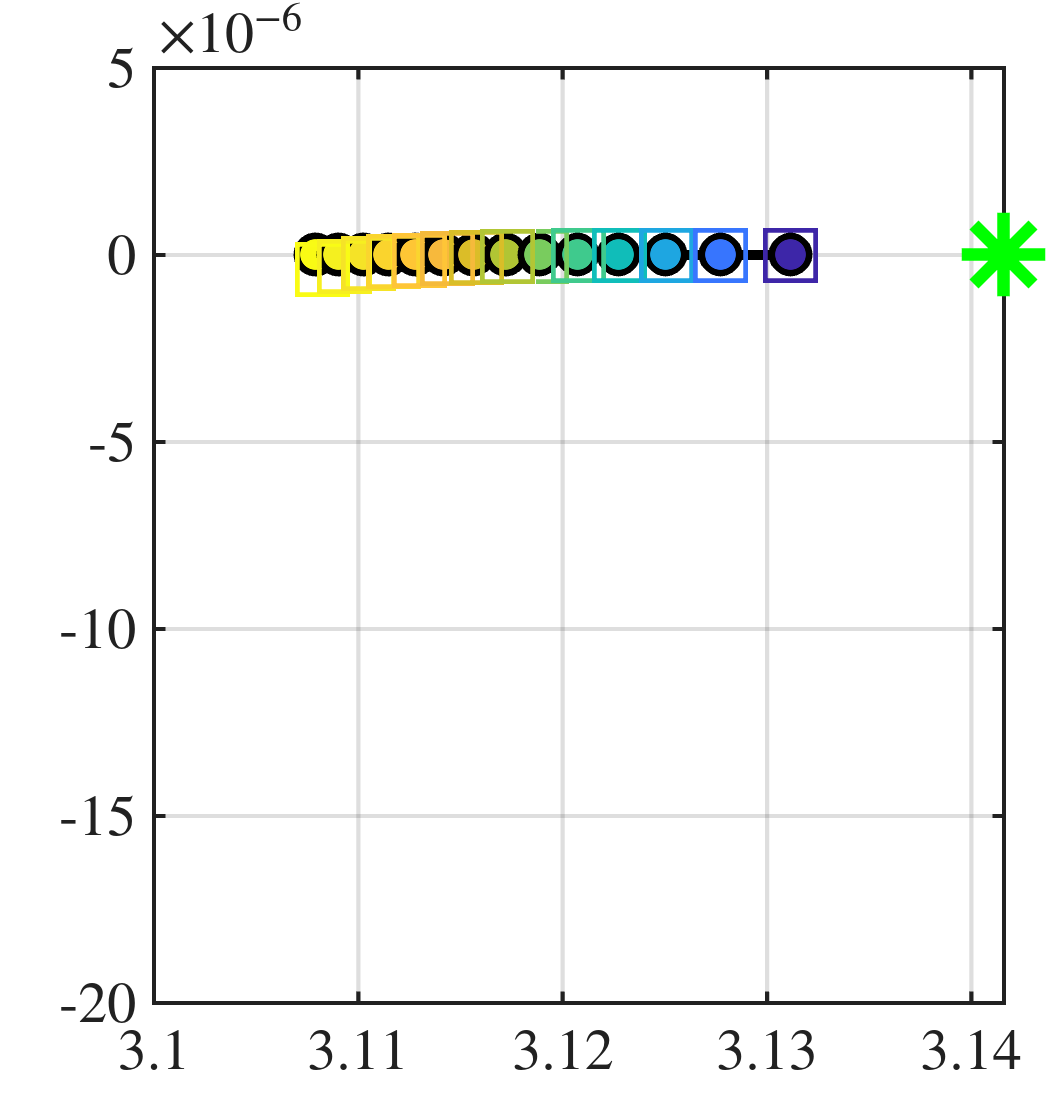}
    \caption{$c_1=-1/2,c_2=1/3.$}
    \label{fig:sub4}
\end{subfigure}

\begin{subfigure}[b]{0.24\textwidth}
    \raisebox{0pt}{\includegraphics[height=0.18\textheight]{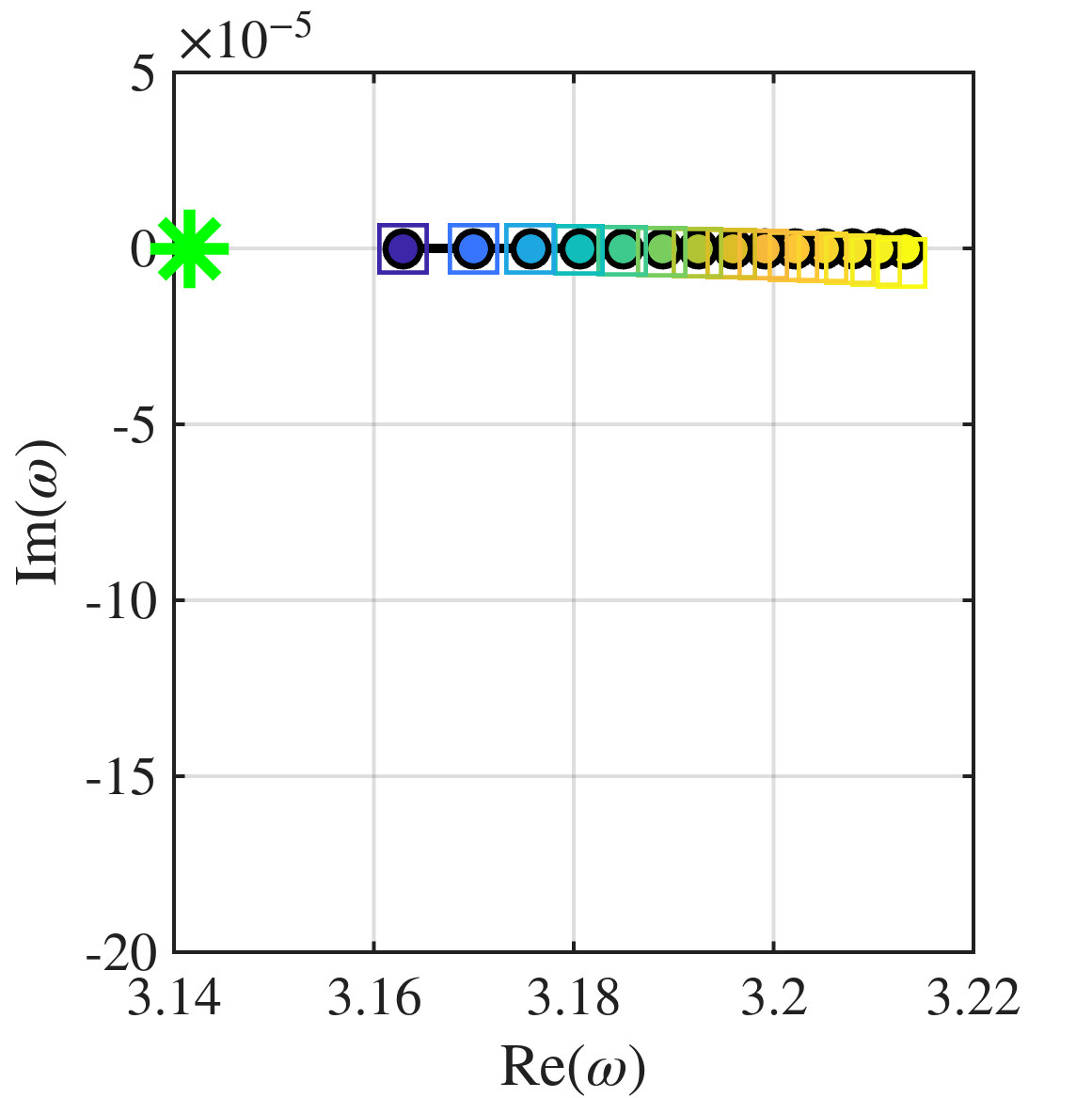}}
    \caption{$c_1=1,c_2=1/6.$}
    \label{fig:sub5}
\end{subfigure}
\begin{subfigure}[b]{0.24\textwidth}
    \raisebox{0pt}{\includegraphics[height=0.18\textheight]{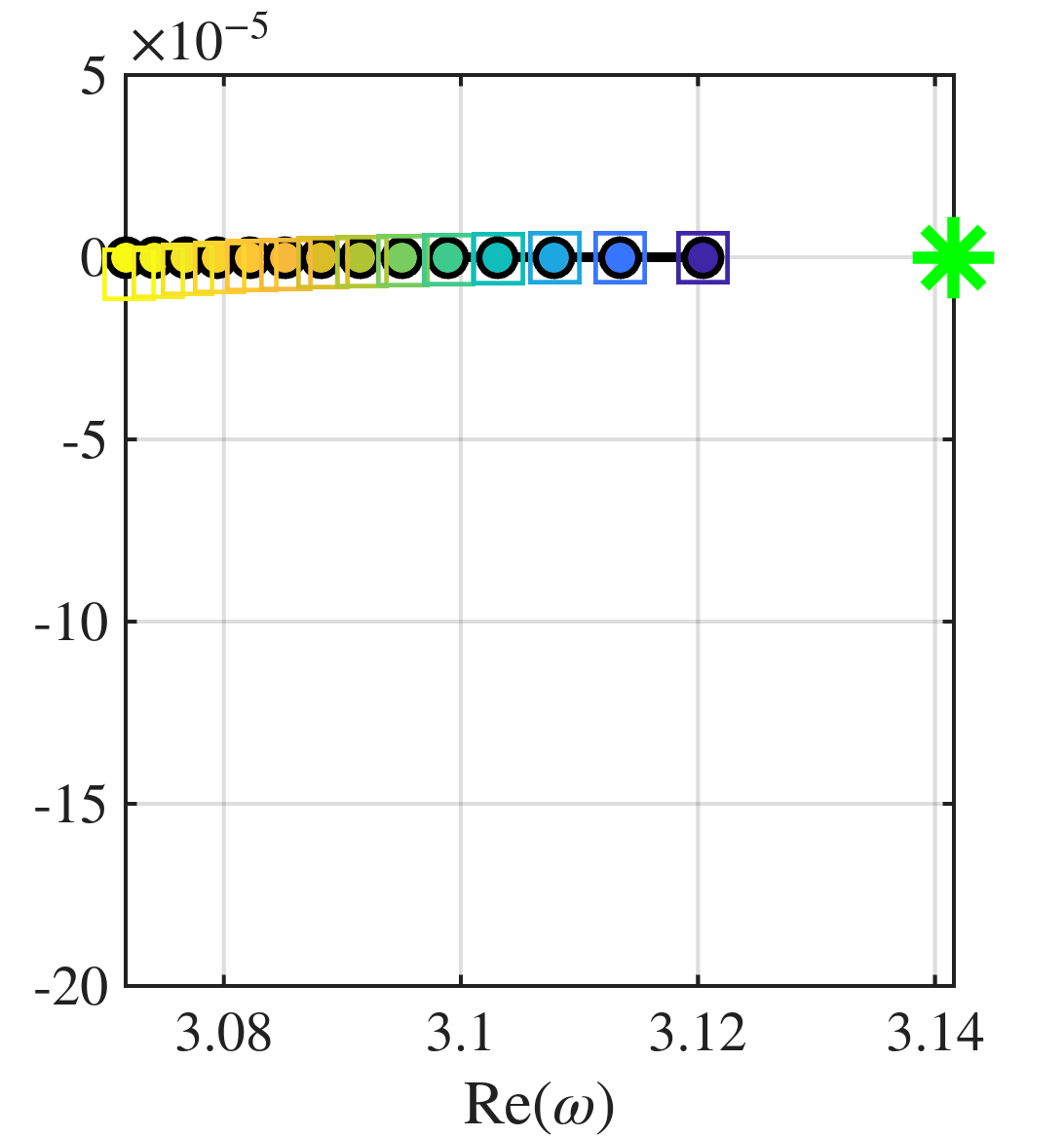}}
    \caption{$c_1=-1,c_2=1/6.$}
    \label{fig:sub6}
\end{subfigure}
\begin{subfigure}[b]{0.24\textwidth}
    \raisebox{0pt}{\includegraphics[height=0.18\textheight]{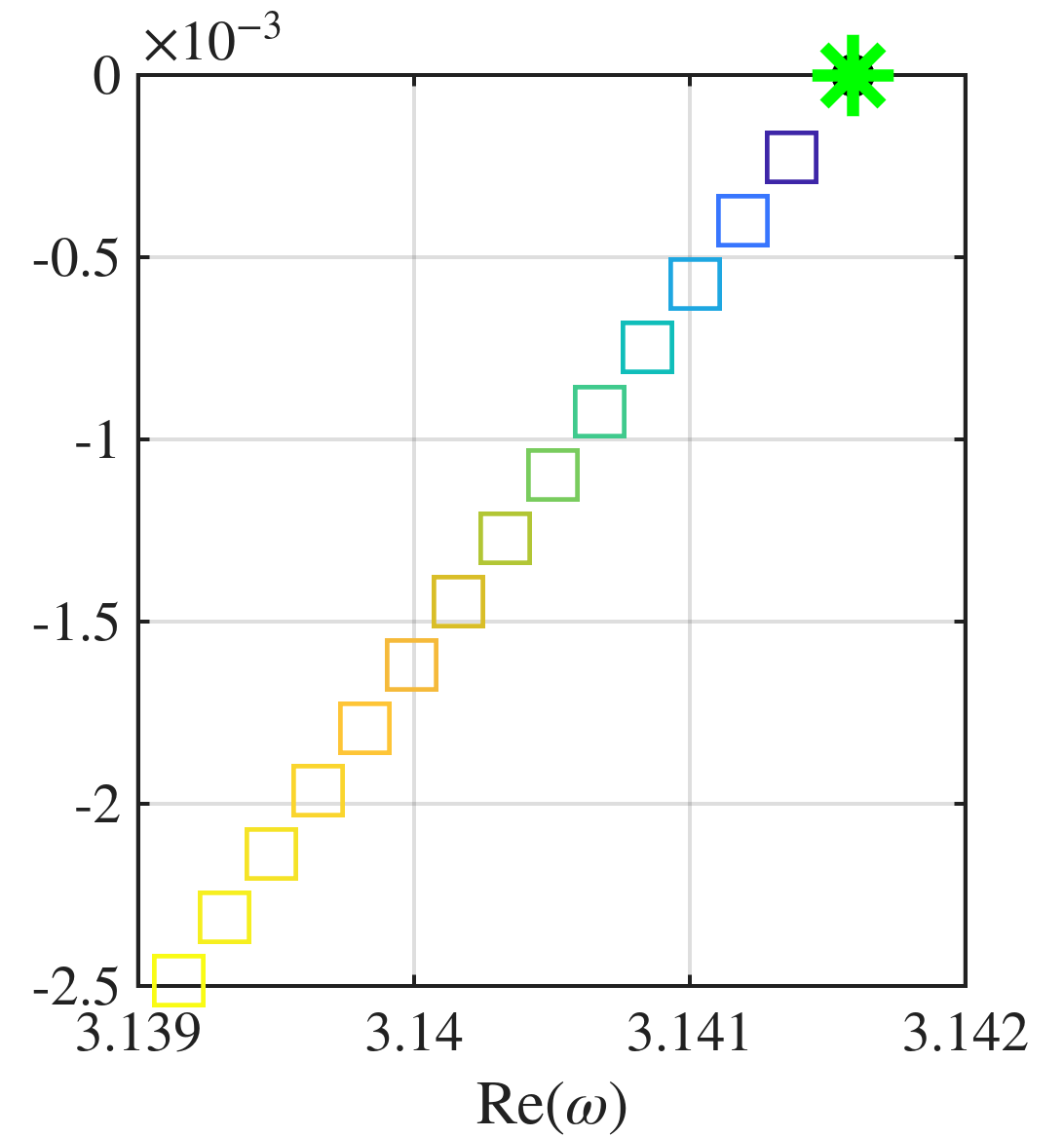}}
    \caption{$\omega(\delta)=\pi+\mc{O}(\delta)$.}
    \label{fig:sub7}
\end{subfigure}
\begin{subfigure}[b]{0.24\textwidth}
\raisebox{16pt}{\includegraphics[height=0.18\textheight]{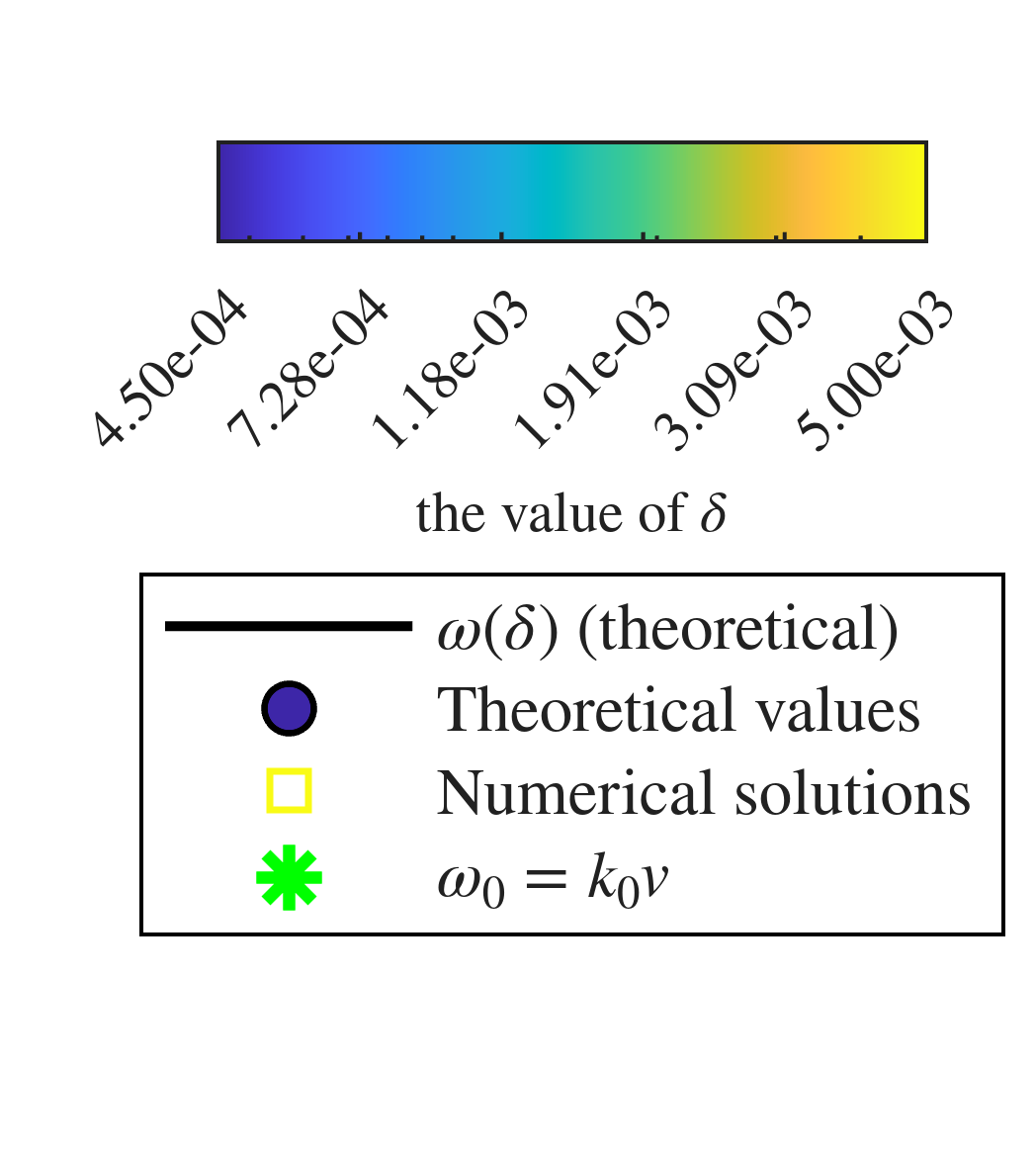}}
    \label{fig:sub8}
\end{subfigure}

\caption{Comparison of theoretical asymptotic expansions with numerical results for the 
resonant frequencies. The parameters are $N=6$, $\bm{t} = (1, 2, 1, 0.75, 1.25, 1, 2, 2, 
1, 1.5, 0.5)^{\top}$, $k_0 = \pi$, $r = 1$, and $v = 1$. The nonzero eigenvalues of 
$\mathcal{C}(k_0)$ are $\lambda = 1$ (with multiplicity two) and $\lambda = 0.25$ 
(simple). For branches~(a)--(f), the refined expansion~\eqref{equ: omega^i,pm 
expand} gives $\omega(\delta) = \pi + c_1 \delta^{1/2} + c_2 \delta + 
\mathcal{O}(\delta^{3/2})$, with coefficients $(c_1, c_2)$ indicated in each 
panel. Branch~(g) corresponds to a resonant frequency satisfying 
$\omega(\delta) = \pi + \mathcal{O}(\delta)$ as in \eqref{eq:eigenfrequencydelta1}.}
\label{fig:all}
\end{figure}

\begin{figure}[!htb]
    \centering
    \includegraphics[width=0.6\textwidth]{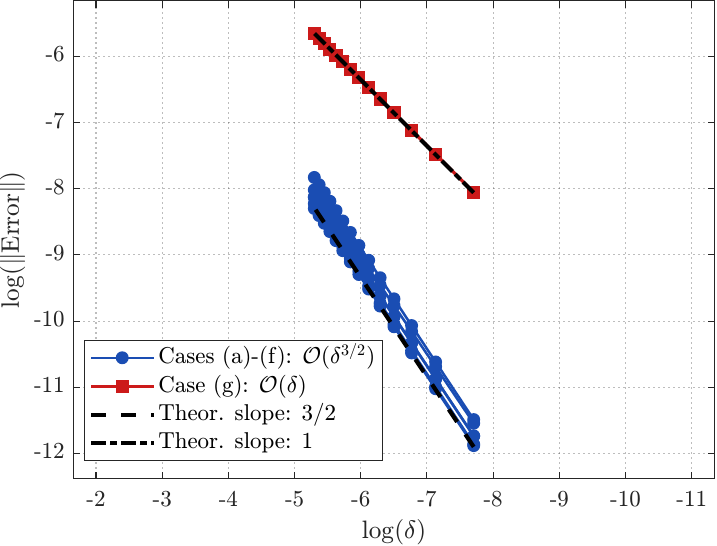} 
    \caption{Log-log plot of the error between numerical resonant frequencies and their asymptotic expansions.
The theoretical slopes of $\mathcal{O}(\delta^{3/2})$ (for branches (a)--(f) in Figure \ref{fig:all}) and $\mathcal{O}(\delta)$ (for branch (g) in Figure \ref{fig:all}) are confirmed by the linear fits, matching the predicted asymptotic orders.}
    \label{fig:convergence}
\end{figure}
     Suppose that $\ww(\d) =k_0v\pm v\sqrt{\lambda_j\delta/r}+\mc{O}(\delta)$ with $\lambda_j$ being a simple nonzero eigenvalue of $\mathcal C(k_0)$ is a resonant frequency for the scattering problem \eqref{equ: scattering problem} and let $u(x)$ be the corresponding non-trivial solution.
    
     For $k_0=0$, we recall from \cite{feppon.cheng.ea2023Subwavelength} that
    \begin{equation}\label{equ: u(x) estimate}
    u(x)=\begin{cases}
            \alpha_j+\mc{O}(\d^{1/2}), &x\in (x_{2j-1}, x_{2j}),\qquad j=1,2,\cdots ,N,\\
            \alpha_j+\beta_j(x-x_{2j})+\mc{O}(\d^{1/2}), &x\in(x_{2j},x_{2j+1}), \qquad j=1,2,\cdots ,N-1,
        \end{cases}
    \end{equation}
    where $\al=(\alpha_1,\alpha_2,\cdots,\alpha_N)^{\top}$ is the corresponding eigenvector of the capacitance matrix $\mathcal{C}(0)$ and $\beta_j=\tfrac{\alpha_{j+1}-\alpha_j}{s_j},j=1,2,\cdots,N-1$. 
    
Beyond the subwavelength regime, that is, for $k_0 \neq 0$, deriving an explicit asymptotic expression for the eigenmodes generalizing~\eqref{equ: u(x) estimate} 
is considerably more involved, as shown in~\Cref{sec:6}. For simplicity, we restrict to the case where $\lambda$ is a simple eigenvalue of $\mathcal{C}(k_0)$. The multiple eigenvalue case requires additional assumptions and is illustrated 
through examples in~\Cref{sec: multiple case}, while a rigorous analysis is deferred to future work. 


We now state the precise setup and result. Without loss of generality, we set $r=1$. We let $\lambda$ be a simple eigenvalue of $\mathcal{C}(k_0)$. By Theorem \ref{thm: C_j's}, there exists a unique index $j^*$ such that $\lambda$ is a simple eigenvalue of a submatrix $C_{j^*}$ of $\mathcal{C}(k_0)$. Let $\mathcal{I}_{j^*}=\lb a_{j^*},b_{j^*}\rb$ be the corresponding target resonant interval, and let $s$ be the number of even integers in $\mathcal{I}_{j^*}$, namely $2\lceil a_{j^*}/2\rceil,2\lceil a_{j^*}/2\rceil+2,\dots,2\lfloor b_{j^*}/2\rfloor$. Set $j_0=2\lceil a_{j^*}/2\rceil-2$. Let $k_0\neq 0$ and suppose that $k(\delta)=k_0\pm\sqrt{\lambda\delta/r}+\mathcal{O}(\delta)$ is a resonant frequency and $u(x)$ its associated eigenmode.

\begin{theorem}[(Eigenmodes for a simple resonance)]\label{thm: eigenmode simple case}
Under the above settings, for each $1\leq j\leq s$, the eigenmode $u(x)$ has the form
\begin{align}\label{equ: u(x) approximate}
u(x)=\beta_{j_0+2j}\sin\bigl(k_0(x-x_{a_{j^*}})\bigr)+\mathcal{O}(\delta^{1/2}),\quad x\in(x_{j_0+2j},x_{j_0+2j+1}),
\end{align}
where $(\lambda,\be=(\beta_{j_0+2j})_{j=1}^s)$ is an eigenpair of the matrix $D_{j^*}$ defined in (\ref{equ: matrix D def}). Moreover, on any interval $(x_i,x_{i+1})$ with $i\notin\{j_0+2j:1\le j\le s\}$, there exists some positive integer $\gamma_i$ such that
\[
u(x)=\mathcal{O}(\delta^{\gamma_i/2}),\quad x\in(x_i,x_{i+1}).
\]
\end{theorem}
For a schematic understanding of Theorem \ref{thm: eigenmode simple case}, see Figure \ref{fig:target interval}; a numerical verification is provided in Figure \ref{fig:eigenmodesimple}.
\begin{figure}[!htb]
    \centering
    \begin{adjustbox}{width=\textwidth}
    \begin{tikzpicture}
        \coordinate (x1) at (1,0);
        \path (x1) +(1,0) coordinate (x2);
        \path (x2) +(1,0) coordinate (x3);
        \path (x3) +(0.5,0) coordinate (dots1);
        \path (dots1) +(0.6,0) coordinate (dots);
        \path (dots) +(0.6,0) coordinate (dots2);
        \path (dots2) +(0.5,0) coordinate (xa1);
        \path (xa1) +(1.5,0) coordinate (xa2);
        \path (xa2) +(1.5,0) coordinate (xa3);
        \path (xa3) +(1.5,0) coordinate (xa4);
        \path (xa4) +(1.5,0) coordinate (xa5);
        \path (xa5) +(0.5,0) coordinate (dots3);
        \path (dots3) +(0.6,0) coordinate (dots4);
        \path (dots4) +(0.6,0) coordinate (dots5);
        \path (dots5) +(0.5,0) coordinate (xb1);
        \path (xb1) +(1.5,0) coordinate (xb2);
        \path (xb2) +(1.5,0) coordinate (xb3);
        \path (xb3) +(1.5,0) coordinate (xb4);
        \path (xb4) +(0.5,0) coordinate (dots6);
        \path (dots6) +(0.6,0) coordinate (dots7);
        \path (dots7) +(0.6,0) coordinate (dots8);
        \path (dots8) +(0.5,0) coordinate (xc1);
        \path (xc1) +(1.2,0) coordinate (xc2);
        \path (xc2) +(1.2,0) coordinate (xc3);

        \def\barHeight{1}   
        \fill[green!20] (xa1 |- 0,-\barHeight) rectangle (xa2 |- 0,\barHeight);
        \fill[green!60] (xa2 |- 0,-\barHeight) rectangle (xb3 |- 0,\barHeight);
        \fill[green!20] (xb3 |- 0,-\barHeight) rectangle (xb4 |- 0,\barHeight);

        \draw[ultra thick] (x1) -- (x2);
        \draw[dashed, thick] (x2) -- (x3);
        \node at (dots) {$\cdots$};
        \draw[ultra thick] (xa1) -- (xa2);
        \draw[dashed, thick] (xa2) -- (xa3);
        \draw[ultra thick] (xa3) -- (xa4);
        \draw[dashed, thick] (xa4) -- (xa5);
        \node at (dots4) {$\cdots$};
        \draw[ultra thick] (xb1) -- (xb2);
        \draw[dashed, thick] (xb2) -- (xb3);
        \draw[ultra thick] (xb3) -- (xb4);
        \node at (dots7) {$\cdots$};
        \draw[ultra thick] (xc1) -- (xc2);
        \draw[dashed, thick] (xc2) -- (xc3);

        \foreach \point/\name in {
            x1/$x_1$,
            x2/$x_2$,
            x3/$x_3$,
            xa1/$x_{j_0+1}$,
            xa2/$x_{j_0+2}$,
            xa3/$x_{j_0+3}$,
            xa4/$x_{j_0+4}$,
            xa5/$x_{j_0+5}$,
            xb1/$x_{j_0+2s-1}$,
            xb2/$x_{j_0+2s}$,
            xb3/$x_{j_0+2s+1}$,
            xb4/$x_{j_0+2s+2}$,
            xc1/$x_{2N-2}$,
            xc2/$x_{2N-1}$,
            xc3/$x_{2N}$}
        {
            \draw[dotted, thick] (\point) -- (\point |- 0,-1);       
            \draw[dotted, thick] (\point |- 0,-0.5) -- (\point |- 1,-0.5); 
            \node[below] at (\point |- 1,-1) {\name};                 
        }

        \foreach \point in {x1, xa2, xa3, xa4, xa5, xb1, xb2, xb3, xc3} {
            \draw[dotted, thick] (\point) -- (\point |- 0,1.5);   
        }
        \draw[dotted, thick] (xa1) -- (xa1 |- 0,1.8);
        \draw[dotted, thick] (xb4) -- (xb4 |- 0,1.8);

        \def\arrowY{1.2}   

        \draw[<->, thick] (x1 |- 0,\arrowY) -- (xa2 |- 0,\arrowY)
            node[midway, above, font=\small] {$\mathcal{O}(\delta^{\gamma_i/2})$};
        \draw[<->, thick] (xa3 |- 0,\arrowY) -- (xa4 |- 0,\arrowY)
            node[midway, above, font=\small] {$\mathcal{O}(\delta^{\gamma_i/2})$};
        \draw[<->, thick] (xb1 |- 0,\arrowY) -- (xb2 |- 0,\arrowY)
            node[midway, above, font=\small] {$\mathcal{O}(\delta^{\gamma_i/2})$};
        \draw[<->, thick] (xb3 |- 0,\arrowY) -- (xc3 |- 0,\arrowY)
            node[midway, above, font=\small] {$\mathcal{O}(\delta^{\gamma_i/2})$};

        \draw[<->, thick, red] (xa2 |- 0,\arrowY) -- (xa3 |- 0,\arrowY)
            node[midway, above, font=\normalfont] {trig.};
        \draw[<->, thick, red] (xa4 |- 0,\arrowY) -- (xa5 |- 0,\arrowY)
            node[midway, above, font=\normalfont] {trig.};
        \draw[<->, thick, red] (xb2 |- 0,\arrowY) -- (xb3 |- 0,\arrowY)
            node[midway, above, font=\normalfont] {trig.};

        \def\braceY{2.0}
        \draw[decorate,decoration={brace,amplitude=5pt}, thick] (xa1 |- 0,\braceY) -- (xb4 |- 0,\braceY) 
            node[midway, above, yshift=5pt] {the target resonant interval $\mathcal{I}_{j^*}=\llbracket a_{j^*},b_{j^*}\rrbracket$};

        \begin{scope}[shift={(xa2)}]
            \draw[thick, red, domain=0:1.5, samples=80] plot (\x, {0.9*sin(3*\x*120)});
        \end{scope}
        \begin{scope}[shift={(xa4)}]
            \draw[thick, red, domain=0:1.5, samples=80] plot (\x, {-0.5*sin(2*\x*120)});
        \end{scope}
        \begin{scope}[shift={(xb2)}]
            \draw[thick, red, domain=0:1.5, samples=80] plot (\x, {0.7*sin(4*\x*120)});
        \end{scope}

\node[font=\bfseries] at ($(xa1)!0.5!(xa2)$) [yshift=6pt] {?};
\node[font=\bfseries] at ($(xb3)!0.5!(xb4)$) [yshift=6pt] {?};

    \end{tikzpicture}
    \end{adjustbox}
  \caption{Schematic illustration of Theorem~\ref{thm: eigenmode simple case}. 
To leading order, the eigenmode is approximated by a trigonometric function on 
the spacings within the target block $\mathcal{I}_{j^*}$, with amplitudes 
determined by the eigenvector of $D_{j^*}$, and vanishes elsewhere. The green 
region marks $\mathcal{I}_{j^*}$; lighter green segments at the two ends (marked 
$?$) may or may not belong to $\mathcal{I}_{j^*}$ depending on the parity of the endpoints $a_{j^*}$ and $b_{j^*}$.}
    \label{fig:target interval}
\end{figure}
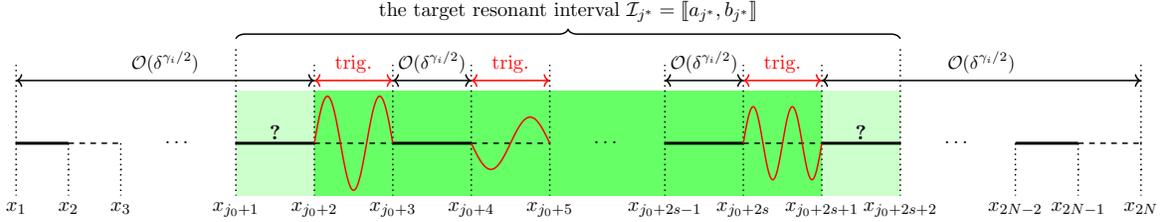
\begin{figure}[!htb] 
\centering
\begin{subfigure}[b]{0.41\textwidth}
    \centering
    \includegraphics[width=\textwidth]{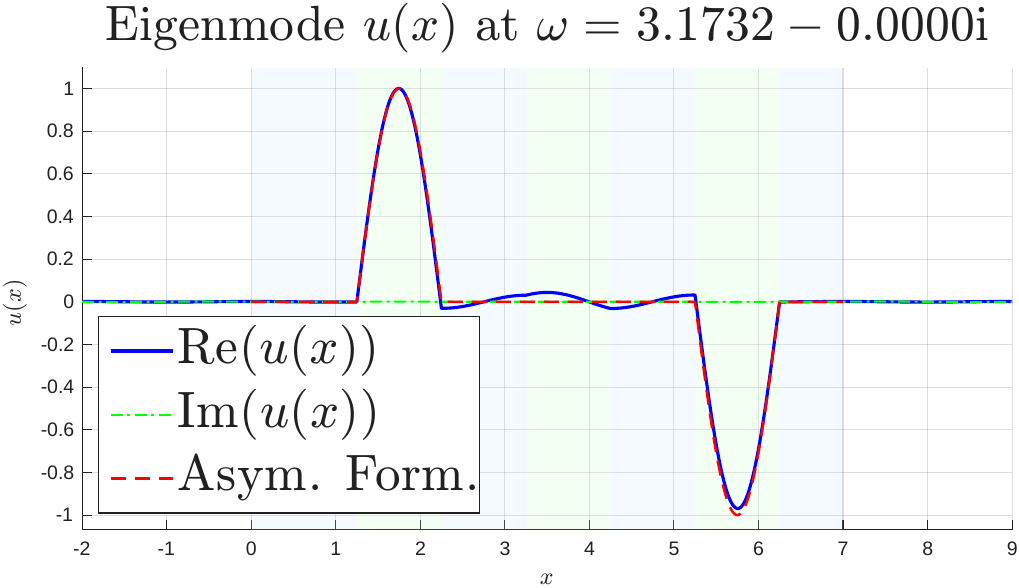}
    \label{fig:mode1}
\end{subfigure}
\vspace{0pt}
\begin{subfigure}[b]{0.41\textwidth}
    \centering
    \includegraphics[width=\textwidth]{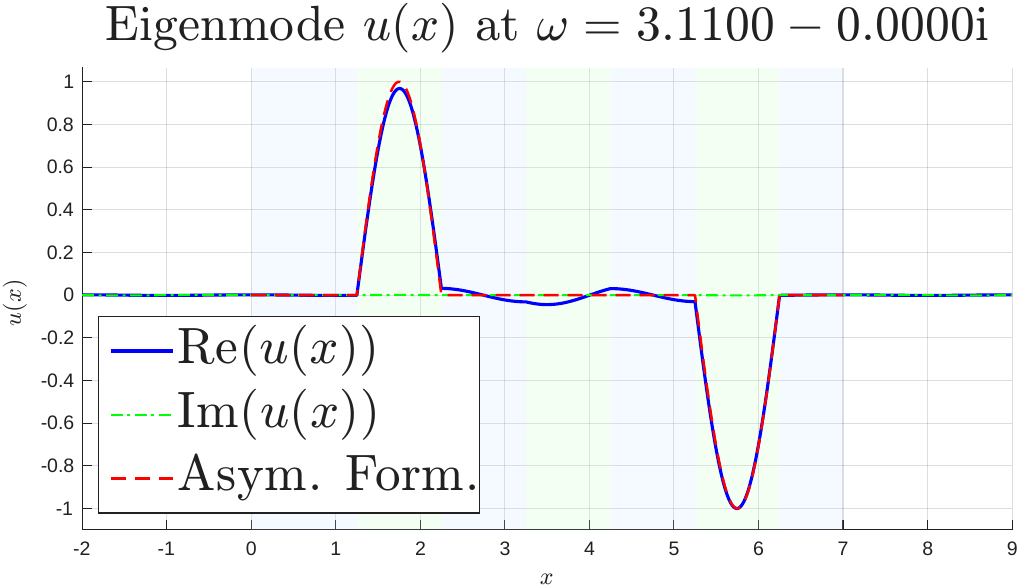}
    \label{fig:mode2}
\end{subfigure}

\begin{subfigure}[b]{0.41\textwidth}
    \centering
    \includegraphics[width=\textwidth]{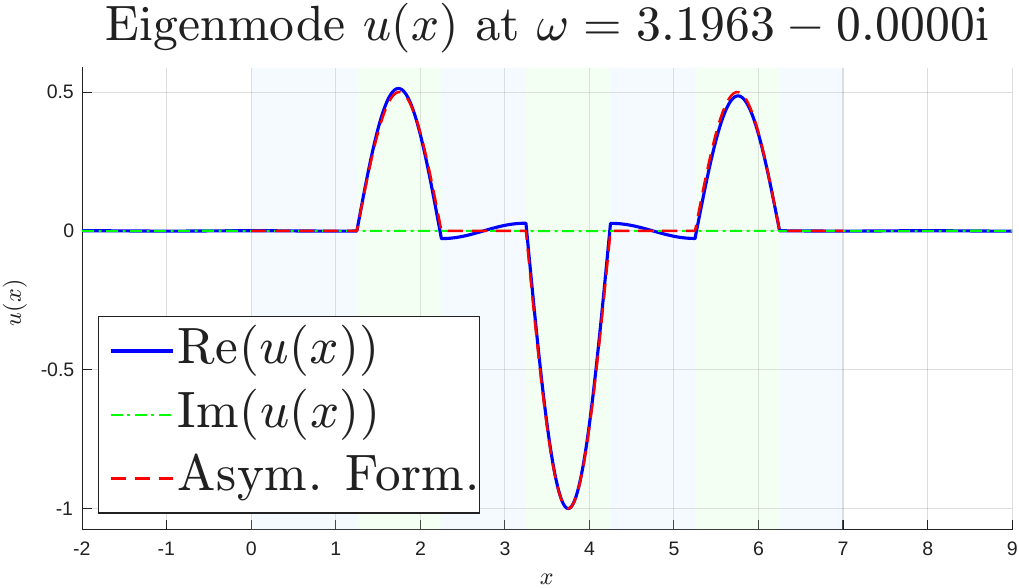}
    \label{fig:mode3}
\end{subfigure}
\vspace{0pt}
\begin{subfigure}[b]{0.41\textwidth}
    \centering
    \includegraphics[width=\textwidth]{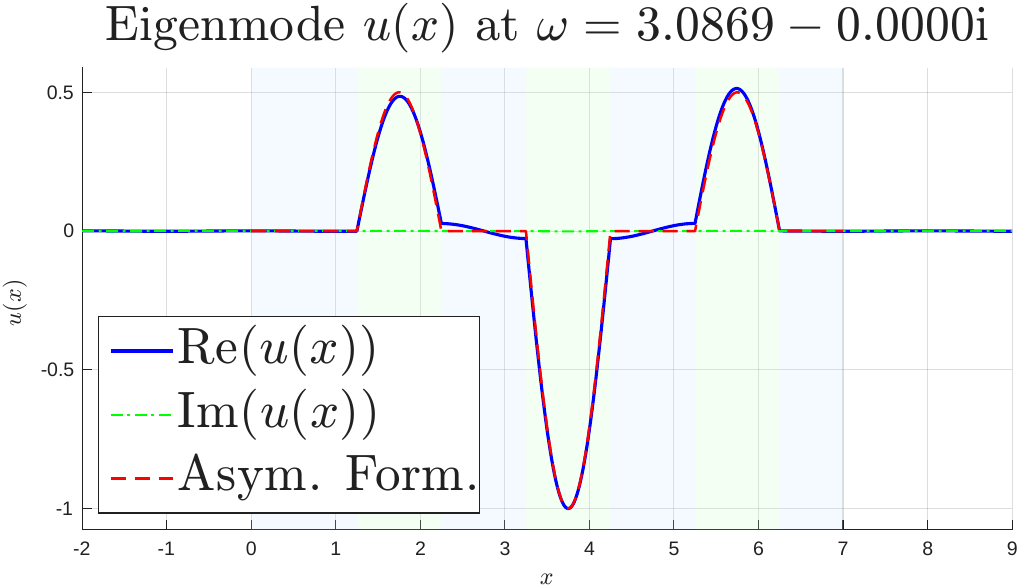}
    \label{fig:mode4} 
\end{subfigure}
\caption{Eigenmodes \(u(x)\) for the four \(k_0v+\mathcal{O}(\delta^{1/2})\) resonant frequencies (parameters: $N=4$, \(\bm t = (1.25\,, 1\,, 1\,, 1\,, 1\,, 1\,, 0.75)^\top\), \(k_0 = \pi\), \(r = 1\), \(v = 1\), \(\delta = 0.001\)), comparing numerical solutions with the asymptotic formula (\ref{equ: u(x) approximate}). The structure of these asymptotic eigenmodes is governed by the matrix \(D_1\), which admits two eigenpairs: \((1, (1, 0, -1)^\top)\) and \((3, (1, -2, 1)^\top)\). Blue and green shaded regions denote the structural lengths and spacings, respectively.}  \label{fig:eigenmodesimple}
\end{figure}

\section{Block partition of the capacitance matrix} \label{sec:3}
In this section, we study the spectral properties of the frequency-dependent capacitance matrix $\mathcal{C}(k_0), k_0 \in E$. In particular, we give a bound on the number of nonzero eigenvalues of $\mathcal{C}(k_0)$ and compute a characteristic polynomial of $\mathcal{C}(k_0)$ (see \Cref{thm: C(k) nonzero eigenvalue}).

\subsection{Preliminaries}\label{sec: block matrices}
Let the parameters $t_1, t_2, \ldots, t_{2N-1}$ and $\theta_1, \theta_2, \ldots, 
\theta_{2N-2}$ be given in \eqref{equ: bm t def} and \eqref{eq:coupling}, respectively. For an integer interval 
$\mathcal{I} = \llbracket a, b \rrbracket \subset \llbracket 1, 2N-1 \rrbracket$, 
let $s$ and $l$ denote the number of even and odd integers in $\mathcal{I}$, 
respectively:
\[
    s = \left\lfloor\frac{b}{2}\right\rfloor - \left\lfloor\frac{a-1}{2}\right\rfloor,
    \qquad
    l = \left\lceil\frac{b}{2}\right\rceil - \left\lceil\frac{a-1}{2}\right\rceil.
\]
Note that $|s - l| \leq 1$. We also define the parities of the endpoints by 
$\xi = a - 2\lfloor a/2 \rfloor$ and $\eta = b - 2\lfloor b/2 \rfloor$, so that 
$\xi = 0$ if $a$ is even and $\xi = 1$ if $a$ is odd, and similarly for $\eta$. 
The $s \times l$ matrix $A$ is then defined by
\begin{align}\label{equ: A def}
    A_{i,\, i-1+\xi} = -1, \qquad A_{i,\, i+\xi} = 1,
\end{align}
with all other entries equal to zero. We introduce
\begin{equation} \label{equ: matrix S T def}
    \begin{aligned} 
        & \mathcal{S} := \operatorname{diag}\!\left\{ 
        t_{2\lceil a/2 \rceil},\, t_{2\lceil a/2 \rceil + 2},\, \ldots,\, 
        t_{2\lfloor b/2 \rfloor}  \right\},\\
        &   \mathcal{L} := \operatorname{diag}\!\left\{ 
        t_{2\lfloor a/2 \rfloor + 1},\, t_{2\lfloor a/2 \rfloor + 3},\, \ldots,\, 
        t_{2\lceil b/2 \rceil - 1} 
    \right\},
    \end{aligned}
\end{equation}
so that $\mathcal{S}$ is an $s \times s$ diagonal matrix with even-indexed entries 
and $\mathcal{L}$ is an $l \times l$ diagonal matrix with odd-indexed entries. 
We then define the $l \times l$ matrix
\[
    C = C^{\xi,\eta}(\theta_a, \ldots, \theta_{b-1}) 
    := \mathcal{L}^{-1} A^\top \mathcal{S}^{-1} A,
\]
which depends only on the coupling coefficients $\theta_a, \ldots, \theta_{b-1}$. 
In the special case $a = 1$, $b = 2N-1$ (so that $(\xi, \eta) = (1,1)$), $C$ reduces for $k_0=0$ to the classical capacitance matrix $\mathcal{C}(0)$ 
introduced in~\cite{feppon.cheng.ea2023Subwavelength}. We will also see in the next subsection that $C$ corresponds to each nonzero block of our frequency-dependent capacitance matrix $\mathcal{C}(k_0)$. The explicit forms of $C^{\xi,\eta}$ for each choice of parities are as follows:

\[
C^{1,1} = \begin{pmatrix}
    \theta_a & -\theta_a & & \\
    -\theta_{a+1} & \theta_{a+1}+\theta_{a+2} & -\theta_{a+2} & \\
    & -\theta_{a+3} & \theta_{a+3}+\theta_{a+4} & -\theta_{a+4} & \\
    & & \ddots & \ddots & \ddots \\
    & & & -\theta_{b-3} & \theta_{b-3}+\theta_{b-2} & -\theta_{b-2} \\
    & & & & -\theta_{b-1} & \theta_{b-1}
\end{pmatrix},
\]
\[
C^{0,1}=\begin{pmatrix}
    \theta_a+\theta_{a+1}&-\theta_{a+1}&\\
    -\theta_{a+2}&\theta_{a+2}+\theta_{a+3}&-\theta_{a+3}&\\
    &-\theta_{a+4}&\theta_{a+4}+\theta_{a+5}&-\theta_{a+5}&\\&&\ddots&\ddots&\ddots&\\
    &&&-\theta_{b-3}&\theta_{b-3}+\theta_{b-2}&-\theta_{b-2}\\&&&&-\theta_{b-1}&\theta_{b-1}
\end{pmatrix},
\]
\[
C^{1,0}=\begin{pmatrix}
    \theta_a&-\theta_a&\\
    -\theta_{a+1}&\theta_{a+1}+\theta_{a+2}&-\theta_{a+2}&\\
    &-\theta_{a+3}&\theta_{a+3}+\theta_{a+4}&-\theta_{a+4}&\\&&\ddots&\ddots&\ddots&\\
    &&&-\theta_{b-4}&\theta_{b-4}+\theta_{b-3}&-\theta_{b-3}\\&&&&-\theta_{b-2}&\theta_{b-2}+\theta_{b-1}
\end{pmatrix},
\]
\[
C^{0,0}=\begin{pmatrix}
    \theta_a+\theta_{a+1}&-\theta_{a+1}&\\
    -\theta_{a+2}&\theta_{a+2}+\theta_{a+3}&-\theta_{a+3}&\\
    &-\theta_{a+4}&\theta_{a+4}+\theta_{a+5}&-\theta_{a+5}&\\&&\ddots&\ddots&\ddots&\\
    &&&-\theta_{b-4}&\theta_{b-4}+\theta_{b-3}&-\theta_{b-3}\\&&&&-\theta_{b-2}&\theta_{b-2}+\theta_{b-1}
\end{pmatrix}.
\]
We also introduce the $s\times s$ matrix $D$ as
\begin{align}\label{equ: matrix D def}
D=D^{\xi,\eta}(\theta_a,\cdots,\theta_{b-1}):=\mathcal{S}^{-1}\cdot A\mathcal{L}^{-1}A^\top.
\end{align}
Then a direct computation yields
\[
D^{\xi,\eta}(\theta_a,\cdots,\theta_{b-1})=C^{1-\xi,1-\eta}(\theta_a,\cdots,\theta_{b-1}).
\]


\begin{lemma}\label{thm: eig of C}
Let $C^{\xi,\eta}$ and $D^{\xi,\eta}$ be defined as above. Let $\mathfrak{n}:=\# I=b-a+1$. For $0\leq  q\leq\lf\mathfrak{n}/2\rf$, define the constants:
\begin{equation*}
    c_{a,b,q}=\sum_{a\leq j_1\prec j_2\prec\cdots\prec j_q\leq b-1}\prod_{i=1}^q\theta_{j_i},\quad q>0.\qquad  c_{a,b,0}=1. 
\end{equation*}
It holds that: 
\begin{enumerate}
    \item[(i)] All eigenvalues of $C$ are distinct. If $(\xi,\eta) = (1,1)$, then $C$ has 
exactly one zero eigenvalue and all remaining eigenvalues are positive; for all 
other choices of $(\xi,\eta)$, every eigenvalue of $C$ is positive.
    \item[(ii)] $C$ and $D$ share the same nonzero eigenvalues, and their common 
    nonzero characteristic polynomial is
    \[
        P(\lambda) = \sum_{q=0}^{\lfloor \mathfrak{n}/2 \rfloor} 
        c_{a,b,q}(-\lambda)^{\lfloor \mathfrak{n}/2 \rfloor - q}.
    \]
    \item[(iii)] Let $\lambda_0$ be a nonzero eigenvalue of 
    $C = C^{\xi,\eta}(\theta_a, \ldots, \theta_{b-1})$ with corresponding 
    eigenvector $\al = (\alpha_1, \ldots, \alpha_l)^\top$. Let
    \[
P_1(\lambda)=\sum_{q=0}^{\lf(\mathfrak{n}-1)/2\rf}c_{a+1,b,q}(-\lambda)^{\lf(\mathfrak{n}-1)/2\rf-q}, \quad  
P_2(\lambda)=\sum_{q=0}^{\lf(\mathfrak{n}-1)/2\rf}c_{a,b-1,q}(-\lambda)^{\lf(\mathfrak{n}-1)/2\rf-q},
    \]
    be the nonzero characteristic polynomials of 
    \begin{equation}\label{equ:tildec1c2}
    \ti{C}_1=C^{1-\xi,\eta}(\theta_{a+1},\cdots,\theta_{b-1}), \quad \ti{C}_2=C^{\xi,1-\eta}(\theta_{a},\cdots,\theta_{b-2}),
    \end{equation}
    respectively. Then
    \[
    \frac{P_1(\lambda_0)}{P'(\lambda_0)}=(-1)^{\xi+\eta}t_a^{2\xi-1}\lambda_0^{\eta(2\xi-1)}\frac{\alpha_1^2}{\al ^\top\mathcal{L}\al},\quad \frac{P_2(\lambda_0)}{P'(\lambda_0)}=(-1)^{\xi+\eta}t_b^{2\eta-1}\lambda_0^{\xi(2\eta-1)}\frac{\alpha_l^2}{\al^\top\mathcal{L}\al};
    \]
    \item[(iv)] Let $\mu_1 \mu_2 = \lambda > 0$, $\al \in \mathbb{R}^l$, and 
    $\be \in \mathbb{R}^s$. Consider the coupled systems
    \[
        \mathrm{(1)}\ A\al = \mu_1 \mathcal{S}\be, \qquad 
        \mathrm{(2)}\ A^\top \be = \mu_2 \mathcal{L}\al.
    \]
    If all equations in \textnormal{(1)} and \textnormal{(2)} are satisfied except 
    possibly one specific row, which may be the first or last row of either 
    \textnormal{(1)} or \textnormal{(2)}, then this exceptional row holds if and 
    only if $\lambda$ is an eigenvalue of $C$ (or equivalently, of $D$). In that 
    case, $(\lambda, \al)$ is an eigenpair of $C$ and $(\lambda, \be)$ is an 
    eigenpair of $D$. 
\end{enumerate}

\end{lemma}
\begin{proof}
As for item (i), the first statement follows from the fact that $C$ is a tridiagonal matrix with 
nonzero off-diagonal elements; see~\cite[Lemma~7.7.1]{parlett1998symmetric}. For the second, we observe that 
$\mathcal{L}^{1/2} C \mathcal{L}^{-1/2} = \mathcal{L}^{-1/2} A^\top 
\mathcal{S}^{-1} A \mathcal{L}^{-1/2}$ is symmetric and positive semidefinite, 
since
\[
    \bm{x}^\top \mathcal{L}^{-1/2} A^\top \mathcal{S}^{-1} A \mathcal{L}^{-1/2} 
    \bm{x} \geq 0
\]
for all $\bm{x}$, with equality if and only if $A \mathcal{L}^{-1/2} \bm{x} = 
\bm{0}$. Since $\operatorname{rank}(A \mathcal{L}^{-1/2}) = \operatorname{rank}(A) 
= \min\{s, l\}$, the equation $A \mathcal{L}^{-1/2} \bm{x} = \bm{0}$ admits a 
nontrivial solution only when $s < l$, that is, when $(\xi, \eta) = (1,1)$, which 
gives the second statement. 

For item (ii), $C=\mathcal{L}^{-1}A^\top\cdot \mathcal{S}^{-1}A$ and  $D=\mathcal{S}^{-1}A\cdot\mathcal{L}^{-1}A^\top$
are products of the same two matrices in reverse order, so they share the same nonzero eigenvalues. 
For the characteristic polynomial, when $(\xi,\eta) = (1,1)$, \cite[Lemma~4.3(2)]{pm1} gives
\[
    |C - \lambda I| = \sum_{q=0}^{(\mathfrak{n}-1)/2} 
    c_{a,b,q}(-\lambda)^{(\mathfrak{n}+1)/2 - q},
\]
and the result follows since $C$ has exactly one zero eigenvalue. When 
$(\xi,\eta) = (0,1)$, let $Q^1$ denote the $(1,1)$-cofactor of $|C - \lambda I|$. 
Expanding along the first column yields 
\begin{align*}
&|C-\lambda I|=\sum_{q=0}^{\mathfrak{n}/2-1}c_{a,b,q}(-\lambda)^{\mathfrak{n}/2-q}+\theta_a Q^1(\lambda)
\\=&\sum_{q=0}^{\mathfrak{n}/2-1}(-\lambda)^{\mathfrak{n}/2-q}\sum_{a\leq j_1\prec j_2\prec\cdots\prec j_q\leq b-1}\prod_{i=1}^q\theta_{j_i}+\theta_a \sum_{q=1}^{\mathfrak{n}/2}(-\lambda)^{\mathfrak{n}/2-q}\sum_{a+1\leq j_1\prec j_2\prec\cdots\prec j_{q-1}\leq b-1}\prod_{i=1}^{q-1}\theta_{j_i}\\
=&\sum_{q=0}^{\mathfrak{n}/2}(-\lambda)^{\mathfrak{n}/2-q}\sum_{a\leq j_1\prec j_2\prec\cdots\prec b-1}\prod_{i=1}^q\theta_{j_i}=\sum_{q=0}^{\mathfrak{n}/2}c_{a,b,q}(-\lambda)^{\mathfrak{n}/2-q},
\end{align*}
and the result follows since $C$ has no zero eigenvalues in this case. The 
remaining cases $(\xi,\eta) = (1,0)$ and $(\xi,\eta) = (0,0)$ are analogous and 
omitted.

For item~(iii), since $\lambda_0$ is simple, $C - \lambda_0 I$ has rank $l-1$, so 
$\operatorname{adj}(C - \lambda_0 I)$ has rank $1$. Hence there exists $c \neq 0$ 
such that $\operatorname{adj}(C - \lambda_0 I) = c \al\widetilde{\al}$, 
where $\al$ and $\widetilde{\al}$ are the right and left eigenvectors of $C$, 
respectively. Since $\mathcal{L}^{1/2} C \mathcal{L}^{-1/2}$ is symmetric, the left 
eigenvector satisfies $\widetilde{\al} = \mathcal{L}\al$, and therefore
\[
    \left.\frac{\mathrm{d}}{\mathrm{d}\lambda}\right|_{\lambda=\lambda_0} 
    |C - \lambda I| 
    = -\operatorname{tr}\!\left(\operatorname{adj}(C - \lambda_0 I)\right) 
    = -c\, \al^\top \mathcal{L}\al.
\]
We now compute $P_1(\lambda_0)/P'(\lambda_0)$ for each parity case.

\medskip
\noindent\textit{Case $\xi = 1$.} Here $\widetilde{C}_1$ has no zero eigenvalues, so
\[
    P_1(\lambda_0) = |\widetilde{C}_1 - \lambda_0 I| 
    = \operatorname{adj}(C - \lambda_0 I)_{11} 
    = \beta \mathcal{L}_{11} \alpha_1^2 = c t_a \alpha_1^2.
\]
If $(\xi,\eta) = (1,1)$, then $P'(\lambda_0) = (c/\lambda_0)\,\al^\top 
\mathcal{L}\al$, giving
\[
    \frac{P_1(\lambda_0)}{P'(\lambda_0)} = \lambda_0 t_a 
    \frac{\alpha_1^2}{\al^\top \mathcal{L}\al}.
\]
If $(\xi,\eta) = (1,0)$, $C$ has no zero eigenvalues, and thus the coefficient becomes $-t_a$.

\medskip
\noindent\textit{Case $\xi = 0$.} Here $C$ has no zero eigenvalues. Letting $E_{11}$ 
denote the elementary matrix with a $1$ in position $(1,1)$ and zeros elsewhere,
\[
    |\widetilde{C}_1 - \lambda_0 I| 
    = |C - \lambda_0 I - \theta_a E_{11}| 
    = |C - \lambda_0 I| - \theta_a \operatorname{adj}(C - \lambda_0 I)_{11} 
    = -c \theta_a t_{a+1} \alpha_1^2 
    = -c \frac{\alpha_1^2}{t_a}.
\]
If $(\xi,\eta) = (0,1)$, then $\widetilde{C}_1$ has a zero eigenvalue, so
\[
    P_1(\lambda_0) = \frac{|\widetilde{C}_1 - \lambda_0 I|}{-\lambda_0} 
    = \frac{c \alpha_1^2}{\lambda_0 t_a}, \qquad
    \frac{P_1(\lambda_0)}{P'(\lambda_0)} 
    = -\frac{1}{\lambda_0 t_a}\frac{\alpha_1^2}{\al^\top \mathcal{L}\al}.
\]
If $(\xi,\eta) = (0,0)$, then $\widetilde{C}_1$ has no zero eigenvalue and the 
coefficient becomes $1/t_a$. The formula for $P_2(\lambda_0)/P'(\lambda_0)$ follows 
by an analogous argument.

As for item (iv), we only consider the case where the exceptional row is the last row of (2); the other cases are similar. From (1) we have  
\begin{align}\label{equ: b relation}
\be = \mu_1^{-1} \mathcal{S}^{-1} A \al.
\end{align}
Substituting into the first \(l-1\) rows of (2) gives, for \(i=1,\dots,l-1\),  
\[
(A^\top \be)_i = \mu_2 (\mathcal{L}\al)_i \Longrightarrow (A^\top\mathcal{S}^{-1}A\al)_i = \lambda (\mathcal{L}\al)_i,
\]  
and multiplying by \(\mathcal{L}^{-1}\) yields  
\begin{align}\label{equ: Ca=lam a}
(C\al)_i = \lambda \al_i,\quad i=1,\cdots,l-1.
\end{align}
The last row of (2) is \((A^\top\be)_l = \mu_2 (\mathcal{L}\al)_l\), which by (\ref{equ: b relation}) is equivalent to  
\begin{align}\label{Ca l=lam a l}
(C\al)_l = \lambda \al_l.
\end{align}
If the last row holds, then (\ref{Ca l=lam a l}) gives \(C\al=\lambda\al\), and from (\ref{equ: b relation}) we obtain  
\[
D\be = \mathcal{S}^{-1}A\mathcal{L}^{-1}A^\top\be = \mathcal{S}^{-1}A\mathcal{L}^{-1}(\mu_2\mathcal{L}\al) = \mu_2\mathcal{S}^{-1}A\al = \mu_2\mathcal{S}^{-1}(\mu_1\mathcal{S}\be) = \lambda\be.
\]

Conversely, assume that \(\lambda\) is an eigenvalue of \(C\). Let \(\tilde{C}\) be the submatrix of \(C-\lambda I\) consisting of its first \(l-1\) rows. Then (\ref{equ: Ca=lam a}) implies \(\al\in\operatorname{Ker}\tilde{C}\). Since \(C\) is tridiagonal with nonzero off-diagonals, its first \(l-1\) rows are linearly independent, so \(\operatorname{rank}(\tilde{C}) = \operatorname{rank}(C-\lambda I) = l-1\). The rank--nullity theorem then gives $\dim\operatorname{Ker}\widetilde{C} = 
\dim\operatorname{Ker}(C - \lambda I)$, and since $\operatorname{Ker}(C - \lambda I) 
\subset \operatorname{Ker}\widetilde{C}$, the two kernels coincide. Hence, \(\al\in\operatorname{Ker}(C-\lambda I)\), \emph{i.e.}, \(C\al=\lambda\al\), and (\ref{Ca l=lam a l}) holds, so the last row of (2) is satisfied.
\end{proof}

\begin{remark}
Lemma \ref{thm: eig of C}(iv) can be interpreted as a singular value decomposition of the matrix $K = \mathcal{L}^{-1/2}A^{\top}\mathcal{S}^{-1/2}$. Indeed, if $\lambda > 0$ is an eigenvalue of $C$ (hence of $D$) and $\al,\be$ satisfy the equations in the lemma with $\mu_1=\mu_2:=\mu = \sqrt{\lambda}$, then setting $\bm{x} = \mathcal{L}^{1/2}\al$ and $\bm{y} = \mathcal{S}^{1/2}\be$ yields
\[
K\bm {y} = \mu \bm{x}, \qquad K^{\top}\bm{x} = \mu \bm{y},
\]
so $\mu$ is a singular value of $K$ with right singular vector $\bm y$ and left singular vector $\bm x$. Conversely, any singular value of $K$ gives rise to an eigenpair of $C$ and $D$ via the same transformation.
\end{remark}

\subsection{Block structures of capacitance matrices}
\label{sec: generalized capcitance matrix}
This section elucidates the block structure of capacitance matrix $\mathcal C(k_0)$ and its symmetrized version. We first recall the following standard result from linear algebra.

\begin{lemma}\label{thm: tridiag eig}
Let $X$ and $\widetilde{X}$ be two tridiagonal matrices of the form
\[
X = \begin{pmatrix}
a_1 & c_1 & & \\
b_1 & a_2 & c_2 & \\
& \ddots & \ddots & c_{n-1} \\
& & b_{n-1} & a_n
\end{pmatrix}, \qquad
\widetilde{X} = \begin{pmatrix}
a_1 & \tilde{c}_1 & & \\
\tilde{b}_1 & a_2 & \tilde{c}_2 & \\
& \ddots & \ddots & \tilde{c}_{n-1} \\
& & \tilde{b}_{n-1} & a_n
\end{pmatrix}.
\]
If $b_i c_i = \tilde{b}_i \tilde{c}_i$ for every $i = 1, \dots, n-1$, then 
$X$ and $\widetilde{X}$ have the same eigenvalues.
\end{lemma}

\begin{proof}
Consider the characteristic polynomials of the leading principal submatrices. 
For $k = 1, \dots, n$, let $P_k(\lambda)$ and $\widetilde{P}_k(\lambda)$ 
denote the characteristic polynomials of the $k\times k$ leading principal 
submatrices of $X$ and $\widetilde{X}$, respectively. Both satisfy the 
three-term recurrence
\[
P_0(\lambda) = 1, \quad P_1(\lambda) = \lambda - a_1, \quad
P_k(\lambda) = (\lambda - a_k)P_{k-1}(\lambda) 
- b_{k-1}c_{k-1}\,P_{k-2}(\lambda), \quad k \geq 2,
\]
and similarly for $\widetilde{P}_k$ with $b_{k-1}c_{k-1}$ replaced by 
$\tilde{b}_{k-1}\tilde{c}_{k-1}$. Since $b_{k-1}c_{k-1} = 
\tilde{b}_{k-1}\tilde{c}_{k-1}$ for all $k$, a straightforward induction 
gives $P_k(\lambda) = \widetilde{P}_k(\lambda)$ for all $k$. In particular, 
$P_n(\lambda) = \widetilde{P}_n(\lambda)$, so $X$ and $\widetilde{X}$ have 
the same characteristic polynomial and hence the same eigenvalues.
\end{proof}
For $k_0 \in E$, we introduce the symmetrized version of the frequency-dependent 
capacitance matrix $\mathcal{C}(k_0)$ defined in~\eqref{eq:mathcalCk_def} by
\begin{align*}
    \left(\mathcal{C}^{\mathrm{sym}}(k_0)\right)_{i,i} 
    &:= \left(\mathcal{C}(k_0)\right)_{i,i}, \\
    \left(\mathcal{C}^{\mathrm{sym}}(k_0)\right)_{i,i+1} 
    = \left(\mathcal{C}^{\mathrm{sym}}(k_0)\right)_{i+1,i} 
    &:= -\sqrt{\left(\mathcal{C}(k_0)\right)_{i,i+1}
               \left(\mathcal{C}(k_0)\right)_{i+1,i}},
\end{align*}
with all remaining entries equal to zero. By Lemma~\ref{thm: tridiag eig}, 
$\mathcal{C}(k_0)$ and $\mathcal{C}^{\mathrm{sym}}(k_0)$ share the same eigenvalues.

For an integer interval $\mathcal{I} = \llbracket a, b \rrbracket$, we define
\[
    \mathrm{Sta}(\mathcal{I}) = \left\lfloor\frac{a}{2}\right\rfloor + 1, \qquad 
    \mathrm{End}(\mathcal{I}) = \left\lceil\frac{b}{2}\right\rceil.
\]
Let $C_j$ (resp.\ $C_j^{\mathrm{sym}}$) denote the principal submatrix of 
$\mathcal{C}(k_0)$ (resp.\ $\mathcal{C}^{\mathrm{sym}}(k_0)$) with rows and columns indexed from $\mathrm{Sta}(\mathcal{I}_j)$ to $\mathrm{End}(\mathcal{I}_j)$, where $\{\mathcal{I}_j\}$ are defined in \eqref{equ:partition}. Similarly, $C_j$ and $C_j^{\mathrm{sym}}$ share the same eigenvalues. The following proposition, illustrated in Figure~\ref{fig:compare}, 
describes the block-diagonal structure of $\mathcal{C}^{\mathrm{sym}}(k_0)$. Theorem \ref{thm: C_j's} follows as a corollary of Proposition \ref{thm: C_j^sym's}, and will be proved accordingly.

\begin{proposition}[(Block-diagonal structure)]\label{thm: C_j^sym's}
$\mathrm{End}(\mathcal{I}_j) < \mathrm{Sta}(\mathcal{I}_{j+1})$ for $j = 1, 
    \ldots, p-1$. Consequently, the submatrices $\{C_j^{\mathrm{sym}}\}_{j=1}^p$ 
    are disjoint principal submatrices of $\mathcal{C}^{\mathrm{sym}}(k_0)$, with 
    all entries outside these blocks equal to zero. Hence, the nonzero eigenvalues 
    of $\mathcal{C}^{\mathrm{sym}}(k_0)$ coincide with the union of the nonzero eigenvalues of the $C_j^{\mathrm{sym}}$'s, counted with their multiplicities. 
\end{proposition}

\begin{proof}[Proof of Proposition \ref{thm: C_j^sym's} and Theorem \ref{thm: C_j's}]
Since $a_{j+1} - b_j \geq 2$, we have
\begin{align*}
    \mathrm{Sta}(\mathcal{I}_{j+1}) - \mathrm{End}(\mathcal{I}_j) 
    &= \left\lfloor\frac{a_{j+1}}{2}\right\rfloor - \left\lceil\frac{b_j}{2}\right\rceil + 1 \\
    &\geq \left\lfloor\frac{b_j + 2}{2}\right\rfloor - \left\lceil\frac{b_j}{2}\right\rceil + 1 
    = \left\lfloor\frac{b_j}{2}\right\rfloor - \left\lceil\frac{b_j}{2}\right\rceil + 2 \geq 1.
\end{align*}
It remains to verify that
\begin{align*}
    \left(\mathcal{C}^{\mathrm{sym}}(k_0)\right)_{i,i} &= 0, 
    \qquad \mathrm{End}(\mathcal{I}_j) + 1 \leq i \leq \mathrm{Sta}(\mathcal{I}_{j+1}) - 1, \\
    \left(\mathcal{C}^{\mathrm{sym}}(k_0)\right)_{i-1,i} &= 0, 
    \qquad \mathrm{End}(\mathcal{I}_j) + 1 \leq i \leq \mathrm{Sta}(\mathcal{I}_{j+1}).
\end{align*}
By definition of $t_i(k_0)$ and $\mathcal{I}_j$, we have $1/t_i(k_0) = 0$ for 
$b_j < i < a_{j+1}$. 

For the diagonal entries, when $\mathrm{End}(\mathcal{I}_j) + 1 \leq i \leq 
\mathrm{Sta}(\mathcal{I}_{j+1}) - 1$, we have
\[
    2i - 1 \in \left[2\left\lceil\frac{b_j}{2}\right\rceil + 1,\, 
    2\left\lfloor\frac{a_{j+1}}{2}\right\rfloor - 1\right] \subset (b_j, a_{j+1}),
\]
so $1/t_{2i-1}(k_0) = 0$ and thus $\left(\mathcal{C}^{\mathrm{sym}}(k_0)\right)_{i,i} = 0$.

For the off-diagonal entries, when $\mathrm{End}(\mathcal{I}_j) + 1 \leq i \leq 
\mathrm{Sta}(\mathcal{I}_{j+1})$, we have
\[
    2i \in \left[2\left\lceil\frac{b_j}{2}\right\rceil + 2,\, 
    2\left\lfloor\frac{a_{j+1}}{2}\right\rfloor + 2\right] \subset (b_j + 1, a_{j+1} + 3),
\]
so at least one of $2i-3$, $2i-2$, $2i-1$ lies in $(b_j, a_{j+1})$, which forces 
at least one of $1/t_{2i-3}(k_0)$, $1/t_{2i-2}(k_0)$, $1/t_{2i-1}(k_0)$ to 
vanish, giving $\left(\mathcal{C}^{\mathrm{sym}}(k_0)\right)_{i-1,i} = 0$. This finally shows that all entries outside $C_j^{\text{sym}},1\leq j\leq p$ equal to zero and thus the nonzero eigenvalues of $\mathcal{C}^{\text{sym}}(k_0)$ coincide with the union of nonzero eigenvalues of the $C^{\text{sym}}_j$'s counted with their multiplicities. Since $\mathcal{C}(k_0)$ and $\mathcal{C}^{\text{sym}}(k_0)$ share the same eigenvalues, and the same holds for $C_j$ and $C^{\text{sym}}_j$, Theorem \ref{thm: C_j's} follows.
\end{proof}
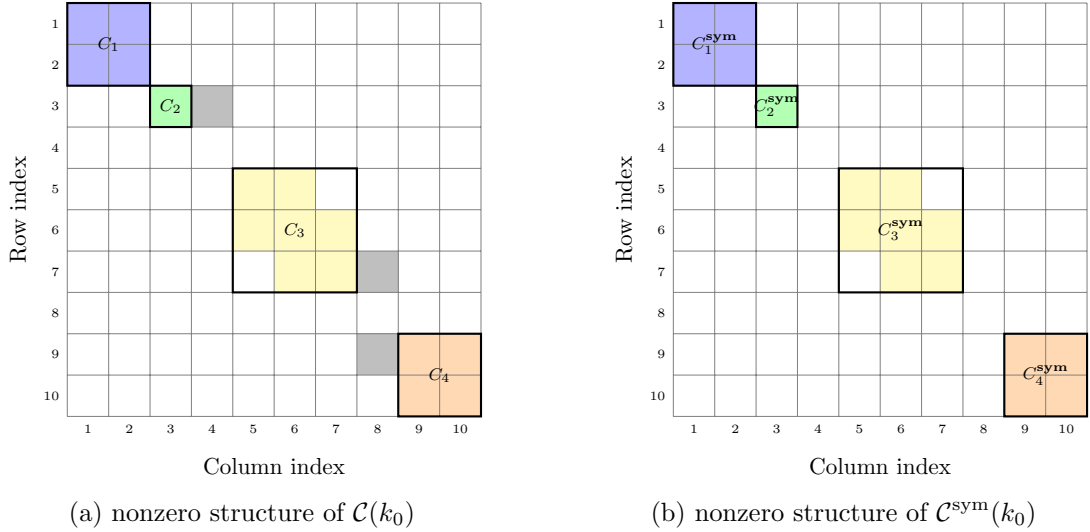
\begin{figure}[!htb]
\centering
\subcaptionbox{nonzero structure of $\mathcal{C}(k_0)$\label{fig:C}}[0.48\textwidth]{%
  \begin{tikzpicture}[scale=0.55, every node/.style={scale=0.8}]
    \colorlet{colorC1}{blue!30}
    \colorlet{colorC2}{green!30}
    \colorlet{colorC3}{yellow!30}
    \colorlet{colorC4}{orange!30}
    \colorlet{colorOther}{gray!50}

    \foreach \r/\c in {1/1,2/2,1/2,2/1} {
        \fill[colorC1] (\c-1,10-\r) rectangle (\c,10-\r+1);
    }
    \fill[colorC2] (2,7) rectangle (3,8);
    \foreach \r/\c in {5/5,6/6,7/7,5/6,6/5,6/7,7/6} {
        \fill[colorC3] (\c-1,10-\r) rectangle (\c,10-\r+1);
    }
    \foreach \r/\c in {9/9,10/10,9/10,10/9} {
        \fill[colorC4] (\c-1,10-\r) rectangle (\c,10-\r+1);
    }
    \foreach \r/\c in {3/4,7/8,9/8} {
        \fill[colorOther] (\c-1,10-\r) rectangle (\c,10-\r+1);
    }

    \draw[step=1cm, gray, very thin] (0,0) grid (10,10);

    \draw[thick] (0,8) rectangle (2,10);
    \draw[thick] (2,7) rectangle (3,8);
    \draw[thick] (4,3) rectangle (7,6);
    \draw[thick] (8,0) rectangle (10,2);

    \node[font=\fontsize{8pt}{12pt}\bfseries] at (1,9) {$C_1$};
    \node[font=\fontsize{8pt}{12pt}\bfseries] at (2.5,7.5) {$C_2$};
    \node[font=\fontsize{8pt}{12pt}\bfseries] at (5.5,4.5) {$C_3$};
    \node[font=\fontsize{8pt}{12pt}\bfseries] at (9,1) {$C_4$};

    \foreach \i in {1,...,10} {
        \node[left, font=\tiny] at (0,10-\i+0.5) {\i};
        \node[below, font=\tiny] at (\i-0.5,0) {\i};
    }

    \node at (5,-1.2) {Column index};
    \node[rotate=90] at (-1.2,5) {Row index};
  \end{tikzpicture}%
}
\hfill
\subcaptionbox{nonzero structure of $\mathcal{C}^{\text{sym}}(k_0)$ \label{fig:Csym}}[0.48\textwidth]{%
  \begin{tikzpicture}[scale=0.55, every node/.style={scale=0.8}]
    \colorlet{colorC1}{blue!30}
    \colorlet{colorC2}{green!30}
    \colorlet{colorC3}{yellow!30}
    \colorlet{colorC4}{orange!30}

    \foreach \r/\c in {1/1,2/2,1/2,2/1} {
        \fill[colorC1] (\c-1,10-\r) rectangle (\c,10-\r+1);
    }
    \fill[colorC2] (2,7) rectangle (3,8);
    \foreach \r/\c in {5/5,6/6,7/7,5/6,6/5,6/7,7/6} {
        \fill[colorC3] (\c-1,10-\r) rectangle (\c,10-\r+1);
    }
    \foreach \r/\c in {9/9,10/10,9/10,10/9} {
        \fill[colorC4] (\c-1,10-\r) rectangle (\c,10-\r+1);
    }

    \draw[step=1cm, gray, very thin] (0,0) grid (10,10);

    \draw[thick] (0,8) rectangle (2,10);
    \draw[thick] (2,7) rectangle (3,8);
    \draw[thick] (4,3) rectangle (7,6);
    \draw[thick] (8,0) rectangle (10,2);

    \node[font=\fontsize{8pt}{12pt}\bfseries] at (1,9) {$C^{\text{sym}}_1$};
    \node[font=\fontsize{8pt}{12pt}\bfseries] at (2.5,7.5) {$C^{\text{sym}}_2$};
    \node[font=\fontsize{8pt}{12pt}\bfseries] at (5.5,4.5) {$C^{\text{sym}}_3$};
    \node[font=\fontsize{8pt}{12pt}\bfseries] at (9,1) {$C^{\text{sym}}_4$};

    \foreach \i in {1,...,10} {
        \node[left, font=\tiny] at (0,10-\i+0.5) {\i};
        \node[below, font=\tiny] at (\i-0.5,0) {\i};
    }

    \node at (5,-1.2) {Column index};
    \node[rotate=90] at (-1.2,5) {Row index};
  \end{tikzpicture}%
}
\caption{An example of the nonzero structure of $\mathcal{C}(k_0)$ (left) and its symmetrized version $\mathcal{C}^{\text{sym}}(k_0)$ (right). Colored squares indicate nonzero entries belonging to the submatrices (different colors distinguish different submatrices), while gray squares represent nonzero entries outside these blocks. In accordance with Proposition \ref{thm: C_j^sym's}, $\mathcal{C}^{\text{sym}}(k_0)$ possesses a block diagonal structure, whereas $\mathcal{C}(k_0)$ does not necessarily.}
\label{fig:compare}
\end{figure}

For each integer interval $\mathcal{I}_j = \llbracket a_j, b_j \rrbracket$, define 
$\xi_j = a_j - 2\lfloor a_j/2 \rfloor$ and $\eta_j = b_j - 2\lfloor b_j/2 \rfloor$. 
In the notation of the preceding subsection,
\begin{align}\label{equ: C_j def}
    C_j = C^{\xi_j,\eta_j}(\theta_{a_j}, \ldots, \theta_{b_j-1}), 
    \qquad j = 1, 2, \ldots, p,
\end{align}
with the decomposition
\begin{align}\label{equ: C_j decomposition}
    C_j = \mathcal{L}_j^{-1} A_j^\top \mathcal{S}_j^{-1} A_j,
\end{align}
where $A_j$, $\mathcal{L}_j$, $\mathcal{S}_j$ are as defined 
in~\eqref{equ: A def} and~\eqref{equ: matrix S T def}. The associated matrix $D_j$ 
is defined by
\begin{align}\label{equ: D_j def}
    D_j = \mathcal{S}_j^{-1} A_j \mathcal{L}_j^{-1} A_j^\top 
    = D^{\xi_j,\eta_j}(\theta_{a_j}, \ldots, \theta_{b_j-1}),
\end{align}
similarly to \eqref{equ: matrix D def}. We now summarize the spectral properties of 
$\mathcal{C}(k_0)$.

\begin{theorem}\label{thm: C(k) nonzero eigenvalue}
Let $m = m(k_0)$ denote the number of nonzero eigenvalues of $\mathcal{C}(k_0)$. 
Then
\[
    m = \sum_{j=1}^p \left\lfloor\frac{n_j}{2}\right\rfloor \leq \left\lfloor\frac{n}{2}\right\rfloor.
\]
Moreover, a nonzero characteristic polynomial of $\mathcal{C}(k_0)$ is
\[
    P_I(\lambda) := \sum_{q=0}^m C_{I,q}(-\lambda)^{m-q},
\]
where
\begin{align}\label{equ: C_I,q def}
    C_{I,q} := \sum_{\substack{0 \leq q_j \leq \lfloor n_j/2 \rfloor \\ q_1 + \cdots + q_p = q}} 
    \prod_{j=1}^p c_{a_j, b_j, q_j} 
    = \sum_{\substack{j_1 \prec j_2 \prec \cdots \prec j_q \\ \{j_i, j_i+1\} \subset I}} 
    \prod_{i=1}^q \theta_{j_i}.
\end{align}
\end{theorem}
\begin{proof}
By Lemma~\ref{thm: eig of C}(ii), the nonzero characteristic polynomial of $C_j$ is
\begin{align}\label{equ: P_j def}
    P_j(\lambda) = \sum_{q=0}^{\lfloor n_j/2 \rfloor} 
    c_{a_j,b_j,q}(-\lambda)^{\lfloor n_j/2 \rfloor - q}.
\end{align}
By Theorem~\ref{thm: C_j's}, $P_I(\lambda) := \prod_{j=1}^p P_j(\lambda)$ is then 
a nonzero characteristic polynomial of $\mathcal{C}(k_0)$, and
\[
    m = \sum_{j=1}^p \left\lfloor\frac{n_j}{2}\right\rfloor 
    \leq \left\lfloor\sum_{j=1}^p \frac{n_j}{2}\right\rfloor 
    = \left\lfloor\frac{n}{2}\right\rfloor.
\]
It remains to verify the second equality in~\eqref{equ: C_I,q def}. Expanding the 
product over blocks and grouping by the number of selected indices in each block 
$\mathcal{I}_j$ gives
\begin{align*}
    \sum_{\substack{j_1 \prec \cdots \prec j_q \\ \{j_i, j_i+1\} \subset I}} 
    \prod_{i=1}^q \theta_{j_i}
    &= \sum_{\substack{q_1 + \cdots + q_p = q \\ 0 \leq q_j \leq \lfloor n_j/2 \rfloor}}
    \prod_{l=1}^p 
    \left(\sum_{a_l \leq j^{(l)}_1 \prec \cdots \prec j^{(l)}_{q_l} \leq b_l - 1} 
    \prod_{i=1}^{q_l} \theta_{j^{(l)}_i}\right)
    = \sum_{\substack{q_1 + \cdots + q_p = q \\ 0 \leq q_j \leq \lfloor n_j/2 \rfloor}}
    \prod_{j=1}^p c_{a_j, b_j, q_j},
\end{align*}
which completes the proof.
\end{proof}

\section{Propagation matrix approach and asymptotic expansions} \label{sec:4}

In this section, we employ a propagation matrix approach to characterize the 
scattering resonances of the one-dimensional system. We construct the total transfer 
matrix for the finite chain of resonators and establish that the resonant frequencies 
coincide precisely with the zeros of a specific analytic function. This characterization 
serves as the foundation for the rigorous asymptotic analysis carried out in 
\Cref{sec:5}.

We introduce the matrices
\begin{align}\label{equ:def LR}
    R(z) := \begin{pmatrix}
        \dfrac{1+z}{2} & \dfrac{1-z}{2} \\[6pt]
        \dfrac{1-z}{2} & \dfrac{1+z}{2}
    \end{pmatrix},
    \qquad
    L(z) := \begin{pmatrix}
        \mathrm{e}^{\mathrm{i} z} & 0 \\
        0 & \mathrm{e}^{-\mathrm{i} z}
    \end{pmatrix}.
\end{align}
We write $L_j(k) := L(t_j k)$ for the matrix associated 
with the $j$\textsuperscript{th} component $t_j$ of $\bm{t}$ defined in~\eqref{equ: bm t def}, and 
set
\[
    \sigma = \frac{\delta}{r}.
\]
The following theorem is a consequence of \cite[Theorem~3.1 and Lemma~3.5]{pm1}.

\begin{theorem}\label{thm: resonant frequency f(z;mu) zeros}
Let $0 \neq \omega = kv$. Then $\omega$ is a resonant frequency if and only if $k$ 
is a zero of $f(k;\sigma) := M_{\mathrm{tot}}(k;\sigma)_{2,2}$, where 
$M_{\mathrm{tot}}(k;\sigma)$ is the analytic matrix function 
\begin{multline}\label{equ: M_total def}
    M_{tot}(k; \sigma) :=\frac{(4\sigma)^{N}}{(1+\sigma)^{2N}}  R\l(\frac{1}{\sigma}\r)L_{2N-1}(k)R(\sigma)L_{2N-2}(k)\cdots\\
    L_4(k)R\l(\frac{1}{\sigma}\r)L_3(k)R(\sigma)L_2(k)R\l(\frac{1}{\sigma}\r)L_1(k)R(\sigma).  
\end{multline}
As $\sigma \to 0$, $f(k;\sigma)$ converges uniformly to
\[
    f(k;0) = -(2\mathrm{i})^{2N-1} \prod_{j=1}^{2N-1} \sin(t_j k)
\]
on every compact set $K \Subset \mathbb{C}$. A point $k \in \mathbb{C}$ is a zero 
of $f(\,\cdot\,;0)$ if and only if $k \in E = \cup_{j=1}^{2N-1}(\pi\mathbb{Z}/t_j)$, 
in which case $k$ is a zero of order
\[
    n(k) := \#\left\{ j : t_j k \in \pi\mathbb{Z},\ 1 \leq j \leq 2N-1 \right\}.
\]
In particular, all zeros of $f(\,\cdot\,;0)$ are real, and $0$ is a zero of order $2N-1$.
\end{theorem}

Following the notation of~\cite[Section~3.3]{pm1}, we introduce
\begin{equation}\label{def:matrixrs}
    \nu := \frac{2\sigma}{1+\sigma} = \frac{2\delta}{\delta+r}, \quad
    R := \begin{pmatrix} -1 & 1 \\ -1 & 1 \end{pmatrix}, \quad
    R_\pm := \begin{pmatrix} 1 & \pm1 \\ \pm1 & 1 \end{pmatrix}, \quad
    S := \begin{pmatrix} 0 & -1 \\ 1 & 0 \end{pmatrix}.
\end{equation}
Define
\begin{equation}\label{equ: G(k;delta) expand}
    G(k;\nu) := (R+\nu S)L_{2N-1}(k)(R+\nu S)L_{2N-2}(k)(R+\nu S)
    \cdots L_1(k)(R+\nu S),
\end{equation}
and set $g(k;\nu) := G(k;\nu)_{2,2}$. Then from \cite{pm1} we have  
\begin{equation}\label{equ:g=f}
    \begin{aligned}
P M_{t o t}(k ; \sigma) P=G(k ; \nu),  \quad f(k ; \sigma)=g(k ; \nu),
\end{aligned}
\end{equation}
with $P=\diag\{-1,1\}$. Therefore, $\omega=kv$ is a resonant frequency if and only if $k$ is a zero of $g(k;\nu)$. Since $G(k;\nu)$ is a polynomial in $\nu$ of 
degree $2N$, we expand
\begin{equation}\label{equ:G_expession}
    G(k;\nu) = \sum_{l=0}^{2N} G_l(k)\,\nu^l, \qquad
    g(k;\nu) = \sum_{l=0}^{2N} g_l(k)\,\nu^l,
\end{equation}
where $g_l(k) = G_l(k)_{2,2}$.

Next, extending \cite[Proposition 4.2]{pm1}, we derive the asymptotic behavior of $G_{l}(k)$ when $k\to k_0$ for a fixed $k_0\in E$. The proof is postponed to Appendix \ref{appex:proofexpG} for ease of reading. 


\begin{theorem}\label{thm: G_{I,l} expand}
Define the convention that $\cot(t_0k_0)=\cot(t_{2N}k_0):=-\i$. Let $z=k-k_0$, $n$ be the cardinality of set $I$ in (\ref{def:Ik0}), and $m$ be the number of nonzero eigenvalues of $\mathcal{C}(k_0)$. For $1\leq l\leq m$, $G_l(k_0+z)$ admits the following expansions: as $z\to 0$,
\begin{equation}\label{expGl}
\begin{aligned}
       G_l(k_0+z)=&C_1\cdot 2^l(-2\i )^{2N-1-2l}z^{n-2l}\left[(1+C_2z) \sum_{\substack{j_1\prec j_2\prec \cdots\prec j_l\\\{j_i,j_i+1\}\subset I}}\left(\prod_{i=1}^{l}\theta_{j_i}\right)\cdot R\right.\\
       &+z\sum_{j=1}^p\frac{\cot(t_{a_j-1}k_0)}{t_{a_j}}\sum_{\substack{j_1\prec j_2\prec\cdots\prec j_{l-1}\\\{j_i,j_i+1\}\subset I\setminus\{a_j\}}}\left(\prod_{i=1}^{l-1}\theta_{j_i}\right)\cdot R_{a_j}\\
       &\left.+z\sum_{j=1}^p\frac{\cot(t_{b_j+1}k_0)}{t_{b_j}}\sum_{\substack{j_1\prec j_2\prec\cdots\prec j_{l-1}\\\{j_i,j_i+1\}\subset I\setminus\{b_j\}}}\left(\prod_{i=1}^{l-1}\theta_{j_i}\right)\cdot R_{b_j}\right]+\mc{O}(z^{n-2l+2}),
\end{aligned}
\end{equation}
where
\begin{align*}
C_1=(-1)^{\sum_{j\in I}m_j}\prod_{j\notin I}\sin (t_jk_0)\prod_{j\in I}t_j,\ C_2=\sum_{j\notin I}t_j\cot(t_jk_0),\ R_j=\begin{cases}
    R_+,&j=1,\\
    R_-,&j=2N-1,\\
    R,&1<j<2N-1,
\end{cases}
\end{align*}
with $m_j$ defined by (\ref{equ:defmj}). Consequently, for $1\leq l\leq m$,
\begin{equation*}\label{expgl}
\begin{aligned}
       g_l(k_0+z)=&C_1\cdot 2^l(-2\i )^{2N-1-2l}z^{n-2l}\left[(1+C_2z)\sum_{\substack{j_1\prec j_2\prec \cdots\prec j_l\\\{j_i,j_i+1\}\subset I}}\left(\prod_{i=1}^{l}\theta_{j_i}\right)\right.\\
       &+z\sum_{j=1}^p\frac{\cot(t_{a_j-1}k_0)}{t_{a_j}}\sum_{\substack{j_1\prec j_2\prec\cdots\prec j_{l-1}\\\{j_i,j_i+1\}\subset I\setminus\{a_j\}}}\left(\prod_{i=1}^{l-1}\theta_{j_i}\right)\\
       &\left.+z\sum_{j=1}^p\frac{\cot(t_{b_j+1}k_0)}{t_{b_j}}\sum_{\substack{j_1\prec j_2\prec\cdots\prec j_{l-1}\\\{j_i,j_i+1\}\subset I\setminus\{b_j\}}}\left(\prod_{i=1}^{l-1}\theta_{j_i}\right)\right]+\mc{O}(z^{n-2l+2}).
\end{aligned}
\end{equation*}
In addition,
\begin{align}\label{equ: g_0 expand}
g_0(k_0+z)=C_1\cdot(-2\i)^{2N-1}z^n\l(1+C_2z\r)+\mc{O}(z^{n+2}).
\end{align}
\end{theorem}

\section{Newton polygon method and the proof of (\ref{equ: omega^i,pm expand})} \label{sec:5}

In this section, we determine the asymptotic behavior of the scattering resonances 
by analyzing the zeros of the analytic function $g(k;\nu)$ characterized in 
\Cref{sec:4}. In \Cref{subsec:newtonresona}, we apply the Newton polygon method 
to classify these zeros into two types according to their asymptotic order in $\nu$. 
In \Cref{subsec:asymexpo}, we derive a rigorous and more precise asymptotic 
expansion for the $2m$ branches of order $k_0 + \mathcal{O}(\nu^{1/2})$, 
culminating in the proof of Theorem~\ref{thm:mainresult1}.

\subsection{Newton polygon and asymptotic analysis of resonances} \label{subsec:newtonresona}

We investigate the asymptotic behavior of the zeros of $g(k;\nu)$ as $\nu \to 0$. 
By Theorem~\ref{thm: resonant frequency f(z;mu) zeros} and (\ref{equ:g=f}), the limiting function is
\[
    g(k;0) = -(2\mathrm{i})^{2N-1} \prod_{j=1}^{2N-1} \sin(t_j k).
\]
For $k_0 \in E$, Rouch\'{e}'s theorem guarantees exactly $n = n(k_0)$ zeros of 
$g(\,\cdot\,;\nu)$ in a small neighborhood of $k_0$ for $\nu$ sufficiently small. 
We aim to determine the asymptotic order of each of these $n$ zeros. To this end, recall $n_j$ in \eqref{def:n_j}, and partition $\{1, \ldots, p\}$ into
\begin{equation}\label{equ:defT}
    T^e = \{j : n_j \text{ is even}\}, \qquad T^o = \{j : n_j \text{ is odd}\},
\end{equation}
and define, for each $j$,
\begin{equation}\label{equ:def_fj}
    f_j(l) = \begin{cases}
        n_j - 2l, & 0 \leq l \leq \left\lfloor n_j/2 \right\rfloor, \\[4pt]
        \dfrac{n_j+1}{2} - l, & \left\lfloor n_j/2 \right\rfloor < l \leq \left\lceil n_j/2 \right\rceil
        \quad (j \in T^o \text{ only}).
    \end{cases}
\end{equation}

The following lemma for the summation of $f_j(l_j)$'s will be useful. 

\begin{lemma}\label{thm: order lower boundary}
For an integer $0 \leq l \leq n-m$, consider the optimization problem
\begin{equation*}\label{equ:defSl}
    S(l) := \min_{\substack{l_1, \ldots, l_p \in \mathbb{Z}_{\geq 0} \\ 
    l_1 + \cdots + l_p = l}} \sum_{j=1}^p f_j(l_j).
\end{equation*}
Then
\begin{equation}\label{equ:Slformula}
    S(l) = \begin{cases}
        n - 2l, & 0 \leq l \leq m, \\
        n - m - l, & m < l \leq n-m.
    \end{cases}
\end{equation}
When $0 \leq l \leq m$, the minimum is attained if and only if $l_j \leq 
\lfloor n_j/2 \rfloor$ for all $1 \leq j \leq p$. When $m < l \leq n-m$, 
the minimum is attained if and only if $l_j = n_j/2$ for $j \in T^e$, exactly 
$l - m$ indices $j \in T^o$ satisfy $l_j = (n_j+1)/2$, and the remaining 
$j \in T^o$ satisfy $l_j = (n_j-1)/2$.
\end{lemma}

\begin{proof}
\textit{Case $0 \leq l \leq m$.}
For each $j$, $f_j(l_j) \geq n_j - 2l_j$, with equality if and only if 
$l_j \leq \lfloor n_j/2 \rfloor$. Summing over $j$ and using $\sum_j l_j = l$,
\[
    \sum_{j=1}^p f_j(l_j) \geq \sum_{j=1}^p (n_j - 2l_j) = n - 2l,
\]
with equality if and only if $l_j \leq \lfloor n_j/2 \rfloor$ for all $j$.

\textit{Case $m < l \leq n-m$.}
For each $j$, $f_j(l_j) \geq \lceil n_j/2 \rceil - l_j$, with equality if and 
only if $\lfloor n_j/2 \rfloor \leq l_j \leq \lceil n_j/2 \rceil$. Similarly, we have 
\[
    \sum_{j=1}^p f_j(l_j) 
    \geq \sum_{j=1}^p \left(\left\lceil\frac{n_j}{2}\right\rceil - l_j\right)
    = \sum_{j=1}^p \left(n_j - \left\lfloor\frac{n_j}{2}\right\rfloor - l_j\right)
    = n - m - l.
\]
The equality holds if and only if $l_j = n_j/2$ for $j \in T^e$, exactly $l - m$ 
indices $j \in T^o$ satisfy $l_j = (n_j+1)/2$, and the remaining $j \in T^o$ satisfy 
$l_j = (n_j-1)/2$.
\end{proof}

Our analysis for the asymptotics of resonances is mainly based on the expansion of $G(k, \nu)$ in (\ref{equ: G(k;delta) expand}) by matrices $L, R, S$. In particular, the asymptotics of $G(k, \nu)$ in the resonant intervals $\mathcal{I}_j = \llbracket a_j, b_j \rrbracket$'s (red interval in Figure \ref{fig:interval_matrix}) and the non-resonant intervals $\llbracket b_{j-1}+1, a_j-1 \rrbracket$'s (blue interval in Figure \ref{fig:interval_matrix}) are different. So we will treat them separately in the following discussions. To this end, for resonant intervals $\mathcal{I}_j = \llbracket a_j, b_j \rrbracket \subset \llbracket 1, 2N-1 \rrbracket$, 
we define the resonant factor 
\begin{equation}\label{equ: G_j def}
    G_j(k;\nu) := (R+\nu S)L_{b_j}(k)(R+\nu S)L_{b_j-1}(k)(R+\nu S)
    \cdots L_{a_j}(k)(R+\nu S),
\end{equation}
and for the non-resonant interval $\llbracket b_{j-1}+1, a_j-1 \rrbracket$, we define the non-resonant factors $\widetilde{G}_j$ by expansion:
\begin{multline}\label{equ: tilde G_j def}
     \widetilde{G}_j(k;\nu) 
     = L_{a_j-1}(k)(R+\nu S)L_{a_j-2}(k) 
    \cdots L_{b_{j-1}+2}(k)(R+\nu S)L_{b_{j-1}+1}(k), 
    \ j = 2,\ldots,p, 
\end{multline}
and 
\begin{align*}
    \widetilde{G}_1(k;\nu) 
    &= L_{a_1-1}(k)(R+\nu S)\cdots L_1(k)(R+\nu S), \\
    \widetilde{G}_{p+1}(k;\nu) 
    &= (R+\nu S)L_{2N-1}(k)\cdots(R+\nu S)L_{b_p+1}(k),
\end{align*}
for the start and end intervals, respectively. If $a_1 = 1$ (resp.\ $b_p = 2N-1$), then $\widetilde{G}_1$ (resp.\ 
$\widetilde{G}_{p+1}$) is understood as the identity matrix.

We first give the following proposition for the asymptotics of $G_j(k, \nu)$ as $k\rightarrow 0$. Its proof follows 
the same argument as~\cite[Proposition~4.2]{pm1}; although~\cite{pm1} treats only 
intervals of odd length, the even-length case is completely analogous and is omitted.

\begin{proposition}\label{prop: asymptotic Gj}
$G_j(k;\nu)$ is a polynomial in $\nu$ of degree $n_j+1$:
\begin{equation}\label{equ: G_j expansion}
    G_j(k;\nu) = \sum_{l=0}^{n_j+1} G_{j,l}(k)\,\nu^l,
\end{equation}
whose coefficients have the following asymptotic behavior as $k \to 0$:
\begin{itemize}
    \item For $l = 0$,
    \[
        G_{j,0}(k) = \prod_{s=a_j}^{b_j} t_s \cdot (-2\mathrm{i}k)^{n_j} \cdot R 
        + \mathcal{O}(k^{n_j+2}).
    \]
    \item For $1 \leq l \leq \lfloor n_j/2 \rfloor$,
    \[
        G_{j,l}(k) = \left(\prod_{s=a_j}^{b_j} t_s\right) 
        2^l(-2\mathrm{i}k)^{n_j-2l} c_{a_j,b_j,l}\, R 
        + \mathcal{O}(k^{n_j-2l+1}).
    \]
    \item For $l = \lceil n_j/2 \rceil$ with $n_j$ odd,
    \[
        G_{j,\lceil n_j/2 \rceil}(k) = 2^{\lceil n_j/2 \rceil} I + \mathcal{O}(k).
    \]
\end{itemize}
In particular, $G_{j,l}(k) = \mathcal{O}(k^{f_j(l)})$ as $k \to 0$ with $f_j(l)$ defined by (\ref{equ:def_fj}), for 
$0 \leq l \leq \lceil n_j/2 \rceil$.
\end{proposition}

With these tools in place, we establish the following theorem for the resonant frequencies $k(\nu)$ and conclude this subsection. A more precise 
asymptotic expansion for the $2m$ branches of $k_0+\mc{O}(\nu^{1/2})$ is given in 
\Cref{subsec:asymexpo}.

\begin{theorem}\label{thm: asy of zeros}
Let $m=m(k_0)$ be the number of nonzero eigenvalues of $\mathcal{C}(k_0)$. As $\nu \to 0$, the $n = n(k_0)$ zeros of $g(k;\nu)$ near $k_0$ split into two 
types: 
\begin{itemize}
    \item[$\bullet$] $2m$ branches of zeros $k(\nu)$ of the form
    \[
    k(\nu)=k_0+\mc{O}(\nu^{1/2});
    \]
    \item[$\bullet$] $n-2m$ branches of zeros $k(\nu)$ of the form
    \[
    k(\nu)=k_0+\mc{O}(\nu).
    \]
\end{itemize}
\end{theorem}

\begin{proof}
We decompose $G(k;\nu)$ in (\ref{equ: G(k;delta) expand}) into resonant factors $G_j$ and 
non-resonant factors $\widetilde{G}_j$:
\begin{align}
    G(k;\nu) = & \widetilde{G}_{p+1}(k;\nu)\,G_p(k;\nu)\cdots 
    \widetilde{G}_2(k;\nu)\,G_1(k;\nu)\,\widetilde{G}_1(k;\nu) \nonumber \\
    =& \widetilde{G}_{p+1}(k_0+z;\nu)\,G_p(k_0+z;\nu)\cdots \widetilde{G}_2(k_0+z;\nu)\,G_1(k_0+z;\nu)\,\widetilde{G}_1(k_0+z;\nu), \label{equ:expanG1}
\end{align}
where $z = k - k_0$. We first estimate the resonant factor $G_j(k_0+z;\nu)$. By the definition of $L$ in (\ref{equ:def LR}), we have $L(z_1+z_2) = L(z_1)L(z_2)$ and $(-1)^{m_j}L_j(k_0)$ is an identity matrix for $m_j$ in (\ref{equ:defmj}), $j \in I$. This yields 
\begin{align*}
    &G_j(k_0+z;\nu) \\
    =& (R+\nu S)L_{b_j}(k_0+z)(R+\nu S)L_{b_j-1}(k_0+z)(R+\nu S)
    \cdots L_{a_j}(k_0+z)(R+\nu S)\\
    =& (R+\nu S)L_{b_j}(k_0)L_{b_j}(z)(R+\nu S)L_{b_j-1}(k_0)L_{b_j-1}(z)(R+\nu S)
    \cdots L_{a_j}(k_0)L_{a_j}(z)(R+\nu S)\\ 
    =& (-1)^{\sum_{s=a_j}^{b_j} m_s}(R+\nu S)L_{b_j}(z)(R+\nu S)L_{b_j-1}(z)(R+\nu S)
    \cdots L_{a_j}(z)(R+\nu S)\\ 
    =&(-1)^{\sum_{s=a_j}^{b_j} m_s} G_j(z;\nu).
\end{align*}
According to Proposition \ref{prop: asymptotic Gj}, the factors $G_j(z; \nu) $'s admit the expansion 
\[
G_j(z; \nu) = \sum_{l=0}^{n_j+1} G_{j,l}(z)\,\nu^l
\]
with $G_{j,l}(z) = \mathcal{O}(z^{f_j(l)})$ as $z \to 0$ with $f_j(l)$ defined by (\ref{equ:def_fj}). 

For the non-resonant factor $\widetilde{G}_j(k;\nu)$, by expanding it in $\nu$, we obtain
\[
    \widetilde{G}_j(k;\nu) = \sum_{l=0}^{c_j} \widetilde{G}_{j,l}(k)\,\nu^l,
\]
for some integer $c_j$, where
\begin{align*}
    \widetilde{G}_{j,0}(k) 
    &= L_{a_j-1}(k)\,R\,L_{a_j-2}(k)\cdots L_{b_{j-1}+2}(k)\,R\,
    L_{b_{j-1}}(k), \quad j = 2,\ldots,p, \\
    \widetilde{G}_{1,0}(k) 
    &= L_{a_1-1}(k)\,R\cdots L_1(k)\,R, \\
    \widetilde{G}_{p+1,0}(k) 
    &= R\,L_{2N-1}(k)\cdots R\,L_{b_p+1}(k).
\end{align*}
Therefore, expanding $G(k;\nu)$ in (\ref{equ:expanG1}) as $\sum_{l=0}^{2N} G_l(k)\nu^l$, we obtain
\begin{multline}\label{equ: G_l expansion NR}
    G_l(k) = (-1)^{\sum_{j \in I} m_j} 
    \sum_{\tilde{l}_1 + \cdots + \tilde{l}_{p+1} + l_1 + \cdots + l_p = l}
    \widetilde{G}_{p+1,\tilde{l}_{p+1}}(k_0+z)\,G_{p,l_p}(z)
  \\  \cdots \widetilde{G}_{2,\tilde{l}_2}(k_0+z)\,G_{1,l_1}(z)\,
    \widetilde{G}_{1,\tilde{l}_1}(k_0+z).
\end{multline}

Based on the expansion above, we employ the Newton polygon method (see Appendix A in \cite{pm1}) to derive the asymptotics of \(k(\nu)\). The key idea is to expand \(g(k_0+z; \nu)\) as \(\sum_{l=0}^{2N} \sum_{q=c(l)}^{\infty} g_{ql} z^{q} \nu^{l}\), where the summation over \(q\) runs over all feasible indices, \emph{i.e.}, those with \(g_{q,l} \neq 0\); the lower boundary of all feasible \((q,l)\)—the Newton polygon—then determines the asymptotic behavior of \(z(\nu)\).

To proceed, let \(\tau_0(f)\) denote the order of the leading term in the asymptotic expansion of \(f(z)\) as \(z\to 0\). The first identity in \eqref{equ:matrixexpan2} yields \(\tau_0(\widetilde{G}_{j,0}(k_0+z)) = 0\) for \(j = 0,1,\dots,p+1\). Moreover, by Proposition~\ref{prop: asymptotic Gj}, Lemma~\ref{thm: order lower boundary}, and the fact that \(S(l)\) in \eqref{equ:Slformula} is decreasing, we have
\begin{equation}\label{equ: tau_0(Q) LB}
\begin{aligned}
&\tau_0\Bigl(\widetilde{G}_{p+1,\tilde{l}_{p+1}}(k_0+z)\,G_{p,l_p}(z)\cdots \widetilde{G}_{2,\tilde{l}_2}(k_0+z)\,G_{1,l_1}(z)\,\widetilde{G}_{1,\tilde{l}_1}(k_0+z)\Bigr) \\
&= \sum_{j=1}^{p+1} \tau_0\!\left(\widetilde{G}_{j,\tilde{l}_j}\right) + \sum_{j=1}^{p} \tau_0\!\left(G_{j,l_j}\right) \ge \sum_{j=1}^{p} f_j(l_j) \ge S\!\left(l - \sum_{j=1}^{p+1} \tilde{l}_j\right) \ge S(l),
\end{aligned}
\end{equation}
with equality if and only if \(\tilde{l}_1 = \cdots = \tilde{l}_{p+1} = 0\) and \(l_1,\dots,l_p\) satisfy the conditions of Lemma~\ref{thm: order lower boundary}. Consequently, from the expansion \eqref{equ: G_l expansion NR}, \(G(k_0+z;\nu)\) can be further expressed as \(\sum_{l=0}^{2N} \sum_{q=S(l)}^{\infty} \widehat{G}_{ql} z^{q} \nu^{l}\). Since \(g(k;\nu) = G(k;\nu)_{2,2}\) and the matrix multiplications involving \(L, R, S\) do not cause \((\widehat{G}_{ql})_{2,2}\) to vanish, we obtain the expansion
\[
g(k_0+z; \nu) = \sum_{l=0}^{2N} \sum_{q=S(l)}^{\infty} g_{ql} z^{q} \nu^{l},
\]
where the coefficients \(g_{ql}\) are non‑zero only if the \(q\geq S(l)\). The expression for \(S(l)\) in \eqref{equ:Slformula} therefore determines the Newton polygon near \((z,\nu) = (0,0)\), as illustrated in Figure~\ref{fig:lower_boundary_k0_general}.
\begin{itemize}
    \item 
For \(m < l \le n-m\), we have \(S(l) = n - m - l\). Hence \(g(k_0+z;\nu) = \sum_{l=0}^{2N} \sum_{q=n-m-l}^{\infty} g_{ql} z^{q} \nu^{l}\). The corresponding edge \(l_1\) in Figure~\ref{fig:lower_boundary_k0_general} (slope \(-1\)) yields \(n-2m\) zero branches satisfying \(z(\nu) = \mathcal{O}(\nu)\).
    \item
For \(0 \le l \le m\), we have \(S(l) = n - 2l\). Hence \(g(k_0+z;\nu) = \sum_{l=0}^{2N} \sum_{q=n-2l}^{\infty} g_{ql} z^{q} \nu^{l}\). The corresponding edge \(l_2\) in Figure~\ref{fig:lower_boundary_k0_general} (slope \(-1/2\)) yields \(2m\) zero branches satisfying \(z(\nu) = \mathcal{O}(\nu^{1/2})\).
\end{itemize}
This completes the proof.

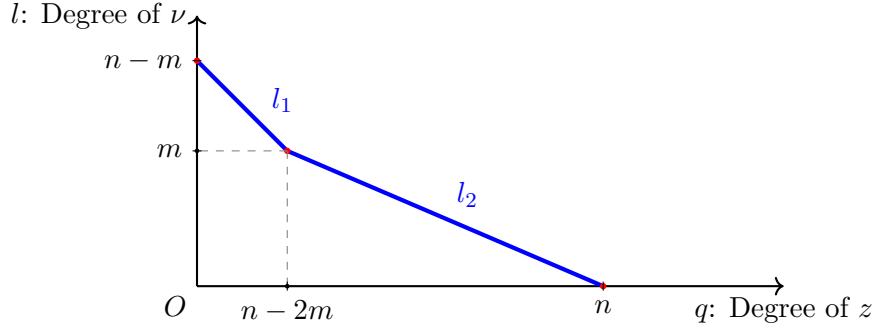
\begin{figure}[!htb]
\centering
\begin{tikzpicture}[scale=0.6]
    \def\N{5}
    \pgfmathsetmacro{\xmax}{2*\N-1}
    \pgfmathsetmacro{\yA}{\N}
    \pgfmathsetmacro{\yB}{\N-2}
    \pgfmathsetmacro{\yC}{0}
    \draw[->, thick] (0,0) -- (0,\N+1) node[left] {$l$: Degree of $\nu$};
    \draw[->, thick] (0,0) -- (\xmax+4,0) node[below] {$q$: Degree of $z$};
    \node at (0,0) [below left] {$O$};
    \draw[ultra thick, blue] (0,\yA) -- node[pos=0.7, above right] {$l_1$} (2,\yB);
    \draw[ultra thick, blue] (2,\yB) -- node[pos=0.5, above right] {$l_2$} (\xmax,\yC);
    \fill[red] (0,\yA) circle (2pt);
    \fill[red] (2,\yB) circle (2pt);
    \fill[red] (\xmax,\yC) circle (2pt);
    \draw[dashed, gray] (2,\yB) -- (2,0);
    \draw[dashed, gray] (2,\yB) -- (0,\yB);
    \draw (0.1,\yA) -- (-0.1,\yA) node[left] {$n-m$};
    \draw (0.1,\yB) -- (-0.1,\yB) node[left] {$m$};
    \draw (2,0.1) -- (2,-0.1) node[below] {$n-2m$};
    \draw (\xmax,0.1) -- (\xmax,-0.1) node[below] {$n$};
    \fill[black] (2,0) circle (1.5pt);
    \fill[black] (0,\yB) circle (1.5pt);
\end{tikzpicture}
\caption{Newton polygon showing the lower boundary of the convex hull near 
$(z,\nu) = (0,0)$. The edges $l_1$ and $l_2$ have slopes $-1$ and $-1/2$, respectively.}
\label{fig:lower_boundary_k0_general}
\end{figure}
\end{proof}

\subsection{{Higher-order expansions of $k_0+\mc{O}(\nu^{1/2})$}} \label{subsec:asymexpo}

We now turn to deriving a more precise asymptotic expansion for the $2m$ branches of zeros of the form $k(\nu) = k_0+\mc{O}(\nu^{1/2})$. In Theorem \ref{thm: G_{I,l} expand}, we have derived the asymptotic form of $g_l(k),0\leq l\leq m$. To proceed, we first need to derive the asymptotic form of $g_{m+1}(k)$. For $j\in T^o$ in (\ref{equ:defT}), let
\[
T^o_j:=t_{a_j}+t_{a_j+2}+t_{a_j+4}+\cdots+t_{b_j-2}+t_{b_j}.
\]
Then we have

\begin{lemma}\label{thm: g_m+1 expand}
Take the convention $\cot(t_0 k_0) = \cot(t_{2N} k_0) := 
-\mathrm{i}$. Let $z=k-k_0$, $n$ be the cardinality of set $I$ in (\ref{def:Ik0}), and $m$ be the number of nonzero eigenvalues of $\mathcal{C}(k_0)$. If $\#T^o = n - 2m > 0$, then as $z \to 0$, 
\begin{align*}
g_{m+1}(k_0+z)=&(-1)^{\sum_{j\in I}m_j}\prod_{j\notin I}\sin(t_jk_0)\prod_{j\in T^o}T^o_j \cdot 2^{m+1}(-2\i)^{2N-2m-3}
\\&z^{n-2m-1}\sum_{j\in T^o}\frac{1}{T^o_j}\l[\cot(t_{a_{j}-1}k_0)+\cot(t_{b_{j}+1}k_0)\r]+\mc{O}(z^{n-2m})
\end{align*}
with $m_j$ defined by (\ref{equ:defmj}).
\end{lemma}

\begin{proof}
We adopt the notation from the proof of Theorem \ref{thm: asy of zeros}. In (\ref{equ: G_l expansion NR}) and (\ref{equ: tau_0(Q) LB}), we have shown that $\tau_0(G_{m+1}(k_0+z))=n-2m+1$. From the equality condition in Lemma \ref{thm: order lower boundary} and the expansion in Proposition \ref{prop: asymptotic Gj}, we obtain
\begin{align*}
G_{m+1}(k_0+z)=&(-1)^{\sum_{j\in I}m_j}\sum_{j\in T^o}\ti{G}_{p+1,0}(k_0)G_{p,\l\lf n_p/2\r\rf}(z)\cdots G_{j+1,\l\lf n_{j+1}/2\r\rf}(z)\ti{G}_{j+1,0}(k_0)\\
&G_{j,\l\lc n_j/2\r\rc}(z)\ti{G}_{j,0}(k_0)G_{j-1,\l\lf n_{j-1}/2\r\rf}(z)\cdots G_{1,\l\lf n_1/2\r\rf}(z)\ti{G}_{1,0}(k_0)\\
=&(-1)^{\sum_{j\in I}m_j}\prod_{j\notin I}\sin(t_jk_0)\prod_{j\in T^o}T^o_j \cdot 2^{m+1}(-2\i)^{2N-2m-3}
\\&z^{n-2m-1}\sum_{j\in T^o}\frac{1}{T^o_j}\l[\cot(t_{a_{j}-1}k_0+\cot(t_{b_{j}+1}k_0)\r]\cdot R_j+\mc{O}(z^{n-2m}),
\end{align*}
where $R_j$ is defined as $R_+$ for $j=1$, $R_-$ for $j=p$, and $R$ for $1<j<p$; see (\ref{def:matrixrs}). The expansion of $g_{m+1}(k_0+z)$ is immediately derived using $g_{m+1}(k_0+z)=G_{m+1}(k_0+z)_{2,2}$.
\end{proof}

The following theorem is the precise version of Theorem~\ref{thm:mainresult1}, 
stated in terms of $\nu$; Theorem~\ref{thm:mainresult1} follows immediately via 
$\nu = 2\delta/(\delta+r)$, $\omega = kv$, and Theorem~\ref{thm: resonant 
frequency f(z;mu) zeros}.

\begin{theorem}\label{thm: k(nu) expand}
The $2m$ branches are analytic in $\nu^{1/2}$ and satisfy
\[
    k_j^\pm(\nu) = k_0 \pm \sqrt{\frac{\lambda_j}{2}\,\nu} + \mathcal{O}(\nu),
    \qquad j = 1, 2, \ldots, m,
\]
where $\lambda_1, \ldots, \lambda_m$ are the $m$ nonzero eigenvalues of 
$\mathcal{C}(k_0)$. More precisely, let $\lambda$ be a nonzero eigenvalue of 
$\mathcal{C}(k_0)$ with multiplicity $\mathfrak{m}$. By Theorem~\ref{thm: C_j's}, 
$\lambda$ is an eigenvalue of exactly $\mathfrak{m}$ submatrices of $\mathcal{C}(k_0)$,
denoted $C_{j_1}, \ldots, C_{j_{\mathfrak{m}}}$. The $2\mathfrak{m}$ branches 
satisfying $k^\pm(\nu) = k_0 \pm \sqrt{\lambda\nu/2} + \mathcal{O}(\nu)$ admit the 
refined expansion
\begin{align}\label{equ: k^i,pm expand}
    k^{i,\pm}(\nu) = k_0 \pm \sqrt{\frac{\lambda}{2}\,\nu} 
    + \frac{\al_{j_i}^\top B_{j_i}\al_{j_i}}
           {4\,\al_{j_i}^\top \mathcal{L}_{j_i}\al_{j_i}}\,\nu 
    + \mathcal{O}(\nu^{3/2}), \qquad i = 1, 2, \ldots, \mathfrak{m},
\end{align}
where $(\lambda, \al_{j_i})$ is an eigenpair of $C_{j_i}$,
\[
    B_j = \operatorname{diag}\!\left\{
        \cot(t_{a_j-1}k_0)\,\chi_{a_j},\; 0, \ldots, 0,\;
        \cot(t_{b_j+1}k_0)\,\chi_{b_j}
    \right\},
\]
with the convention $\cot(t_j k_0) := -\mathrm{i}$ for $j = 0$ or $j = 2N$. 
Here $\mathcal{L}_j$ is defined in~\eqref{equ: C_j decomposition}, $a_j, b_j$ are 
the endpoints of $\mathcal{I}_j$, and
\[
    \chi_j = \begin{cases}
        1, & j \text{ odd}, \\[4pt]
        \dfrac{1}{\lambda t_j^2}, & j \text{ even}.
    \end{cases}
\]
\end{theorem}

\begin{proof}
Expanding $$g(k_0+z;\nu) = \sum_{j=0}^\infty \sum_{l=0}^{2N} c_{jl}\,z^j\nu^l,$$
the Newton polygon in Figure~\ref{fig:lower_boundary_k0_general} shows that 
$c_{jl} = 0$ whenever $(j,l)$ lies strictly below the line segments $l_1$ and 
$l_2$. With the notation of Theorem~\ref{thm: C(k) nonzero eigenvalue}, 
Expansion (\ref{equ:G_expession}) and Theorem~\ref{thm: G_{I,l} expand} give
\[
    c_{n-2l,\,l} = C_1 \cdot 2^l(-2\mathrm{i})^{2N-1-2l} C_{I,l}, 
    \quad 0 \leq l \leq m, \qquad
    c_{n+1,\,0} = C_1 C_2 (-2\mathrm{i})^{2N-1},
\]
and, for $1 \leq l \leq m$,
\begin{multline*}
     c_{n-2l+1,\,l} = C_1 \cdot 2^l(-2\mathrm{i})^{2N-1-2l}
    \bigg[C_2 C_{I,l} \\ + \sum_{j=1}^p \left(
        \frac{\cot(t_{a_j-1}k_0)}{t_{a_j}} C_{I\setminus\{a_j\},\,l-1}
        + \frac{\cot(t_{b_j+1}k_0)}{t_{b_j}} C_{I\setminus\{b_j\},\,l-1}
    \right)\bigg],
\end{multline*}
where $C_{I,l}$ is defined by (\ref{equ: C_I,q def}). If $\#T^o > 0$, Lemma~\ref{thm: g_m+1 expand} yields
\begin{multline}\label{equ: c_n-2m-1,m+1}
    c_{n-2m-1,\,m+1} = (-1)^{\sum_{j\in I}m_j}
    \prod_{j\notin I}\sin(t_jk_0) \\ \prod_{j\in T^o}T^o_j
    \cdot 2^{m+1}(-2\mathrm{i})^{2N-2m-3}
    \sum_{j\in T^o}\frac{\cot(t_{a_j-1}k_0)+\cot(t_{b_j+1}k_0)}{T^o_j}.
\end{multline}
If $\#T^o = 0$, the edge $l_1$ degenerates and the proof is analogous; we therefore restrict to $\#T^o > 0$. The proof will be divided into three steps. 

\medskip
\noindent
\textbf{Step 1: Leading-order expansion.}
The edge $l_2$ gives $2m$ zeros $z \sim c\nu^{1/2}$ with $c \neq 0$. The balance 
equation is
\[
    \Phi(c) := \sum_{(j,l)\in l_2\cap\mathbb{Z}^2} c_{jl}\,c^j = 0.
\]
Recalling $P_I(\lambda)$ from Theorem~\ref{thm: C(k) nonzero eigenvalue}, a direct 
computation gives
\begin{equation}\label{equ: Phi(c) def}
    \Phi(c) = \sum_{l=0}^m c_{n-2l,\,l}\,c^{n-2l}
    = C_1\, 2^m(-2\mathrm{i})^{2N-1-2m} \cdot c^{n-2m} P_I(2c^2).
\end{equation}
The nonzero roots of $\Phi(c)$ are therefore
\[
    c_j^\pm = \pm\sqrt{\frac{\lambda_j}{2}}, \qquad 1 \leq j \leq m,
\]
and the $2m$ zeros of $g(k;\nu)$ satisfy
\[
    k_j^\pm(\nu) = k_0 \pm \sqrt{\frac{\lambda_j}{2}\,\nu} + o(\nu^{1/2}),
    \qquad 1 \leq j \leq m.
\]

\noindent
\textbf{Step 2: Higher-order expansion, simple case $\mathfrak{m} = 1$.}
Here $\lambda$ is a simple eigenvalue of a unique submatrix $C_{j_1}$ of 
$\mathcal{C}(k_0)$. Set
\[
    h^\pm\!\left(\alpha;\nu^{1/2}\right)
    = \nu^{-n/2}\,g\!\left(k_0 + \nu^{1/2}\!\left(\pm\sqrt{\tfrac{\lambda}{2}}
    + \alpha\right);\nu\right),
    \qquad
    \widetilde{\Phi}(c) := \sum_{l=0}^{m+1} c_{n-2l+1,\,l}\,c^{n-2l+1}.
\]
Since $\Phi(\pm\sqrt{\lambda/2}) = 0$, expanding $h^{\pm}$ gives
\[
    h^\pm\!\left(\alpha;\nu^{1/2}\right)
    = \Phi'\!\left(\pm\sqrt{\tfrac{\lambda}{2}}\right)\alpha
    + \widetilde{\Phi}\!\left(\pm\sqrt{\tfrac{\lambda}{2}}\right)\nu^{1/2}
    + \sum_{j+l\geq 2} d^\pm_{jl}\,\alpha^j\nu^{l/2}.
\]
The function $h^\pm(\alpha;\nu^{1/2})$ is analytic in $\alpha$ and $\nu^{1/2}$, 
with $h^\pm(0;0) = 0$. Since $\lambda$ is simple, \eqref{equ: Phi(c) def} gives
\[
    \frac{\partial h^\pm}{\partial\alpha}(0;0)
    = \Phi'\!\left(\pm\sqrt{\tfrac{\lambda}{2}}\right)
    = C_1 \cdot 2^{m+2}(-2\mathrm{i})^{2N-1-2m}
    \left(\pm\sqrt{\tfrac{\lambda}{2}}\right)^{n-2m+1} P_I'(\lambda) \neq 0.
\]
By the implicit function theorem, there exists a unique analytic function 
$\alpha_\pm(\nu^{1/2})$ near $\nu = 0$ such that 
$h^\pm(\alpha_\pm(\nu^{1/2});\nu^{1/2}) = 0$ and $\alpha_\pm(0) = 0$. To find 
$\alpha_\pm'(0)$, we use the nonzero characteristic polynomials 
$P_{I\setminus\{a_j\}}(\lambda)$ from Theorem~\ref{thm: C(k) nonzero eigenvalue}. 
Their degrees satisfy
\[
    \deg P_{I\setminus\{a_j\}} = \begin{cases} m-1, & j \in T^e, \\ m, & j \in T^o, \end{cases}
\]
and for $j \in T^o$, \eqref{equ: C_I,q def} gives
\[
    C_{I\setminus\{a_j\},\,m} = \frac{t_{a_j}}{T^o_j}
    \cdot \frac{\prod_{s\in T^o} T^o_s}{\prod_{s\in I} t_s},
\]
with an analogous formula for $C_{I\setminus\{b_j\},m}$. Combining with 
\eqref{equ: c_n-2m-1,m+1},
\[
    c_{n-2m+1,\,m+1} = C_1 \cdot 2^{m+1}(-2\mathrm{i})^{2N-2m-3}
    \sum_{j\in T^o}\!\left(
        \frac{\cot(t_{a_j-1}k_0)}{t_{a_j}} C_{I\setminus\{a_j\},\,m}
        + \frac{\cot(t_{b_j+1}k_0)}{t_{b_j}} C_{I\setminus\{b_j\},\,m}
    \right).
\]
Collecting all coefficients $c_{n-2l+1,\,l}$ for $0 \leq l \leq m+1$,
\begin{align*}
    \widetilde{\Phi}(c)
    &= C_1 C_2 \cdot c\,\Phi(c) \\
    &\quad + C_1\, 2^{m+1}(-2\mathrm{i})^{2N-3-2m} \cdot c^{n-1-2m}
    \sum_{j\in T^o}\!\left(
        \frac{\cot(t_{a_j-1}k_0)}{t_{a_j}} P_{I\setminus\{a_j\}}(2c^2)
        + \frac{\cot(t_{b_j+1}k_0)}{t_{b_j}} P_{I\setminus\{b_j\}}(2c^2)
    \right) \\
    &\quad + C_1\, 2^m(-2\mathrm{i})^{2N-1-2m} \cdot c^{n+1-2m}
    \sum_{j\in T^e}\!\left(
        \frac{\cot(t_{a_j-1}k_0)}{t_{a_j}} P_{I\setminus\{a_j\}}(2c^2)
        + \frac{\cot(t_{b_j+1}k_0)}{t_{b_j}} P_{I\setminus\{b_j\}}(2c^2)
    \right).
\end{align*}
Define, for $j = 1, \ldots, p$,
\[
    \widetilde{C}_j = C^{1-\xi_j,\eta_j}(\theta_{a_j+1},\ldots,\theta_{b_j-1}),
    \qquad
    \widetilde{\widetilde{C}}_j=C^{\xi_j,1-\eta_j}(\theta_{a_j},\cdots,\theta_{b_j-2}),
\]
and 
\begin{equation*}
       \widetilde{P}_j(\lambda) = \sum_{q=0}^{\lfloor(n_j-1)/2\rfloor}
    c_{a_j+1,b_j,q}(-\lambda)^{\lfloor(n_j-1)/2\rfloor-q},
    \quad\widetilde{\widetilde{P}}_j(\lambda) = \sum_{q=0}^{\lfloor(n_j-1)/2\rfloor}
    c_{a_j,b_j-1,q}(-\lambda)^{\lfloor(n_j-1)/2\rfloor-q}.
\end{equation*}
Since $\lambda$ is a simple eigenvalue of $C_{j_1}$ and not of $C_j$ for $j \neq j_1$,
\[
    P_{I\setminus\{a_j\}}(\lambda) = \begin{cases}
        0, & j \neq j_1, \\
        \prod_{j\neq j_1} P_j(\lambda)\cdot\widetilde{P}_{j_1}(\lambda), & j = j_1,
    \end{cases}
    \qquad
    P_I'(\lambda) = \prod_{j\neq j_1} P_j(\lambda)\cdot P_{j_1}'(\lambda),
\]
and analogously for $P_{I\setminus\{b_j\}}(\lambda)$. Using $\Phi(\pm\sqrt{\lambda/2})=0$,
\begin{multline*}
     \widetilde{\Phi}\!\left(\pm\sqrt{\tfrac{\lambda}{2}}\right) 
    = C_1\,2^{m+2}(-2\mathrm{i})^{2N-3-2m} \times
   \\ \left(\pm\sqrt{\tfrac{\lambda}{2}}\right)^{n+1-2m}
    \widetilde{\chi}_{j_1}\left(
        \frac{\cot(t_{a_{j_1}-1}k_0)}{t_{a_{j_1}}}P_{I\setminus\{a_{j_1}\}}(\lambda)
        + \frac{\cot(t_{b_{j_1}+1}k_0)}{t_{b_{j_1}}}P_{I\setminus\{b_{j_1}\}}(\lambda)
    \right),
\end{multline*}
where $\widetilde{\chi}_j = 1/\lambda$ for $j\in T$ and $\widetilde{\chi}_j = -1$ 
for $j\in S$. Applying the implicit function theorem together with 
Lemma~\ref{thm: eig of C}(iii), 
\begin{align*}
       \alpha_\pm'(0)
    & = -\frac{\partial_{\nu^{1/2}} h^\pm(0;0)}{\partial_\alpha h^\pm(0;0)}
    = \frac{\widetilde{\chi}_{j_1}}{4P_I'(\lambda)}\left(
        \frac{\cot(t_{a_{j_1}-1}k_0)}{t_{a_{j_1}}} P_{I\setminus\{a_{j_1}\}}(\lambda)
        + \frac{\cot(t_{b_{j_1}+1}k_0)}{t_{b_{j_1}}} P_{I\setminus\{b_{j_1}\}}(\lambda)
    \right) \\
    & = \frac{\al_{j_1}^\top B_{j_1}\al_{j_1}}
           {4\,\al_{j_1}^\top \mathcal{L}_{j_1}\al_{j_1}}.
\end{align*}
Hence,
\[
    k^{1,\pm}(\nu) = k_0 \pm \sqrt{\frac{\lambda}{2}\,\nu}
    + \frac{\al_{j_1}^\top B_{j_1}\al_{j_1}}
           {4\,\al_{j_1}^\top \mathcal{L}_{j_1}\al_{j_1}}\,\nu
    + \mathcal{O}(\nu^{3/2})
\]
satisfies $g(k^{1,\pm}(\nu);\nu) = 0$ near $\nu = 0$ and is analytic in $\nu^{1/2}$.

\medskip

\noindent
\textbf{Step 3: General case $\mathfrak{m} \geq 2$.}
By Lemma~\ref{thm: eig of C}(i), $\lambda$ is a simple eigenvalue of exactly 
$\mathfrak{m}$ submatrices $C_{j_1}, \ldots, C_{j_\mathfrak{m}}$ with 
$j_1 < \cdots < j_\mathfrak{m}$. Write $\mathcal{C}(k_0, \bm{t})$ for 
$\mathcal{C}(k_0)$ to make the dependence on $\bm{t}$ explicit. Fix $\mathfrak{m}$ 
distinct real numbers $\beta_1, \ldots, \beta_\mathfrak{m}$ and define the perturbed 
vector $\bm{t}(\varepsilon) = (t_1(\varepsilon), \ldots, t_{2N-1}(\varepsilon))^\top$ by
\[
    t_s(\varepsilon) = \begin{cases}
        \dfrac{t_s}{1+\varepsilon\beta_i}, 
        & \text{if } s \in \mathcal{I}_{j_i} \text{ for some } 1 \leq i \leq \mathfrak{m}, \\[6pt]
        t_s, & \text{otherwise},
    \end{cases}
\]
and set $g(k;\nu,\varepsilon) = G(k;\nu,\varepsilon)_{2,2}$, where $G(k;\nu,\varepsilon)$ 
is defined as in \eqref{equ: G(k;delta) expand} with $t_j$ replaced by $t_j(\varepsilon)$.  
For $\varepsilon > 0$ sufficiently small, 
$I_i(\varepsilon) := \{j : \pi \mid t_j(\varepsilon)k_0(1+\varepsilon\beta_i)\}$ 
equals $\mathcal{I}_{j_i}$ for each $i$. By Theorem~\ref{thm: C_j's}, 
$\mathcal{C}(k_0(1+\varepsilon\beta_i), \bm{t}(\varepsilon))$ has the same nonzero 
eigenvalues as $\widetilde{C}_{j_i}(\varepsilon)$, where 
$\widetilde{C}_{j_i}(\varepsilon)$ denotes the principal submatrix of 
$\mathcal{C}(k_0(1+\varepsilon\beta_i), \bm{t}(\varepsilon))$ with rows and columns 
indexed from $\mathrm{Sta}(\mathcal{I}_{j_i})$ to $\mathrm{End}(\mathcal{I}_{j_i})$. 
A direct calculation gives
\[
    \widetilde{C}_{j_i}(\varepsilon) = (1+\varepsilon\beta_i)^2 C_{j_i},
    \qquad 1 \leq i \leq \mathfrak{m}.
\]
Thus, $(1+\varepsilon\beta_i)^2\lambda$ is a simple nonzero eigenvalue of $\widetilde{C}_{j_i}(\varepsilon)$ (and also $\mathcal{C}(k_0(1+\varepsilon\beta_i), \bm{t}(\varepsilon))$). Applying Step~2, we have 
\[
    k^{i,\pm}(\nu,\varepsilon) = (1+\varepsilon\beta_i)\left[
        k_0 \pm \sqrt{\frac{\lambda}{2}\,\nu}
        + \frac{\al_{j_i}^\top B_{j_i}(\varepsilon)\al_{j_i}}
               {4\,\al_{j_i}^\top \mathcal{L}_{j_i}\al_{j_i}}\,\nu
    \right] + \mathcal{O}(\nu^{3/2}), \qquad 1 \leq i \leq \mathfrak{m},
\]
with $B_j(\varepsilon) \to B_j$ as $\varepsilon \to 0$, satisfies 
$g(k^{i,\pm}(\nu,\varepsilon);\nu,\varepsilon) = 0$ near $(\nu,\varepsilon) = (0,0)$. 
Since $g(\,\cdot\,;\nu,\varepsilon) \rightrightarrows g(\,\cdot\,;0)$ on compact sets 
as $(\nu,\varepsilon) \to (0,0)$, Rouch\'{e}'s theorem guarantees uniform boundedness 
of $\{k^{i,\pm}(\nu,\varepsilon)\}_{\varepsilon>0}$. Montel's theorem then yields a 
sequence $\varepsilon_n \to 0$ along which $k^{i,\pm}(\nu,\varepsilon_n)$ converges 
uniformly to a function $k^{i,\pm}(\nu)$ analytic in $\nu^{1/2}$. The claim below identifies the coefficients of $k^{i,\pm}(\nu)$, completing the proof.

\noindent \textbf{Claim:}
Let $f_n(u,v)$ and $f(u,v)$ be continuous on a region in $\mathbb{C}^2$ with 
$f_n \rightrightarrows f$ (\emph{i.e.}, uniform convergence). Let $u_n(v)$ and $u(v)$ be analytic near $v = 0$ with 
$f_n(u_n(v),v) \equiv 0$ and $u_n(v) \rightrightarrows u(v)$. Then $f(u(v),v) \equiv 0$. 
Moreover, letting 
\begin{equation*}
    u_n(v)=\sum_{m=0}^\infty c_{nm}v^m,\qquad u(v)=\sum_{m=0}^\infty c_m v^m,
\end{equation*}
we have $\lim_{n\to\infty} c_{nm} = c_m$ for every $m$.

\medskip 

\noindent \textit{Proof of Claim.} 
The first assertion follows from continuity and uniform convergence. For the 
coefficients, fix $\rho > 0$ small enough so that all $u_n$ and $u$ are analytic on 
$\{|\zeta| \leq \rho\}$; the conclusion follows from Cauchy's integral formula and 
uniform convergence on $|\zeta| = \rho$.
\end{proof}

\begin{remark} 
Theorems~\ref{thm: asy of zeros} and~\ref{thm: k(nu) expand} extend to the case 
where the wave speeds inside the resonators are not identical. Suppose that the ratio $r_i$ between the wave speed outside and inside the $i$\textsuperscript{th} resonator is positive. 
We replace definition~\eqref{equ: bm t def} by two vectors:
\[
    \mathfrak{\bm t}=(\mathfrak{t}_1, \cdots, \mathfrak{t}_{2N-1})^\top:=(r_1\ell_1,\ s_{1},\ r_2\ell_{2},\ s_{2},\ \cdots,\ r_{N-1}\ell_{N-1},\ s_{N-1},\ r_N\ell_N)^\top\in\R^{2N-1}_{>0},
\]
\[
     \bm t=(t_1, \cdots, t_{2N-1})^\top:=(r_1^2\ell_1,\ s_{1},\ r_2^2\ell_{2},\ s_{2},\ \cdots,\ r_{N-1}^2\ell_{N-1},\ s_{N-1},\ r_N^2\ell_N)^\top\in\R^{2N-1}_{>0}.
\]
and modify the definitions of $t_j(k)$ in~\eqref{equ: t_j(k) def} and $E$ 
in~\eqref{def:E} to
\[
    t_j(k) = \begin{cases}
        t_j, & \text{if } \pi \mid \mathfrak{t}_j k, \\
        \infty, & \text{otherwise},
    \end{cases}
    \qquad
    E = \bigcup_{j=1}^{2N-1} \frac{\pi}{\mathfrak{t}_j}\mathbb{Z}.
\]
With these modifications, the nonzero eigenvalues $\lambda_1, \ldots, \lambda_m$ of 
the corresponding frequency-dependent capacitance matrix $\mathcal{C}(k_0)$ govern 
the first-order approximations of the resonant frequencies. Specifically, the 
scattering problem has exactly $n$ resonant frequencies near $k_0 v$ for $k_0 \in E$, 
of which the first $2m$ satisfy
\begin{equation}\label{equ:differentriconclu1}
    \omega_j^\pm(\delta) = k_0 v \pm v\sqrt{\delta\lambda_j} + \mathcal{O}(\delta),
    \qquad j = 1, \ldots, m,
\end{equation}
with a higher-order expansion followed by the same argument as in 
Theorem~\ref{thm: k(nu) expand}, while the remaining $n - 2m$ satisfy
\begin{equation}\label{equ:differentriconclu2}
    \omega_j(\delta) = k_0 v + \mathcal{O}(\delta), \qquad j = 1, \ldots, n-2m.
\end{equation}
Since $\mathcal{C}(k_0)$ can be symmetrized, it is always diagonalizable when the 
$r_i$ are real and positive. If the $r_i$ are allowed to be complex, $\mathcal{C}(k_0)$ 
need not be diagonalizable; nevertheless, conclusions~\eqref{equ:differentriconclu1} 
and~\eqref{equ:differentriconclu2} remain valid, with a proof identical to that of 
Theorems~\ref{thm: asy of zeros} and~\ref{thm: k(nu) expand}. We refer 
to~\cite{alex} for a further discussion.
\end{remark}


\section{Characterization of eigenmodes} \label{sec:6}

In this section, we study the eigenmodes beyond the subwavelength regime, that is, 
for $k_0 \neq 0$; the subwavelength case has been treated in detail 
in~\cite{feppon.cheng.ea2023Subwavelength, pm1}. Without loss of generality, we 
set $r = 1$ throughout.

Roughly speaking, we will show that the eigenmode $u(x)$ corresponding to a resonant frequency 
$k(\delta) = k_0 + \mathcal{O}(\delta^{1/2})$ is approximated, to leading order $\mathcal{O}(\delta^{1/2})$, by trigonometric functions on specific spacings between resonators and vanishes 
elsewhere. The amplitudes of these  trigonometric functions are determined by the eigenvector of the matrix $D$ \eqref{equ: matrix D def} 
introduced in \Cref{sec: block matrices}. We derive these results using a propagation matrix approach; throughout, we write $'$ in place of $d/dx$ to denote 
differentiation with respect to $x$.

Assume that $u(x)$ is a solution to the equation $u''+k^2 u=0$ with $k\ne0$, then
\begin{align}\label{equ: P(k,a)}
\begin{pmatrix}
    u(x+a)\\u'(x+a)
\end{pmatrix}=P(k,a)\begin{pmatrix}
    u(x)\\u'(x)
\end{pmatrix},\forall x,a\in\R, \q P(k,a) = \begin{pmatrix}
    \cos{ka}&\frac{1}{k}\sin{ka}\\
    -k\sin{ka}&\cos{ka}
\end{pmatrix}.
\end{align}
Assume that $k(\delta)=k_0\pm\sqrt{\lambda\delta}+\mc{O}(\delta)$ is a resonant frequency. If $\pi\nmid k_0a$, then
\begin{align}\label{equ: P(k,a) expand non}
P(k,a)=P(k_0,a)+\mc{O}(\delta^{1/2}).    
\end{align}
Otherwise, assume that $k_0a=\mathfrak{m}\pi$ for some $\mathfrak{m}\in\Z$. Then
\begin{align}\label{equ: P(k,a) expand res}
(-1)^\mathfrak{m}P(k,a)=I\pm\sqrt{\lambda}a\begin{pmatrix}
    &\frac{1}{k_0}\\-k_0
\end{pmatrix}\delta^{1/2}+\mc{O}(\delta).    
\end{align}

\subsection{Simple eigenvalue case}
We investigate the case where $\lambda$ is a simple eigenvalue of $\mathcal{C}(k_0)$. Recalling the setup of Theorem~\ref{thm: eigenmode simple case}, $\lambda$ is a 
simple eigenvalue of a unique submatrix $C_{j^*}$ and is not an eigenvalue of  $C_j$ for any $j \neq j^*$. The $s$ even integers in $\mathcal{I}_{j^*}$ are 
$2\lceil a_{j^*}/2\rceil,\, 2\lceil a_{j^*}/2\rceil + 2,\, \ldots,\, 
2\lfloor b_{j^*}/2\rfloor$. We also recall that $j_0 = 2\lceil a_{j^*}/2\rceil - 2$, and $C_{j^*}$, $D_{j^*}$ are defined in~\eqref{equ: C_j decomposition} and~\eqref{equ: D_j def}, respectively. 

We are now ready to prove \Cref{thm: eigenmode simple case}.

\begin{proof}[Proof of Theorem \ref{thm: eigenmode simple case}.]
To avoid a cumbersome case-by-case analysis of endpoint parities of the intervals 
$\mathcal{I}_j = \llbracket a_j, b_j \rrbracket$, we present the proof for the case 
where all $(a_j, b_j)$ have parity (even, even); the remaining cases follow 
analogously. 

Set $b_0 = 0$, $a_{p+1} = 2N$, and write $a_j = 2\hat{a}_j$, $b_j = 2\hat{b}_j$ with 
integers satisfying
\[
    0 = \hat{b}_0 < \hat{a}_1 < \hat{b}_1 < \hat{a}_2 < \hat{b}_2 < \cdots 
    < \hat{a}_p < \hat{b}_p < \hat{a}_{p+1} = N.
\]
We use $a_j, b_j$ and $2\hat{a}_j, 2\hat{b}_j$ interchangeably. We analyze $u(x)$ on 
the intervals $[x_{b_j+1}, x_{a_{j+1}}]$ ($0 \leq j \leq p$) and 
$[x_{a_j}, x_{b_j+1}]$ ($1 \leq j \leq p$) by considering the following five 
regions (see Figure~\ref{fig:regions}):
\begin{enumerate} 
    \setlength{\itemsep}{0pt}
    \item Non-resonant intervals on the left: $[x_{b_j+1}, x_{a_{j+1}}]$, 
    $0 \leq j < j^*$;
    \item Resonant intervals on the left: $[x_{a_j}, x_{b_j+1}]$, 
    $1 \leq j < j^*$;
    \item The target resonant interval: $[x_{a_{j^*}}, x_{b_{j^*}+1}]$;
    \item Non-resonant intervals on the right: $[x_{b_j+1}, x_{a_{j+1}}]$, 
    $j^* \leq j \leq p$;
    \item Resonant intervals on the right: $[x_{a_j}, x_{b_j+1}]$, 
    $j^* < j \leq p$.
\end{enumerate}

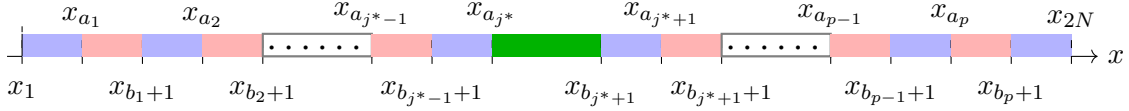
\begin{figure}[!htb]
\centering
\begin{tikzpicture}
    \def\w{0.8}   
    \def\wl{1.45} 

    \def\x{0}

    \draw[->] (-0.2,0) -- (14.3,0) node[right] {$x$};

    \draw[dashed] (0,-0.1) -- (0,0.4);
    \node[below] at (0,-0.2) {$x_1$};

    \setcounter{pt}{1}

    \fill[blue!30] (\x,0) rectangle ++(\w,0.3);
    \draw (\x,0) -- (\x,0.3);
    \pgfmathsetmacro\x{\x+\w}
    \stepcounter{pt}
    \draw[dashed] (\x,-0.1) -- (\x,0.3);
    \ifodd\value{pt}
        \node[below] at (\x,-0.2) {$x_{1}$};
    \else
        \node[above] at (\x,0.25) {$x_{a_1}$};
    \fi

    \fill[red!30] (\x,0) rectangle ++(\w,0.3);
    \pgfmathsetmacro\x{\x+\w}
    \stepcounter{pt}
    \draw[dashed] (\x,-0.1) -- (\x,0.3);
    \ifodd\value{pt}
        \node[below] at (\x,-0.2) {$x_{b_1+1}$};
    \else
        \node[above] at (\x,0.25) {$ $};
    \fi

    \fill[blue!30] (\x,0) rectangle ++(\w,0.3);
    \pgfmathsetmacro\x{\x+\w}
    \stepcounter{pt}
    \draw[dashed] (\x,-0.1) -- (\x,0.3);
    \ifodd\value{pt}
        \node[below] at (\x,-0.2) {$ $};
    \else
        \node[above] at (\x,0.25) {$x_{a_2}$};
    \fi

    \fill[red!30] (\x,0) rectangle ++(\w,0.3);
    \pgfmathsetmacro\x{\x+\w}
    \stepcounter{pt}
    \draw[dashed] (\x,-0.1) -- (\x,0.3);
    \ifodd\value{pt}
        \node[below] at (\x,-0.2) {$x_{b_2+1}$};
    \else
        \node[above] at (\x,0.25) {$ $};
    \fi

    \draw[gray, thick] (\x,0) rectangle ++(\wl,0.3);
    \node at (\x + \wl/2, 0.1) {\Large $\cdots\cdots$}; 
    \pgfmathsetmacro\x{\x+\wl}
    \stepcounter{pt}
    \draw[dashed] (\x,-0.1) -- (\x,0.3);
    \ifodd\value{pt}
        \node[below] at (\x,-0.2) {$x_{a_{j^*-1}}$};
    \else
        \node[above] at (\x,0.25) {$x_{a_{j^*-1}}$};
    \fi

    \fill[red!30] (\x,0) rectangle ++(\w,0.3);
    \pgfmathsetmacro\x{\x+\w}
    \stepcounter{pt}
    \draw[dashed] (\x,-0.1) -- (\x,0.3);
    \ifodd\value{pt}
        \node[below] at (\x,-0.2) {$x_{b_{j^*-1}+1}$};
    \else
        \node[above] at (\x,0.25) {$x_{b_{j^*-1}+1}$};
    \fi

    \fill[blue!30] (\x,0) rectangle ++(\w,0.3);
    \pgfmathsetmacro\x{\x+\w}
    \stepcounter{pt}
    \draw[dashed] (\x,-0.1) -- (\x,0.3);
    \ifodd\value{pt}
        \node[below] at (\x,-0.2) {$x_{a_{j^*}}$};
    \else
        \node[above] at (\x,0.25) {$x_{a_{j^*}}$};
    \fi

    \fill[green!70!black] (\x,0) rectangle ++(\wl,0.3);
    \pgfmathsetmacro\x{\x+\wl}
    \stepcounter{pt}
    \draw[dashed] (\x,-0.1) -- (\x,0.3);
    \ifodd\value{pt}
        \node[below] at (\x,-0.2) {$x_{b_{j^*+1}}$};
    \else
        \node[above] at (\x,0.25) {$x_{b_{j^*+1}}$};
    \fi

    \fill[blue!30] (\x,0) rectangle ++(\w,0.3);
    \pgfmathsetmacro\x{\x+\w}
    \stepcounter{pt}
    \draw[dashed] (\x,-0.1) -- (\x,0.3);
    \ifodd\value{pt}
        \node[below] at (\x,-0.2) {$x_{a_{j^*+1}}$};
    \else
        \node[above] at (\x,0.25) {$x_{a_{j^*+1}}$};
    \fi

    \fill[red!30] (\x,0) rectangle ++(\w,0.3);
    \pgfmathsetmacro\x{\x+\w}
    \stepcounter{pt}
    \draw[dashed] (\x,-0.1) -- (\x,0.3);
    \ifodd\value{pt}
        \node[below] at (\x,-0.2) {$x_{b_{j^*+1}+1}$};
    \else
        \node[above] at (\x,0.25) {$x_{b_{j^*+1}+1}$};
    \fi

    \draw[gray, thick] (\x,0) rectangle ++(\wl,0.3);
    \node at (\x + \wl/2, 0.1) {\Large $\cdots\cdots$}; 
    \pgfmathsetmacro\x{\x+\wl}
    \stepcounter{pt}
    \draw[dashed] (\x,-0.1) -- (\x,0.3);
    \ifodd\value{pt}
        \node[below] at (\x,-0.2) {$x_{a_{p-1}}$};
    \else
        \node[above] at (\x,0.25) {$x_{a_{p-1}}$};
    \fi

    \fill[red!30] (\x,0) rectangle ++(\w,0.3);
    \pgfmathsetmacro\x{\x+\w}
    \stepcounter{pt}
    \draw[dashed] (\x,-0.1) -- (\x,0.3);
    \ifodd\value{pt}
        \node[below] at (\x,-0.2) {$x_{b_{p-1}+1}$};
    \else
        \node[above] at (\x,0.25) {$x_{b_{p-1}+1}$};
    \fi

    \fill[blue!30] (\x,0) rectangle ++(\w,0.3);
    \pgfmathsetmacro\x{\x+\w}
    \stepcounter{pt}
    \draw[dashed] (\x,-0.1) -- (\x,0.3);
    \ifodd\value{pt}
        \node[below] at (\x,-0.2) {$x_{a_p}$};
    \else
        \node[above] at (\x,0.25) {$x_{a_p}$};
    \fi

    \fill[red!30] (\x,0) rectangle ++(\w,0.3);
    \pgfmathsetmacro\x{\x+\w}
    \stepcounter{pt}
    \draw[dashed] (\x,-0.1) -- (\x,0.3);
    \ifodd\value{pt}
        \node[below] at (\x,-0.2) {$x_{b_p+1}$};
    \else
        \node[above] at (\x,0.25) {$x_{b_p+1}$};
    \fi

    \fill[blue!30] (\x,0) rectangle ++(\w,0.3);
    \pgfmathsetmacro\x{\x+\w}
    \stepcounter{pt}
    \draw[dashed] (\x,-0.1) -- (\x,0.3);
    \ifodd\value{pt}
        \node[below] at (\x,-0.2) {$x_{2N}$};
    \else
        \node[above] at (\x,0.25) {$x_{2N}$};
    \fi

\end{tikzpicture}
\caption{Schematic illustration of non-resonant intervals (blue), resonant 
intervals (red), and the target resonant interval (green), with $j < j^*$ on 
the left and $j > j^*$ on the right.}
\label{fig:regions}
\end{figure}

Throughout Steps~1--6, we work to leading order in $\delta$ and write $a\delta^\gamma$ 
to denote a quantity with asymptotic expansion $a\delta^\gamma + \mathcal{O}(\delta^{\gamma+1/2})$. 

\medskip
\noindent\textbf{Step 1: Non-resonant propagation from $x_1$ to $x_{a_1}$.}
For $x < x_1$, up to a constant factor, $u(x) = \delta^{-1} \e^{-\mathrm{i}kx}$; 
we fix this expression and choose the constant factor later. The initial data at 
$x_1$ are
\[
    (u(x_1), u'_-(x_1))^\top = (\alpha_1\delta^{-1}, \beta_1\delta^{-1})^\top,
    \qquad \alpha_1 = \e^{-\mathrm{i}k_0x_1},\quad \beta_1 = -\mathrm{i}k_0 \e^{-\mathrm{i}k_0x_1}.
\]
For $1 \leq i < a_1$, we have $\pi \nmid k_0 t_i$, so~\eqref{equ: P(k,a) expand non} 
applies. The jump relation (see the fourth and fifth equations in (\ref{equ: scattering problem})) at $x_1$ gives 
$(u(x_1), u'_+(x_1))^\top = (\alpha_1\delta^{-1}, \beta_1)^\top$, and applying 
\eqref{equ: P(k,a) expand non} yields
\[
    \begin{pmatrix} u(x_2) \\ u'_-(x_2) \end{pmatrix}
    = P(k, t_1) \begin{pmatrix} u(x_1) \\ u'_+(x_1) \end{pmatrix}
    = \begin{pmatrix} \alpha_2\delta^{-1} \\ \beta_2\delta^{-1} \end{pmatrix},
    \qquad
    \begin{cases} \alpha_2 = \alpha_1\cos(k_0 t_1), \\ \beta_2 = -\alpha_1 k_0\sin(k_0 t_1). \end{cases}
\]
Continuing this alternating application of jump relations and~\eqref{equ: P(k,a) 
expand non}, we obtain the general pattern
\[
    \begin{pmatrix} u(x_{2i-1}) \\ u'_+(x_{2i-1}) \end{pmatrix}
    = \begin{pmatrix} \alpha_{2i-1}\delta^{-i} \\ \beta_{2i-1}\delta^{-i+1} \end{pmatrix},
    \qquad
    \begin{pmatrix} u(x_{2i}) \\ u'_+(x_{2i}) \end{pmatrix}
    = \begin{pmatrix} \alpha_{2i}\delta^{-i} \\ \beta_{2i}\delta^{-i-1} \end{pmatrix},
\]
where the sequences $(\alpha_j)_{1 \leq j \leq 2\hat{a}_1}$ and 
$(\beta_j)_{1 \leq j \leq 2\hat{a}_1}$ satisfy
\[
    \begin{cases}
        \alpha_{2i-1} = \beta_{2i-2}/k_0 \cdot \sin(k_0 t_{2i-2}), \\
        \beta_{2i-1} = \beta_{2i-2} k_0 \cos(k_0 t_{2i-2}),
    \end{cases}
    \qquad
    \begin{cases}
        \alpha_{2i} = \alpha_{2i-1}\cos(k_0 t_{2i-1}), \\
        \beta_{2i} = -\alpha_{2i-1} k_0\sin(k_0 t_{2i-1}).
    \end{cases}
\]
In particular,
\[
    \begin{pmatrix} u(x_{a_1}) \\ u'_+(x_{a_1}) \end{pmatrix}
    = \begin{pmatrix} \alpha_{2\hat{a}_1}\delta^{-\hat{a}_1} \\ 
    \beta_{2\hat{a}_1}\delta^{-\hat{a}_1-1} \end{pmatrix},
    \qquad
    \beta_{2\hat{a}_1} = \alpha_1 k_0(-1)^{\hat{a}_1}\prod_{i=1}^{2\hat{a}_1-1}\sin(k_0 t_i) \neq 0.
\]

\medskip
\noindent\textbf{Step 2: Resonant propagation on $[x_{a_1}, x_{b_1+1}]$.}
For $a_1 \leq i \leq b_1$, we have $k_0 t_i = m_i\pi$ with $m_i \in \mathbb{Z}$, 
so~\eqref{equ: P(k,a) expand res} applies. Applying it with the jump relation gives
\[
    \begin{pmatrix} u(x_{2\hat{a}_1+1}) \\ u'_+(x_{2\hat{a}_1+1}) \end{pmatrix}
    = (-1)^{m_{2\hat{a}_1}}
    \begin{pmatrix} \alpha_{2\hat{a}_1+1}\delta^{-\hat{a}_1-1/2} \\ 
    \beta_{2\hat{a}_1+1}\delta^{-\hat{a}_1} \end{pmatrix},
    \qquad
    \begin{cases}
        \alpha_{2\hat{a}_1+1} = \pm \beta_{2\hat{a}_1}/k_0 \cdot \sqrt{\lambda}\,t_{2\hat{a}_1}, \\
        \beta_{2\hat{a}_1+1} = \beta_{2\hat{a}_1}.
    \end{cases}
\]
Repeating this procedure gives, for $\hat{a}_1 + 1 \leq i \leq \hat{b}_1$,
\[
    \begin{pmatrix} u(x_{2i+1}) \\ u'_+(x_{2i+1}) \end{pmatrix}
    = (-1)^{\sum_{l=a_1}^{2i} m_l}
    \begin{pmatrix} \alpha_{2i+1}\delta^{-\hat{a}_1-1/2} \\ \beta_{2i+1}\delta^{-\hat{a}_1} \end{pmatrix},
\]
where $\alpha_{2i+1} = \alpha_{2i} \pm \beta_{2i}/k_0 \cdot \sqrt{\lambda}\,t_{2i}$ and 
$\beta_{2i+1} = \beta_{2i}$; and for $\hat{a}_1 \leq i \leq \hat{b}_1 - 1$,
\[
    \begin{pmatrix} u(x_{2i+2}) \\ u'_+(x_{2i+2}) \end{pmatrix}
    = (-1)^{\sum_{l=a_1}^{2i+1} m_l}
    \begin{pmatrix} \alpha_{2i+2}\delta^{-\hat{a}_1-1/2} \\ 
    \beta_{2i+2}\delta^{-\hat{a}_1-1} \end{pmatrix},
\]
where $\alpha_{2i+2} = \alpha_{2i+1}$ and 
$\beta_{2i+2} = \beta_{2i+1} \mp \alpha_{2i+1} k_0\sqrt{\lambda}\,t_{2i+1}$. In particular,
\[
    \begin{pmatrix} u(x_{b_1+1}) \\ u'_+(x_{b_1+1}) \end{pmatrix}
    = (-1)^{\sum_{l=a_1}^{b_1} m_l}
    \begin{pmatrix} \alpha_{2\hat{b}_1+1}\delta^{-\hat{a}_1-1/2} \\ 
    \beta_{2\hat{b}_1+1}\delta^{-\hat{a}_1} \end{pmatrix}.
\]
Set $\al_1 = (\alpha_{2\hat{a}_1+1}, a_{2\hat{a}_1+3}, \ldots, a_{2\hat{b}_1-1})^\top$ 
and $\be_1 = (\beta_{2\hat{a}_1}, \beta_{2\hat{a}_1+2}, \ldots, \beta_{2\hat{b}_1})^\top$, and 
apply Lemma \ref{thm: eig of C}(iv) with $\mu_1 = \pm\sqrt{\lambda}/k_0$ and 
$\mu_2 = \pm\sqrt{\lambda}k_0$. All equations in (1) and (2) are satisfied except 
the last row of (1). Since $\lambda$ is not an eigenvalue of $C_1$, 
Lemma~\ref{thm: eig of C}(iv) implies that this row fails, giving
\[
    \alpha_{2\hat{b}_1+1} = \alpha_{2\hat{b}_1-1} \mp \beta_{2\hat{b}_1}\sqrt{\lambda}/k_0 
    \cdot t_{2\hat{b}_1} \neq 0.
\]

\medskip
\noindent\textbf{Step 3: Propagation across remaining intervals left of $j^*$.}
For convenience, define
\[
    M_1 = \sum_{l=a_1}^{b_1} m_l, \quad A_1 = \hat{a}_1, \qquad
    M_i = M_{i-1} + \sum_{l=a_i}^{b_i} m_l, A_i = A_{i-1} + \left(\hat{a}_i - \hat{b}_{i-1} - \tfrac{1}{2}\right), 
    \quad 1 < i \leq j^*.
\]
Using the same procedures as in Steps~1 and~2, for $1 \leq j < j^*$ on the 
non-resonant intervals $[x_{b_j+1}, x_{a_{j+1}}]$,
\[
    \begin{pmatrix} u(x_{2\hat{b}_j+2i}) \\ u'_+(x_{2\hat{b}_j+2i}) \end{pmatrix}
    = (-1)^{M_j}
    \begin{pmatrix} \alpha_{2\hat{b}_j+2i}\delta^{-A_j-i+1/2} \\ 
    \beta_{2\hat{b}_j+2i}\delta^{-A_j-i-1/2} \end{pmatrix},
    \quad 1 \leq i \leq \hat{a}_{j+1} - \hat{b}_j,
\]
\[
    \begin{pmatrix} u(x_{2\hat{b}_j+2i+1}) \\ u'_+(x_{2\hat{b}_j+2i+1}) \end{pmatrix}
    = (-1)^{M_j}
    \begin{pmatrix} \alpha_{2\hat{b}_j+2i+1}\delta^{-A_j-i-1/2} \\ 
    \beta_{2\hat{b}_j+2i+1}\delta^{-A_j-i+1/2} \end{pmatrix},
    \quad 1 \leq i < \hat{a}_{j+1} - \hat{b}_j,
\]
with recurrences
\[
    \begin{cases}
    \alpha_{2\hat{b}_j+2i} = \alpha_{2\hat{b}_j+2i-1}\cos(k_0 t_{2\hat{b}_j+2i-1}), \\
    \beta_{2\hat{b}_j+2i} = -\alpha_{2\hat{b}_j+2i-1} k_0\sin(k_0 t_{2\hat{b}_j+2i-1}),
    \end{cases}
    \quad
    \begin{cases}
        \alpha_{2\hat{b}_j+2i+1} =\beta_{2\hat{b}_j+2i}/k_0 \cdot \sin(k_0 t_{2\hat{b}_j+2i}), \\
        \beta_{2\hat{b}_j+2i+1} = \beta_{2\hat{b}_j+2i}\cos(k_0 t_{2\hat{b}_j+2i}).
    \end{cases}
\]
In particular,
\[
    \begin{pmatrix} u(x_{a_{j+1}}) \\ u'_+(x_{a_{j+1}}) \end{pmatrix}
    = (-1)^{M_j}
    \begin{pmatrix} \alpha_{2\hat{a}_{j+1}}\delta^{-A_{j+1}} \\ 
    \beta_{2\hat{a}_{j+1}}\delta^{-A_{j+1}-1} \end{pmatrix},
\]
where $\beta_{2\hat{a}_{j+1}} = (-1)^{\hat{a}_{j+1}-\hat{b}_j} \alpha_{2\hat{b}_j+1} k_0 
\prod_{i=2\hat{b}_j+1}^{2\hat{a}_{j+1}-1} \sin(k_0 t_i) \neq 0$. 

On the resonant  intervals $[x_{a_j}, x_{b_j+1}]$ for $1 < j \leq j^*$, the same procedure as in Step~2 gives
\[
    \begin{pmatrix} u(x_{2i+1}) \\ u'_+(x_{2i+1}) \end{pmatrix}
    = (-1)^{M_{j-1}+\sum_{l=a_j}^{2i} m_l}
    \begin{pmatrix} \alpha_{2i+1}\delta^{-A_j-1/2} \\ \beta_{2i+1}\delta^{-A_j} \end{pmatrix},
    \quad \hat{a}_j \leq i \leq \hat{b}_j,
\]
\begin{align}\label{equ: u(x_2i+2)}
    \begin{pmatrix} u(x_{2i+2}) \\ u'_+(x_{2i+2}) \end{pmatrix}
    = (-1)^{M_{j-1}+\sum_{l=a_j}^{2i+1} m_l}
    \begin{pmatrix} \alpha_{2i+2}\delta^{-A_j-1/2} \\ \beta_{2i+2}\delta^{-A_j-1} \end{pmatrix},
    \quad \hat{a}_j \leq i \leq \hat{b}_j - 1,
\end{align}
with recurrences $\alpha_{2\hat{a}_j+1} = \pm \beta_{2\hat{a}_j}/k_0 \cdot \sqrt{\lambda}\,t_{2\hat{a}_j}$, 
$\beta_{2\hat{a}_j+1} = \beta_{2\hat{a}_j}$, and
\[
    \begin{cases}
        \alpha_{2i} = \alpha_{2i-1}, \\
        \beta_{2i} = \beta_{2i-1} \mp \alpha_{2i-1} k_0\sqrt{\lambda}\,t_{2i-1},
    \end{cases}
    \quad
    \begin{cases}
        \alpha_{2i+1} = \alpha_{2i} \pm \beta_{2i}/k_0 \cdot \sqrt{\lambda}\,t_{2i}, \\
        \beta_{2i+1} = \beta_{2i},
    \end{cases}
    \quad \hat{a}_j + 1 \leq i \leq \hat{b}_j.
\]
In particular, $\bigl(u(x_{b_j+1}), u'_+(x_{b_j+1})\bigr)^\top = 
(-1)^{M_j}\bigl(\alpha_{2\hat{b}_j+1}\delta^{-A_j-1/2},\, 
\beta_{2\hat{b}_j+1}\delta^{-A_j}\bigr)^\top$, and by the same reasoning as in Step~2, 
$\alpha_{2\hat{b}_j+1} \neq 0$ for $1 < j < j^*$.

\medskip
\noindent\textbf{Step 4: Target resonant interval $[x_{a_{j^*}}, x_{b_{j^*}+1}]$.}
Since $\lambda$ is a positive eigenvalue of $C_{j^*}$, Lemma~\ref{thm: eig of C}(iv) 
implies that $\bigl(\lambda,\, \be = (\beta_{2\hat{a}_{j^*}}, \beta_{2\hat{a}_{j^*}+2}, 
\ldots, \beta_{2\hat{b}_{j^*}})^\top\bigr)$ is an eigenpair of $D_{j^*}$ and the last row 
of (1) holds, giving
\[
    \alpha_{2\hat{b}_{j^*}+1} = \alpha_{2\hat{b}_{j^*}-1} \mp \beta_{2\hat{b}_{j^*}}\sqrt{\lambda}/k_0 
    \cdot t_{2\hat{b}_{j^*}} = 0.
\]
Hence there exists $\widetilde{\alpha}_{2\hat{b}_{j^*}+1}$ (possibly zero) such that
\begin{align}\label{equ: b_j_1+1 left}
    \begin{pmatrix} u(x_{b_{j^*}+1}) \\ u'_+(x_{b_{j^*}+1}) \end{pmatrix}
    = (-1)^{M_{j^*}}
    \begin{pmatrix}
        \widetilde{\alpha}_{2\hat{b}_{j^*}+1}\delta^{-A_{j^*}} + \mathcal{O}(\delta^{-A_{j^*}+1/2}) \\
        \beta_{2\hat{b}_{j^*}+1}\delta^{-A_{j^*}} + \mathcal{O}(\delta^{-A_{j^*}+1/2})
    \end{pmatrix}.
\end{align}
Since $\be$ is an eigenvector of a tridiagonal matrix with nonzero off-diagonal 
entries, all its entries are nonzero (as $\beta_{2\hat{a}_{j^*}} \neq 0$). In particular, 
$\beta_{2\hat{b}_{j^*}+1} = \beta_{2\hat{b}_{j^*}} \neq 0$.

\medskip
\noindent\textbf{Step 5: Leftward propagation from $x_{2N}$ to $x_{b_{j^*}+1}$.}
Since $\widetilde{\alpha}_{2\hat{b}_{j^*}+1}$ is undetermined from the rightward propagation, 
we propagate leftward from $x_{2N}$ and match at $x_{b_{j^*}+1}$. For $x > x_{2N}$, 
$u(x) = \e^{\mathrm{i}kx}$ up to a constant factor, so there exists an exponent $B$ 
such that
\begin{align}\label{equ: x_2N order}
    (u(x_{2N}), u'_+(x_{2N}))^\top = (\alpha_{2N}\delta^{B-1/2}, \beta_{2N}\delta^{B-1/2})^\top,
\end{align}
with $\alpha_{2N} = \e^{\mathrm{i}k_0 x_{2N}}$ and $\beta_{2N} = \mathrm{i}k_0 \e^{\mathrm{i}k_0 x_{2N}}$.
Proceeding leftward via~\eqref{equ: P(k,a) expand non} on non-resonant intervals 
$[x_{b_j+1}, x_{a_{j+1}}]$, $j^* \leq j \leq p$, and~\eqref{equ: P(k,a) expand res} 
on resonant intervals $[x_{a_j}, x_{b_j}]$, $j^* < j \leq p$ (using $P(k, -t_i)$ 
to account for the leftward direction), we record only the resulting orders. Define
\[
    B_{p+1} = B, \qquad
    B_j = B_{j+1} - \left(\hat{a}_{j+1} - \hat{b}_j - \tfrac{1}{2}\right),
    \quad j = p, p-1, \ldots, j^*.
\]
A straightforward computation gives
\[
    \begin{pmatrix} u(x_{2\hat{a}_j}) \\ u'_+(x_{2\hat{a}_j}) \end{pmatrix}
    = \begin{pmatrix} \mathcal{O}(\delta^{B_j-1/2}) \\ \mathcal{O}(\delta^{B_j-1}) \end{pmatrix},
    \quad j = p, p-1, \ldots, j^*+1,
\]
\[
    \begin{pmatrix} u(x_{2\hat{b}_j+1}) \\ u'_+(x_{2\hat{b}_j+1}) \end{pmatrix}
    = \begin{pmatrix} \mathcal{O}(\delta^{B_j}) \\ \mathcal{O}(\delta^{B_j}) \end{pmatrix},
    \quad j = p, p-1, \ldots, j^*.
\]
In particular,
\begin{align}\label{equ: b_j_1+1 right}
    \begin{pmatrix} u(x_{b_{j^*}+1}) \\ u'_+(x_{b_{j^*}+1}) \end{pmatrix}
    = \begin{pmatrix} \mathcal{O}(\delta^{B_{j^*}}) \\ \mathcal{O}(\delta^{B_{j^*}}) \end{pmatrix}.
\end{align}
Comparing~\eqref{equ: b_j_1+1 left} with~\eqref{equ: b_j_1+1 right} gives 
$-A_{j^*} = B_{j^*}$, and solving for $B$ yields
\[
    B = \sum_{j=j^*}^p \left(\hat{a}_{j+1} - \hat{b}_j - \tfrac{1}{2}\right)
      - \sum_{j=1}^{j^*-1} \left(\hat{a}_{j+1} - \hat{b}_j - \tfrac{1}{2}\right) - \hat{a}_1.
\]
The orders of $(u(x_i), u'_+(x_i))^\top$ for $b_{j^*}+1 \leq i \leq 2N$ then follow 
from the leftward propagation.

\medskip
\noindent\textbf{Step 6: Eigenmode approximation on each interval.}
Having determined the orders of $(u(x_i), u'_+(x_i))^\top$ for $1 \leq i \leq 2N$, 
we use~\eqref{equ: P(k,a)}: if $(u(x_i), u'_+(x_i))^\top = 
(\mathcal{O}(\delta^{\gamma_1}), \mathcal{O}(\delta^{\gamma_2}))^\top$, then for 
$x \in (x_i, x_{i+1})$,
\[
    \begin{pmatrix} u(x) \\ u'(x) \end{pmatrix}
    = \begin{pmatrix} 
        \mathcal{O}(\delta^{\min\{\gamma_1,\gamma_2\}}) \\ 
        \mathcal{O}(\delta^{\min\{\gamma_1,\gamma_2\}}) 
    \end{pmatrix}.
\]
Normalizing by choosing the constant factor so that $u(x) = \delta^{A_{j^*}} 
\e^{-\mathrm{i}kx}$ for $x < x_1$, we find that for $i \notin \{2\hat{a}_{j^*}, 
2\hat{a}_{j^*}+2, \ldots, 2\hat{b}_{j^*}\}$ there exists $\gamma_i \in \mathbb{Z}_+$ 
such that $u(x) = \mathcal{O}(\delta^{\gamma_i/2})$ on $(x_i, x_{i+1})$, while for 
$i \in \{2\hat{a}_{j^*}, 2\hat{a}_{j^*}+2, \ldots, 2\hat{b}_{j^*}\}$, $u(x) = \mathcal{O}(1)$. 
From~\eqref{equ: u(x_2i+2)}, for $i = \hat{a}_{j^*}, \ldots, \hat{b}_{j^*}$,
\[
    \begin{pmatrix} u(x_{2i}) \\ u'_+(x_{2i}) \end{pmatrix}
    = (-1)^{M_{j^*-1}+\sum_{l=a_{j^*}}^{2i-1} m_l}
    \begin{pmatrix} \mathcal{O}(\delta^{1/2}) \\ \beta_{2i} + \mathcal{O}(\delta^{1/2}) \end{pmatrix}.
\]
Applying~\eqref{equ: P(k,a)} once more, for $x \in (x_{2i}, x_{2i+1})$,
\begin{align*}
    u(x) &= (-1)^{M_{j^*-1}+\sum_{l=a_{j^*}}^{2i-1} m_l}
    \frac{\beta_{2i}}{k_0}\sin(k_0(x - x_{2i})) + \mathcal{O}(\delta^{1/2}) \\
    &= (-1)^{M_{j^*-1}} \frac{\beta_{2i}}{k_0}\sin(k_0(x - x_{a_{j^*}})) 
    + \mathcal{O}(\delta^{1/2}),
\end{align*}
where $\bigl(\lambda,\; (-1)^{M_{j^*-1}}\be/k_0\bigr)$ with 
$\be = (\beta_{2\hat{a}_{j^*}}, \beta_{2\hat{a}_{j^*}+2}, \ldots, \beta_{2\hat{b}_{j^*}})^\top$ 
is an eigenpair of $D_{j^*}$.
\end{proof}

\begin{remark}\label{rem: subwavelength vs beyond}
In the subwavelength regime (see~\eqref{equ: u(x) estimate}), 
$(\lambda, \al = (\alpha_i)_{i=1}^N)$ is an eigenpair of $C = \mathcal{C}(0)$ and $(\lambda, \be = ((\alpha_{i+1}-\alpha_i)/s_i)_{i=1}^{N-1})$ is an eigenpair of $D$, so that $u(x)$ is approximately constant inside the resonators and linear on the spacings. By contrast, beyond the subwavelength regime, $u(x)$ is approximated by trigonometric functions on specific spacings and vanishes elsewhere to leading order.

This contrast has two sources. First, in the subwavelength regime, $n(0) = 2N-1$ and $m = N-1$~\cite{feppon.cheng.ea2023Subwavelength}, because 
$k_0 = 0$ trivially satisfies $t_j k_0 \in \pi\mathbb{Z}$ for every $j$, so the entire chain forms a single resonant block. Beyond the subwavelength regime, $n(k_0) \leq 2N-1$ generically, and the index set $I$ splits into several disjoint 
blocks unless the geometry is specially tuned. This localization of resonance to specific blocks forces the eigenmode to be supported only on those blocks.

Second, the propagation matrix $P(k,a)$ behaves fundamentally differently in the two regimes. For subwavelength frequencies $k(\delta) = \pm\sqrt{\lambda\delta/r} 
+ \mathcal{O}(\delta)$, a direct expansion gives
\[
    P(k,a) = \begin{pmatrix} 1 & a \\ 0 & 1 \end{pmatrix}
    + \delta\lambda \begin{pmatrix} -\tfrac{1}{2}a^2 & -\tfrac{1}{6}a^3 \\ 
    -a & -\tfrac{1}{2}a^2 \end{pmatrix} + \mathcal{O}(\delta^{3/2}),
\]
which contrasts sharply with the expansions~\eqref{equ: P(k,a) expand non} 
and~\eqref{equ: P(k,a) expand res} valid for $k_0 \neq 0$. Together, these two factors account for the qualitatively different structure of the eigenmodes in the two regimes.
\end{remark}

\subsection{Multiple eigenvalue case}\label{sec: multiple case}

In this section, we briefly discuss the difficulties for extending Theorem~\ref{thm: eigenmode simple case} to the case where $\lambda$ is a degenerate eigenvalue of $\mc{C}(k_0)$. Let $k^\pm(\delta) =\omega^\pm(\delta)/v= k_0 \pm \sqrt{\lambda\delta} + d\delta + \mathcal{O}(\delta^{3/2})$ 
be an eigenfrequency with coefficients $\lambda, d$ given by 
Theorem~\ref{thm:mainresult1}. One can verify that $\lambda$ depends only on 
$\bm{t}|_{\mathcal{I}_j}$ for some block $\mathcal{I}_j = \llbracket a_j, b_j \rrbracket$, 
while $d$ depends only on $\bm{t}|_{\widetilde{\mathcal{I}}_j}$, where 
$\widetilde{\mathcal{I}}_j := \llbracket a_j - 1, b_j + 1 \rrbracket$. In general, neither $\lambda$ nor the pair $(\lambda, d)$ determines $j$ uniquely: 
multiple blocks may share the same eigenvalue $\lambda$, or even the same pair 
$(\lambda, d)$, as illustrated in the examples below. This non-uniqueness reflects 
a genuine degeneracy in the eigenmode structure, which we now describe in detail. 
\begin{enumerate}
    \setlength{\itemsep}{0pt}
    \item When $\lambda$ is an eigenvalue of $\mathfrak{m} \geq 2$ blocks $C_{j_1}, \ldots, C_{j_\mathfrak{m}}$, the eigenmode $u(x)$ is a superposition of contributions 
    from all blocks sharing $\lambda$, and the mixing coefficients cannot be determined from the leading-order analysis alone.
    \item The $2\mathfrak{m}$ branches 
    $k^{i,\pm}(\delta) = k_0 \pm \sqrt{\lambda\delta} + \mathcal{O}(\delta)$ are  indistinguishable at order $\mathcal{O}(\delta^{1/2})$. Whether and at which 
    order the degeneracy is lifted depends on $\bm{t}$: if the second-order 
    coefficients $d_i$ are distinct, the branches separate at $\mathcal{O}(\delta)$; 
    if some $d_i$ coincide, one must expand to $\mathcal{O}(\delta^{3/2})$ or 
    beyond. In the worst case, all coefficients agree to every finite order, and the degeneracy is never lifted by a formal asymptotic expansion.
    \item In the simple case, Theorem~\ref{thm: eigenmode simple case} identifies the eigenmode uniquely via 
    a single matching condition in the propagation matrix argument (Step~4 of the proof). In the multiple case, one must instead diagonalise an effective coupling matrix within the degenerate subspace, a procedure whose 
    outcome depends sensitively on the higher-order structure of $\bm{t}$ and does not reduce to a single closed-form expression.
\end{enumerate}

To illustrate, consider the two parameter choices
\begin{align}\label{equ: bm t setting1}
    \bm{t} = (0.3,\ 0.3,\ 1,\ 2,\ 1,\ 1.3,\ 1.3,\ 1,\ 1)^\top, \qquad k_0 = \pi,
\end{align}
or
\begin{align}\label{equ: bm t setting2}
    \bm{t} = (0.3,\ 0.3,\ 1,\ 2,\ 1,\ 1.3,\ 1,\ 1,\ 1.3)^\top, \qquad k_0 = \pi
\end{align}
Then,
\[
    \mathcal{I}_1 = \llbracket 3, 5 \rrbracket, \qquad \mathcal{I}_2 = \llbracket 8, 9 \rrbracket 
    \quad \text{or} \quad \mathcal{I}_2 = \llbracket 7, 8 \rrbracket,
\]
and both blocks yield $\lambda_1 = \lambda_2 = 1$, while the coefficients $d$ 
determined by $\widetilde{\mathcal{I}}_1$ and $\widetilde{\mathcal{I}}_2$ differ. 
Hence $\lambda$ alone does not identify $j$.

For an example where $(\lambda, d)$ also fails to identify $j$ uniquely, consider
\begin{align}\label{equ: bm t setting3}
    \setcounter{MaxMatrixCols}{13}
    \bm{t} = (0.3,\ 1.3,\ 2,\ 3,\ 1.7,\ 2,\ 2,\ 2,\ 1.7,\ 3,\ 2,\ 2.3,\ 0.3)^\top, 
    \qquad k_0 = \pi,
\end{align}
for which $\mathcal{I}_1 = \llbracket 3,4 \rrbracket$, $\mathcal{I}_2 = \llbracket 6,8 \rrbracket$, 
$\mathcal{I}_3 = \llbracket 10,11 \rrbracket$. The six eigenfrequency branches are
\begin{align*}
    k^{1,\pm}(\delta) &= \pi \pm \sqrt{\tfrac{1}{6}\,\delta} 
    + \tfrac{1}{12}\cot(0.3\pi)\,\delta + \mathcal{O}(\delta^{3/2}), \\
    k^{2,\pm}(\delta) &= \pi \pm \sqrt{\tfrac{1}{2}\delta} 
    - \tfrac{1}{4}\cot(0.3\pi)\,\delta + \mathcal{O}(\delta^{3/2}), \\
    k^{3,\pm}(\delta) &= \pi \pm \sqrt{\tfrac{1}{6}\,\delta} 
    + \tfrac{1}{12}\cot(0.3\pi)\,\delta + \mathcal{O}(\delta^{3/2}),
\end{align*}
where $k^{i,\pm}$ is determined by $\widetilde{\mathcal{I}}_i$, $i = 1, 2, 3$. 
Here $(\lambda, d)$ for $\widetilde{\mathcal{I}}_1$ and $\widetilde{\mathcal{I}}_3$ 
coincide, so $(\lambda, d)$ does not identify $j$ uniquely either.

We now describe the structure of the eigenmodes. Given $(\lambda, d)$, let $j$ be 
any index such that $(\lambda, d)$ comes from $\widetilde{\mathcal{I}}_j$. To leading 
order $\mathcal{O}(\delta^{1/2})$, the eigenmode $u(x)$ corresponding to 
$k^\pm(\delta)$ is a linear combination of trigonometric functions on the spacings within blocks $\mathcal{I}_j$ sharing the same $\lambda$, with amplitudes determined by the eigenvectors of the corresponding $D_j$ 
in~\eqref{equ: D_j def}, and vanishes elsewhere. Moreover, for two blocks $\mathcal{I}_{j_1}$ and $\mathcal{I}_{j_2}$ sharing the same $\lambda$ (while $d$ can be different): when the gap between them is large, the eigenmode exhibits trigonometric behavior on the spacings of only one block; when the gap is small, it may exhibit trigonometric behavior on the spacings of both blocks. In particular, when the gap is exactly 
$1$, the eigenmode necessarily exhibits trigonometric behavior on the spacings of both $\mathcal{I}_{j_1}$ and $\mathcal{I}_{j_2}$. A numerical illustration of the eigenmodes in different settings can be found in Figure \ref{fig: eigenmode}. 
\begin{figure}[!htb] 
\centering
\begin{subfigure}[t]{0.48\textwidth}
    \centering
    \includegraphics[width=\textwidth]{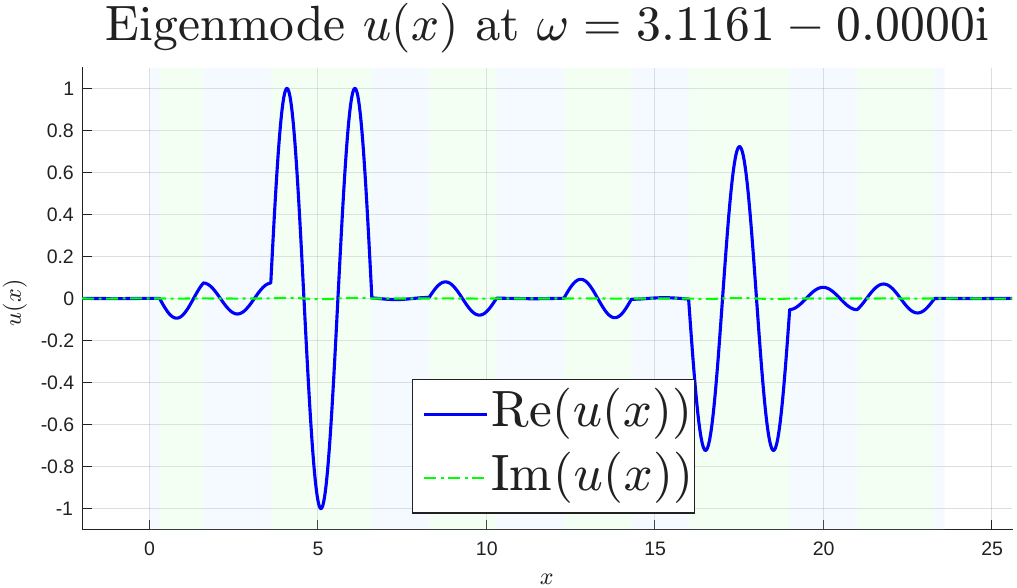}
  \caption{Eigenmode $u(x)$ corresponding to $k^{1,-}(\delta)$ in 
setting~\eqref{equ: bm t setting3}, approximated by trigonometric functions on 
the spacings of $\mathcal{I}_1$ and $\mathcal{I}_3$.}
\end{subfigure}
\hfill
\begin{subfigure}[t]{0.48\textwidth}
    \centering
    \includegraphics[width=\textwidth]{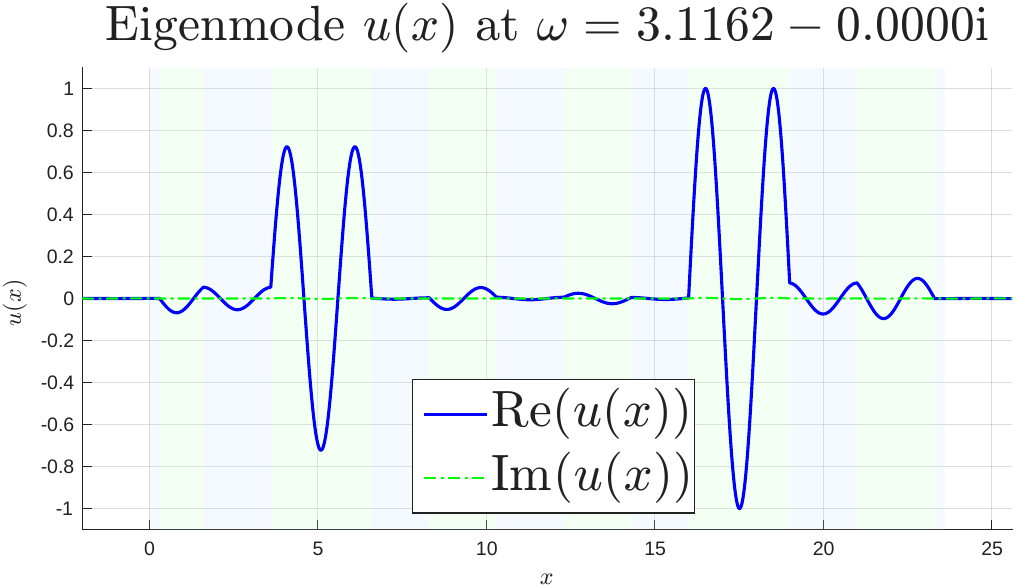}
   \caption{Eigenmode $u(x)$ corresponding to $k^{3,-}(\delta)$ in 
setting~\eqref{equ: bm t setting3}, approximated by trigonometric functions on 
the spacings of $\mathcal{I}_1$ and $\mathcal{I}_3$.}
\end{subfigure}
\vfill
\begin{subfigure}[t]{0.48\textwidth}
    \centering    \includegraphics[width=\textwidth]{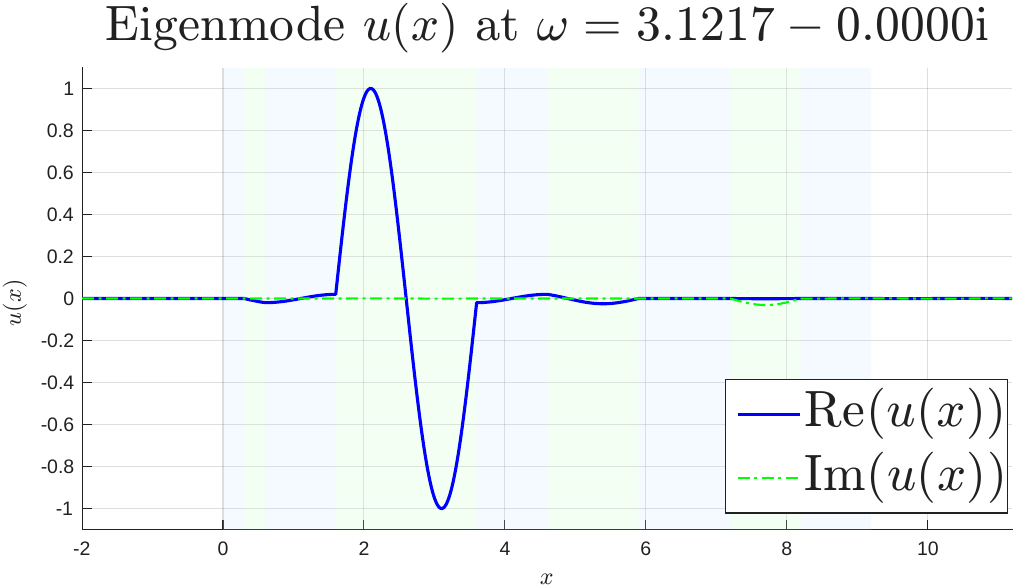}
    \caption{An eigenmode in setting (\ref{equ: bm t setting1}), approximated by trigonometric functions only on the spacings in $\mathcal{I}_1$ since the gap between $\mathcal{I}_1$ and $\mathcal{I}_2$ is large.}
\end{subfigure}
\hfill
\begin{subfigure}[t]{0.48\textwidth}
    \centering
\includegraphics[width=\textwidth]{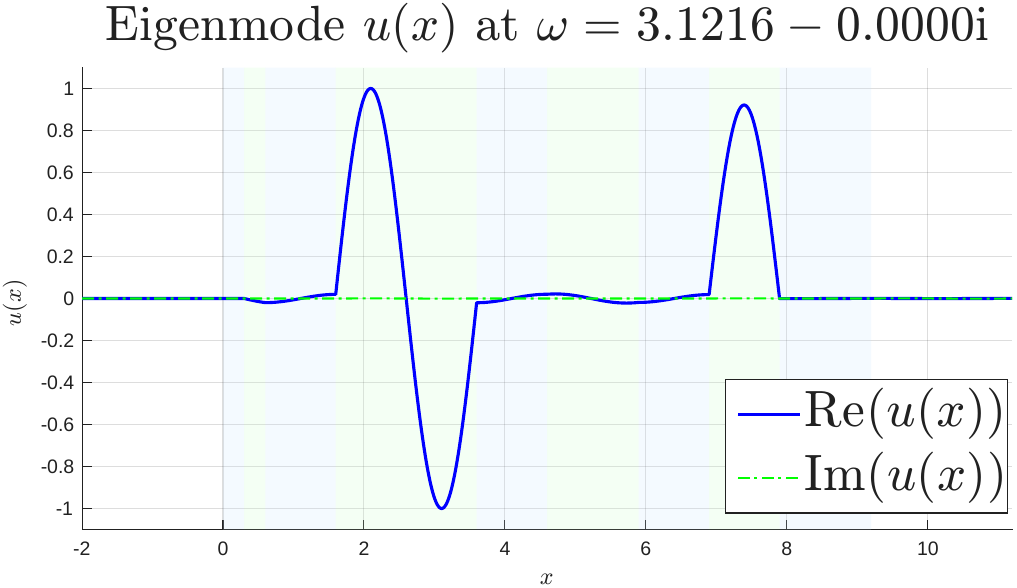} 
    \caption{An eigenmode in setting (\ref{equ: bm t setting2}), approximated by trigonometric functions both on the spacings in $\mathcal{I}_1$ and $\mathcal{I}_2$ since the gap between $\mathcal{I}_1$ and $\mathcal{I}_2$ is small (of length $1$).}
\end{subfigure}
\caption{Eigenmodes $u(x)$ in different settings with $\delta=0.004,r=1,v=1$. Blue and green shaded regions denote the structural lengths and spacings, respectively.}
    \label{fig: eigenmode}
\end{figure}



\section*{Acknowledgments}
This work was partially supported by the National Key R\&D Program of China grant number 2024YFA1016000 and the Fundamental Research Funds for the Central Universities grant number 226-2025-00192.

\section*{Data Availability Statement}
Data and codes supporting the findings of this work are available upon request.

\section*{Conflict of interest} 
The authors have no conflicts of interest to declare. 

\appendix 
\section{Proof of Theorem \ref{thm: G_{I,l} expand}} \label{appex:proofexpG}
\begin{proof}[Proof of Theorem \ref{thm: G_{I,l} expand}]
The proof of this expansion relies on (\ref{equ: G(k;delta) expand}) and the following identities:
\begin{equation}\label{equ:matrixexpan2}
\begin{aligned}
        1. \ &RMR=\eta(M)R;\\
        2. \ &\eta(L_j)=\begin{cases}
            -2\i\l[\sin(t_jk_0)+t_j\cos(t_jk_0)\cdot z\r]+ \mc{O}(z^2),&j\notin I,\\
            (-1)^{m_j}(-2\i t_jz)+ \mc{O}(z^3),&j\in I.
        \end{cases},\q \tau_0(\eta(L_j))=\begin{cases}
            0,&j\notin I,\\1,&j\in I;
        \end{cases}\\
        3. \
        &\eta(L_{j+1}SL_{j})=(-1)^{m_{j+1}+m_j}2+\mc{O}(z^2),\tau_0(\eta(L_{j+1}SL_j))=0,\q\{j,j+1\}\subset I;\\
        4. \
        &\tau_0(\eta(L_{j+2}SL_{j+1}SL_j))\geq 1,\q \{j,j+1,j+2\}\subset I;\\
        5. \
        &RL_1(k)S=(-1)^{m_1}R_++\mc{O}(z),\q\text{if } 1\in I,\\ &SL_{2N-1}(k)R=(-1)^{m_{2N-1}}R_-+\mc{O}(z),\q\text{if } 2N-1\in I,
    \end{aligned}
\end{equation}
where $\eta(M) = M_{21}+M_{22}-M_{11}-M_{12}$ for $M\in \mathbb{C}^{2\times 2}$, and $\tau_0(f)$ denotes the order of the first-order term in the asymptotic expansion of $f(z)$ as $z\to 0$, in particular, $\tau_0(0)=+\infty$.

For $1\leq l\leq m$, by (\ref{equ: G(k;delta) expand}), the expression of $G_l(k)$ consists of $\tbinom{2N}{l}$ terms, each corresponding to a selection of $l$ $S$-matrices (or equivalently, $2N-l$ $R$-matrices) from the $2N$ available matrices $R + \nu S$.  We denote by $Q$ the matrix product corresponding to a specific selection and analyze its order in $z$ as $z\to 0$. 

To this end, we first consider the $2N-1-l$ gaps formed by the $2N-l$ $R$-matrices and denote the product of all matrices within each gap by $M_1^{c_1}, M_2^{c_2},\ldots, M_{2N-l-1}^{c_{2N-l-1}}$, where $c_i$ denotes the number of $S$-matrices in the $i$\textsuperscript{th} gap; see \Cref{fig: M_s_basic} for an illustration. Moreover, we denote by $M_L^{c_L}$ the product of all matrices to the left of the first $R$-matrix, and by $M_R^{c_R}$ the product of all matrices to the right of the last $R$-matrix, where $c_L$ and $c_R$ denote the corresponding numbers of $S$-matrices involved. In particular, $M_L^0$ and $M_R^0$ are identity matrices.

\begin{figure}[!htb]
    \centering
    \begin{tikzpicture}[
        block/.style={rectangle, draw, inner sep=3pt, font=\footnotesize},
        label/.style={font=\footnotesize, below=2pt}
    ]
        \node[block] (ML1) at (0,0) {$SL_{2N-1}$};
        \node[label] at (ML1.south) {$M_L^1$};
        \node[font=\footnotesize, right=2pt of ML1] (R1) {$R$};
        \node[block, right=2pt of R1] (M10) {$L_{2N-2}$};
        \node[label] at (M10.south) {$M_1^0$};
        \node[font=\footnotesize, right=2pt of M10] (R2) {$R$};
        \node[block, right=2pt of R2] (M22) {$L_{2N-3}SL_{2N-4}SL_{2N-5}$}; 
        \node[label] at (M22.south) {$M_2^2$};
        \node[font=\footnotesize, right=2pt of M22] (dots) {$\cdots$};
        \node[font=\footnotesize, right=2pt of dots] (Rmid) {$R$};
        \node[block, right=2pt of Rmid] (Mend) {$L_4SL_3$};
        \node[label] at (Mend.south) {$M_{2N-1-l}^1$};
        \node[font=\footnotesize, right=2pt of Mend] (Rlast) {$R$};
        \node[block, right=2pt of Rlast] (MR2) {$L_2SL_1S$}; 
        \node[label] at (MR2.south) {$M_R^2$};
    \end{tikzpicture}
    \caption{Structure of $Q$ with $c_L=1,c_1=0,c_2=2,\cdots, c_{2N-l-1}=1,c_R=2$.}
    \label{fig: M_s_basic}
\end{figure}

For a certain gap $M_i^{c_i}$, let $\text{Sub}(M_i^{c_i})$ denote the set of all subscripts $j$ that the term $L_j$ appears in $M_i^{c_i}$. $\text{Sub}(M_L^{c_L})$ and $\text{Sub}(M_L^{c_L})$ can be defined similarly. As an illustration, for the configuration in \Cref{fig: M_s_basic}, we  have
\[
\text{Sub}(M_1^0)=\{2N-2\},\text{Sub}(M_2^2)=\{2N-3,2N-4,2N-5\},\text{Sub}(M_{2N-1-l}^1)=\{4,3\},
\]
\[
\text{Sub}(M_L^1)=\{2N-1\},\text{Sub}(M_R^2)=\{2,1\}.
\]
Then, by \eqref{equ:matrixexpan2}(1), the total matrix product $Q$ can be written as 
\begin{equation}\label{equ:msexpress1}
Q=\left[\prod_{i=1}^{2N-1-l}\eta(M_i^{c_i})\right] M_L^{c_L}RM_R^{c_R}.    
\end{equation}
Now, we define the index sets
\[
A=\{i:\text{Sub}(M_i^{c_i})\subset I\},\quad B=\{i:c_i=0\}.
\]
Applying \eqref{equ:matrixexpan2}(2) to (\ref{equ:msexpress1}) and using the definition of $\tau_0$, we observe that $\tau_0(M_i^{c_i})=1,i\in A\cap B.$ Note that since $\tau_0(M_L^{c_L}RM_{R}^{c_R})$ should always be $0$, we have
\begin{align}\label{equ:tau_0(Q)}
\tau_0(Q)=\sum_{i=1}^{2N-1-l}\tau_0(\eta(M_i^{c_i}))\geq\sum_{i=1}^{2N-1-l}\tau_0(\eta(M_i^{c_i}))\mathbf{1}_{A\cap B}=\# A\cap B.
\end{align}
Furthermore, we introduce the quantities
\[
\ti{c}_L:=\#\l(\text{Sub}(M_L^{c_L})\cap I\r)\leq\#\text{Sub}(M_L^{c_L})=c_L,\quad\ti{c}_R:=\#\l(\text{Sub}(M_R^{c_R})\cap I\r)\leq\#\text{Sub}(M_R^{c_R})=c_R.
\]
These count the matrices $L_j$ within the left and right end segments whose subscripts belong to $I$. With these definitions, we obtain the lower bound
\begin{equation}\label{equ: A cap B LB}
\begin{aligned}
\# A\cap B&\geq \# B-[(2N-1-c_L-c_R)-(n-\ti{c}_L-\ti{c}_R)]\\
&=\#B-(2N-1-n)+(c_L-\ti{c}_L)+(c_R-\ti{c}_R).    
\end{aligned}
\end{equation}
Since we select exactly $l$ $S$-matrices, we have $c_L + c_R + \sum_{i=1}^{2N-1-l}c_i=l$. This forces at least $2N-1-2l + c_L + c_R$ of $c_i$'s to be zero. Therefore,
\begin{align}\label{equ: B LB}
\# B\geq 2N-1-2l+c_L+c_R.
\end{align}
Combining this with (\ref{equ: A cap B LB}) yields
\[
\#A\cap B\geq n-2l+(2c_L-\ti{c}_L)+(2c_R-\ti{c}_R)\geq n-2l.
\]
From (\ref{equ:tau_0(Q)}) we conclude that $\tau_0(Q)\geq n-2l$. The equality holds if and only if $c_L=c_R=0$, $\#B=2N-1-2l$, exactly $2N-1-n$ gaps with $c_i=0$ satisfy $\text{Sub}(M_i^{c_i})\cap I=\emptyset$, the left $n-2l$ gaps with $c_i=0$ and the $l$ gaps with $c_i=1$ satisfy $\text{Sub}(M_i^{c_i})\subset I$. In this case (\emph{i.e.}, $\tau_0(Q) = n-2l$), the $l$ gaps $M_{i}^{c_i}$ with $c_i=1$ take the form $L_{j_i+1}SL_{j_i}$ with $ j_1\prec j_2\prec \cdots\prec j_l,\{j_i,j_i+1\}\subset I$. Substituting the asymptotic expansions from (2,3) of (\ref{equ:matrixexpan2}), the leading term of such a product $Q$ becomes
\begin{align*}
    Q=&\left[\prod_{i=1}^{2N-l-1}\eta(M_i^{c_i})\right] R=\prod_{j\neq j_i ,j_{i+1}}\eta(L_j)\prod_{i=1}^{l}\eta(L_{j_i +1}SL_{j_i})\cdot R\\
    =&\prod_{j\notin I}\eta(L_j)\prod_{I\ni j\ne j_i,j_i+1}\eta(L_j)\prod_{i=1}^l\eta(L_{j_i+1}SL_{j_i})\cdot R
    \\
    =&C_1\cdot2^l(-2\i)^{2N-1-2l}z^{n-2l}(1+C_2 z)\l(\prod_{i=1}^l\theta_{j_i}\r)\cdot R+\mc{O}(z^{n-2l+2})\notag.
\end{align*}
Thus, the expansion of order $\mc{O}(z^{n-2l})$ of $G_l$ can be given by
\[
G_l(k_0+z)=C_1\cdot 2^l(-2\i )^{2N-1-2l}z^{n-2l}(1+C_2z)\sum_{\substack{j_1\prec j_2\prec \cdots\prec j_l\\\{j_i,j_i+1\}\subset I}}\bigl(\prod_{i=1}^{l}\theta_{j_i}\bigr)\cdot R+\mc{O}(z^{n-2l+2}).
\]

Next, we examine the selections that yield $\tau_0(Q) = n-2l+1$. If $\# B\geq 2N+1-2l$, (\ref{equ: A cap B LB})(\ref{equ:tau_0(Q)}) gives $\#A\cap B\geq n-2l+2$, which implies $\tau_0(Q)\geq n-2l+2$. Thus, $\tau_0(Q)= n-2l+1$ cannot be attained in this case. By inequality (\ref{equ: B LB}), we have $\#B \geq 2N-1-2l$, so the only possibilities to achieve $\tau_0(Q)=n-2l+1$ are when $\#B = 2N-2l$ or $\#B = 2N-1-2l$.

\textbf{Case 1: $\# B=2N-2l$}. Exactly $2N-2l$ of the $c_i$'s equal to zero. Since $c_L+c_R+\sum_{i=1}^{2N-l-1}c_i\mathbf{1}_{c_i\neq 0}=l$, we distinguish two subcases:

\noindent $\bullet$ $c_L=c_R=0$, exactly one $c_i$ equals two, and all remaining nonzero $c_i$'s equal to one. Then (\ref{equ: A cap B LB}) yields $\# A\cap B\geq n-2l+1$, with equality holds only if the $2N-1-n$ gaps with $c_i=0$ satisfy $\text{Sub}(M_i^{c_i})\cap I=\emptyset$. This forces the unique gap with $c_i=2$ to satisfy $\text{Sub}(M_i^{c_i})\subset I$. By (\ref{equ:matrixexpan2})(4), we then have $\tau_0(\eta(M_i^{c_i}))\geq1$. Consequently,
\begin{align*}
    \tau_0(Q)=\sum_{i=1}^{2N-1-l}\tau_0(\eta(M_i^{c_i}))&\geq\sum_{i=1}^{2N-1-l}\tau_0(\eta(M_i^{c_i}))(\mathbf{1}_{A\cap B}+\mathbf{1}_{A\cap\{{i:c_i=2\}}})\\
    &\geq \# A\cap B+1\geq n-2l+2. 
\end{align*}
Hence, $\tau_0(Q)$ cannot be $n-2l+1$ in this subcase.

\noindent $\bullet$ $c_L=0, \ c_R=1$ or $c_L=1, \ c_R=0$, and all nonzero $c_i$ equal to one. We first expand $Q$ for the case $c_L=0, \ c_R=1$. In this case, (\ref{equ: A cap B LB}) yields $\# A\cap B\geq n-2l+1+(1-\ti{c}_R),$ and thus $\tau_0(Q)=n-2l+1$ occurs only when $\ti{c}_R=1$, \emph{i.e.}, $1\in I$. Assume now that $1\in I$, then  $\tau_0(Q)=n-2l+1$ if and only if exactly $2N-1-n$ gaps with $c_i=0$ satisfy $\text{Sub}(M_i^{c_i})\cap I=\emptyset$, the left $n-2l+1$ gaps with $c_i=0$ and the $l-1$ gaps with $c_i=1$ satisfy $\text{Sub}(M_i^{c_i})\subset I\setminus\{1\}$. The $l-1$ matrices $M_{i}^{c_i}$ with $c_i=1$ have the form $L_{j_i+1}SL_{j_i}$, where $j_1\prec j_2\prec \cdots\prec j_{l-1},\{j_i,j_i+1\}\subset I\setminus\{1\}$. Then, applying \eqref{equ:matrixexpan2}(2)(3)(5), we obtain 
\begin{align*}
    Q=C_1\cdot 2^{l-1}(-2\i )^{2N-2l}z^{n-2l+1}\frac{1}{t_1}\prod_{i=1}^{l-1}\theta_{j_i}\cdot R_{+}+\mc{O}(z^{n-2l+2}).
\end{align*}
Similarly, we can expand $Q$ for the case where $c_L=1, \ c_R=0$. In this circumstance, $\tau_0(Q)=n-2l+1$ occurs only when $2N-1\in I$. The $l-1$ gaps $M_{i}^{c_i}$ with $c_i=1$ have the form $L_{j_i+1}SL_{j_i}$, where $j_1\prec j_2\prec \cdots\prec j_{l-1},\{j_i,j_i+1\}\subset I\setminus\{2N-1\}$. A certain term $Q$ has the form
\begin{align*}
    Q=C_1\cdot 2^{l-1}(-2\i )^{2N-2l}z^{n-2l+1}\frac{1}{t_{2N-1}}\prod_{i=1}^{l-1}\theta_{j_i}\cdot R_{-}+\mc{O}(z^{n-2l+2}).
\end{align*}

\textbf{Case 2: $\# B=2N-1-2l$}. Here, the constraint $c_L+c_R+\sum_{i=1}^{2N-l-1}c_i\mathbf{1}_{c_i\neq 0}=l$ forces $c_L=c_R=0$. Consequently, there are $2N-1-2l$ gaps with $c_i=0$ and $l$ gaps with $c_i=1$. In this situation $\tau_0(Q)=\# A\cap B$. This implies that $\tau_0(Q)=n-2l+1$ occurs if and only if $n-2l+1$ gaps with $c_i=0$ satisfy $\text{Sub}(M_i^{c_i})\subset I$ while the left $2N-2-n$ gaps with $c_i=0$ satisfy $\text{Sub}(M_i^{c_i})\cap I=\emptyset$. Geometrically, these conditions mean that every position not belonging to $I$ has been occupied, with exactly one position left free. Because each $c_i=1$ occupies two positions, the one free position must be adjacent to $I$. Recall that $I$ is partitioned into $p$ disjoint integer intervals, the one free position can only be of the form $a_j-1\text{ or }b_j+1,1\leq j\leq p$. But if $a_1=1\text{ or }b_{p}=2N-1$ (equivalently $1\in I$ or $2N-1-\in I$), position $a_1-1\text{ or }b_p+1$ does not exist and is disregarded. Suppose that a free position has the form of $a_j-1$, a certain $Q$ should be
\begin{align*}
Q=&\prod_{j\notin I\cup\{a_j-1\}}\eta(L_j)\prod_{I\setminus\{a_j\}\ni j\ne j_i,j_i+1}\eta(L_j)\cdot\eta(L_{a_j}SL_{a_j-1})\cdot\prod_{i=1}^{l-1}\eta(L_{j_i+1}SL_j)\cdot R\\
=&C_1\cdot 2^l(-2\i)^{2N-1-2l}z^{n-2l+1}\frac{\cot(t_{a_j-1}k_0)}{t_{a_j}}\prod_{i=1}^{l-1}\theta_{j_i}\cdot R+\mc{O}(z^{n-2l+2}),
\end{align*}
where the indices $j_1,j_2,\cdots,j_{l-1}$ satisfy $j_1\prec j_2\prec\cdots\prec j_{l-1},\{j_i,j_i+1\}\subset I\setminus\{a_j\}$. If the free position is $b_j+1$, then a similar formula can be derived.

We now combine the two cases $\# B=2N-2l$ and $\#B=2N-1-2l$. When $1\in I$ or $2N-1\in I$, the terms discussed in case $\# B=2N-2l$ should be taken into account; otherwise, the terms in cases $\#B=2N-1-2l$ should be considered. This finally gives the expansion in (\ref{expGl}).

Regarding the expansion for $g_0(k_0+z)$, Theorem \ref{thm: resonant frequency f(z;mu) zeros} yields
\[
g_0(k_0+z)=g(k_0+z;0)=f(k_0+z;0)=-(2\i)^{2N-1}\prod_{j=1}^{2N-1}\sin [t_j(k_0+z)].
\]
A direct computation gives (\ref{equ: g_0 expand}).
\end{proof}

\bibliographystyle{plain}
\bibliography{highcon}

\end{document}